\documentclass[12pt]{amsart}
\usepackage{fullpage}
\usepackage{amsthm}
\usepackage{amssymb}
\usepackage{enumerate}
\usepackage[all]{xy}
\usepackage{hyperref}
\usepackage{verbatim}
\usepackage{mathrsfs}

\SelectTips{cm}{}			
\calclayout 

\numberwithin{equation}{section}  

\theoremstyle{plain}
\newtheorem{lemma}{Lemma}[section]
\newtheorem{theorem}[lemma]{Theorem}
\newtheorem*{theorem*}{Theorem}
\newtheorem{proposition}[lemma]{Proposition}
\newtheorem{corollary}[lemma]{Corollary}

\theoremstyle{definition}
\newtheorem{definition}[lemma]{Definition}
\newtheorem{example}[lemma]{Example}
\newtheorem{notation}[lemma]{Notation}
\newtheorem{construction}[lemma]{Construction}
\newtheorem{convention}[lemma]{Convention}
\newtheorem{warning}[lemma]{Warning}
\newtheorem{hyp}[lemma]{Hypothesis}

\theoremstyle{remark} 
\newtheorem{remark}[lemma]{Remark}

\newtheorem{claim}{Claim}[lemma]

\newcommand{\coind}{\operatorname{\mathrm{coind}}}

\newcommand{\colim}{\operatorname{colim}}

\renewcommand{\dim}{\operatorname{dim}}

\newcommand{\End}{\operatorname{End}}
\newcommand{\Ext}{\operatorname{Ext}}

\newcommand{\ev}{{\mathrm{ev}}}
\newcommand{\odd}{{\mathrm{odd}}}

\newcommand{\fl}{\operatorname{\mathsf{fl}}}

\newcommand{\GL}{\operatorname{GL}}
\newcommand{\Gr}{\mathsf{Gr}}

\newcommand{\Hom}{\operatorname{Hom}}

\renewcommand{\Im}{\operatorname{Im}}
\newcommand{\ind}{\operatorname{\mathrm{ind}}}

\newcommand{\Ker}{\operatorname{Ker}}

\newcommand{\Lie}{\operatorname{\mathsf{Lie}}}

\renewcommand{\mod}{\operatorname{\mathsf{mod}}}
\newcommand{\Mod}{\operatorname{\mathsf{Mod}}}

\newcommand{\res}{\operatorname{res}}

\newcommand{\sGr}{\mathsf{sGr}}

\newcommand{\stmod}{\operatorname{\mathsf{stmod}}}

\newcommand{\half}{{\textstyle\frac{1}{2}}}
\newcommand{\gr}{\operatorname{gr}}
\newcommand{\Dist}{\operatorname{Dist}}

\newcommand{\Fl}{\mathrm{Fold}} 

\newcommand{\bul}{\bullet}

\newcommand{\da}{{\downarrow}}

\newcommand{\ov}{\overline}

\def\mcA{\mathcal{A}}

\def\mcG{\mathcal{G}} 
\def\mcH{\mathcal{H}}

\def\mcL{\mathcal{L}}

\def\mcU{\mathcal{U}}

\def\bbF{\mathbb F} 
\def\bbG{\mathbb G}
\def\bbM{\mathbb M}

\def\bbZ{\mathbb Z}

\newcommand{\fA}{\mathfrak{A}}

\newcommand{\fg}{\mathfrak{g}}

\newcommand{\si}{\sigma}
\newcommand{\bsi}{{\boldsymbol\sigma}}

\newcommand{\cP}{\mathscr{P}}

\date{\today}

\title[Finite supergroup schemes]
{Detecting nilpotence and projectivity over finite unipotent supergroup schemes}

\author[BIK$\Pi$]{Dave Benson, Srikanth B. Iyengar,
  Henning Krause\\ and Julia Pevtsova}

\address{Dave Benson \\ 
Institute of Mathematics\\ 
University of Aberdeen\\ 
King's College\\ 
Aberdeen AB24 3UE\\ 
Scotland U.K.}

\address{Srikanth B. Iyengar\\ 
Department of Mathematics\\
University of Utah\\ 
Salt Lake City, UT 84112\\ 
U.S.A.}

\address{Henning Krause\\ 
Fakult\"at f\"ur Mathematik\\ 
Universit\"at Bielefeld\\ 
33501 Bielefeld\\ 
Germany.}

\address{Julia Pevtsova\\
Department of Mathematics\\
University of Washington\\
Seattle, WA 98195\\
U.S.A.}

\begin{document}

\begin{abstract}
This work concerns the representation theory and cohomology of a finite unipotent supergroup scheme $G$ over a perfect field $k$ of positive characteristic $p\ge 3$. It is proved that an element $x$ in the cohomology of $G$ is nilpotent if and only if for every extension field $K$ of $k$ and every elementary
sub-supergroup scheme $E\subseteq G_K$, the restriction of $x_K$ to $E$ is nilpotent.  It is also shown that a $kG$-module $M$ is projective if and only if for every extension field $K$ of $k$ and every elementary sub-supergroup scheme $E\subseteq G_K$, the restriction of $M_K$ to $E$ is projective. The statements are motivated by, and are analogues of, similar results for finite groups and finite group schemes, but the structure of elementary supergroups schemes necessary for detection is more complicated than in either of these cases. 
One application is a detection theorem for the nilpotence of cohomology, and projectivity of modules, over finite dimensional Hopf subalgebras of the Steenrod algebra.
\end{abstract}

\thanks{The work was supported by the NSF grant DMS-1440140 while DB, SBI and JP were in residence at the MSRI.  SBI was partly supported by NSF grant DMS-1700985 and JP was partly supported by NSF grants DMS-0953011 and DMS-1501146 and Brian and Tiffinie Pang faculty fellowship.}

\maketitle
\setcounter{tocdepth}{1}
\tableofcontents

\section{Introduction}
There has been considerable research, some of recent vintage, aimed at understanding  representations of finite group schemes through the lens of their support varieties; see \cite{Bendel:2001a,Bendel:2001b, Bendel/Friedlander/Suslin:1997a, Bendel/Friedlander/Suslin:1997b, Benson/Iyengar/Krause/Pevtsova:2017b, Benson/Iyengar/Krause/Pevtsova:bikp4, Benson/Iyengar/Krause/Pevtsova:2018a, Friedlander/Pevtsova:2005a, Friedlander/Pevtsova:2007a}. The paradigm for these developments is the work on the modular representation theory of finite groups due to Alperin and Evens~\cite{Alperin/Evens:1982a},  Avrunin and Scott~\cite{Avrunin/Scott:1982a}, Chouinard~\cite{Chouinard:1976a}, Carlson~\cite{Carlson:1983a}, Dade~\cite{Dade:1978b}, Quillen~\cite{Quillen:1971b+c}, among others.   This paper is part of a project aimed at finding analogues of some of these results and techniques for finite supergroup schemes. The first step in this direction was taken by Drupieski \cite{Drupieski:2013a,Drupieski:2016a}, who proved finite generation of cohomology for finite supergroup schemes, generalizing the theorem of Friedlander and Suslin for finite group schemes \cite{Friedlander/Suslin:1997a}. Drupieski and Kujawa \cite{Drupieski/Kujawa:sv,Drupieski/Kujawa:ga,Drupieski/Kujawa:cs}  have initiated a study of support varieties for restricted Lie superalgebras.

A starting point for any theory of support varieties is the identification of a family of subgroups that detect nilpotence of cohomology classes and projectivity of representations. Once again, finite groups provide a model: Quillen~\cite{Quillen:1971b+c} proved that a class in mod $p$ cohomology  of a finite group $G$ is nilpotent if (and only if) its restriction to any elementary abelian $p$-subgroup $E < G$ is nilpotent in $H^*(E, \overline{\bbF}_p)$; see also  Quillen and Venkov~\cite{Quillen/Venkov:1972a}. This detection theorem is a key ingredient in the proof of Quillen's  stratification theorem that gives a complete description of the Zariski  spectrum of $H^*(G, \overline{\mathbb F}_p)$. Around the same time, Chouinard~\cite{Chouinard:1976a} proved that a representation $M$ of $G$ is projective if (and only if) the restriction of $M$ to any elementary abelian $p$-subgroups $E < G$ is projective.

In this work we establish analogues of the results of Quillen and Chouinard for  finite supergroup schemes. Throughout  we fix a perfect field $k$ of positive characteristic $p\ge 3$. A finite supergroup scheme over $k$ may be viewed either as a  functor on the category of $\bbZ/2$-graded commutative $k$-algebras with values in finite groups, or a finite dimensional $\bbZ/2$-graded cocommutative Hopf algebra; see Section~\ref{se:prelim} for details. The focus will be on unipotent supergroup schemes, though some of the preliminary results apply more generally. Each finite supergroup scheme has an \emph{even} part which is a finite  group scheme. In turn any finite group or group scheme furnishes an example of a supergroup scheme, but there are many more. Notably, the odd version of the additive group $\bbG_a$, denoted $\bbG_a^{-}$ and defined as a functor by $\bbG_a^{-}(R)=R_{1}^{+}$, the additive group on the odd part of $R$. The corresponding Hopf algebra  is $k[\si]/(\si^2)$, where $\si$ is in odd degree and a primitive element.

The notion of an ``elementary" supergroup scheme is a lot more
involved than in the case of finite groups. To begin with, we
construct a two-parameter family of finite supergroup schemes related
to the Witt vectors, denoted $E^-_{m,n}$, with $m \geq 2, n \geq 1$;
see Construction~\ref{not:Emn^-}. For example, $E^{-}_{m,1}$ can be
realised as an extension of $\bbG_{a}^{-}$ by $W_{m,1}$, the first
Frobenius kernel of Witt vectors of length $m$, recalled in
Appendix~\ref{sec:Witt}. Also
$E^-_{1,n}\cong \bbG_{a(n)}\times \bbG_a^-$, where $\bbG_{a(n)}$ is
the $n$th Frobenius kernel of $\bbG_a$.

\begin{definition}
\label{def:elementary}
A finite supergroup scheme $E$ over $k$ is   \emph{elementary} if it is isomorphic  to a quotient of some $E^-_{m,n}\times (\bbZ/p)^{\times s}$. \end{definition}

A special role is played by the quotients of $E^-_{m,n}$  by an even
subgroup scheme; these are the \emph{Witt elementary} supergroup
schemes, and described completely in
Theorem~\ref{th:Witt-classification}. Besides the $E^-_{m,n}$
themselves, one has also finite supergroup schemes that we denote
$E^-_{m,n,\mu}$, involving an element $\mu$ in $k^\times
/(k^\times)^{p^{m+n}-1}$.  The Hopf algebra corresponding to
$E^-_{m,n,\mu}$ is described in \eqref{not:Emnmu^-}. Any elementary supergroup scheme is of the form $E\cong E^0\times (\bbZ/p)^{\times s}$ where $E^0$ is isomorphic to either $\bbG_{a(r)}$ or a Witt elementary supergroup scheme.

The group algebra $kE$ of an elementary finite supergroup scheme $E$ is isomorphic to a tensor product of algebras of the form 
\begin{enumerate}[{\quad\rm(i)}]
\item
$k[s]/(s^{p})$
\item
$k[\si]/(\si^2)$, and 
\item
$k[s,\si]/(s^{p^n},\si^2-s^p)$, where $n\ge 1$,
\end{enumerate}
with $|s|$ even and $|\si|$ odd, and no more than one factor of types (ii) and (iii) combined is present. In particular, there is at most one generator of odd degree, and as an ungraded algebra $kE$ is a commutative complete intersection, even though case (iii) is not graded commutative.

Our main detection theorem is proved in Section~\ref{se:role}.

\begin{theorem}
\label{th:detect}
Let $G$ be a finite unipotent supergroup scheme over a field $k$ of positive characteristic $p\ge 3$. Then the following hold.
\begin{enumerate}[\quad\rm(i)]
\item
An element $x\in H^{*,*}(G,k)$ is nilpotent if and only if for every extension field $K$ of $k$ and every elementary  sub-supergroup scheme $E$ of $G_K$, the restriction of $x_K\in H^{*,*}(G_K,K)$ to $H^{*,*}(E,K)$ is nilpotent.
\item
A $kG$-module $M$ is projective if and only if for every extension field $K$ of $k$ and every elementary sub-supergroup scheme $E$ of $G_K$, the
restriction of $M_K$ to $E$ is projective.
\end{enumerate}
\end{theorem}

We also prove two versions of (i) for arbitrary coefficients. Theorem~\ref{th:main}(i)  proves the detection of nilpotents  for $H^{*,*}(G,M)$ for any $G$-module $M$ where nilpotents are  understood in the sense of Definition~\ref{de:nilp}. Theorem~\ref{th:nilp}, which generalises a  theorem of Bendel~ \cite{Bendel:2001a} for unipotent group schemes, gives detection of  nilpotents for $H^{*,*}(G, \Lambda)$ with coefficients in a unital $G$-algebra $\Lambda$.  

We also formulate and prove $\bbZ$-graded versions of our theorems, and apply them to finite dimensional subalgebras of the Steenrod algebra over $\bbF_{p}$. The structure of the Steenrod algebra  is well understood and the detection theorem in that case takes on a particularly simple form; see Theorem~\ref{th:Steenrod}.

\subsection*{Looking ahead}
Our results only cover  unipotent supergroup schemes, and it would be interesting to understand what more needs to be done in order to cover the general case. Unlike the case of finite group schemes, for a general finite supergroup scheme it is not true that cohomology modulo nilpotents and projectivity of modules are  detected on unipotent sub-supergroup schemes. Conversations with Chris Drupieski lead us to suspect that there is a mild generalisation of the Witt elementaries that are not unipotent, but which leads to a 
suitable detection family in this context. 

In a different direction, the detection theorems are only the first steps towards developing a theory of support varieties. Again we turn to groups to show us the way: While Chouinard's work highlights the role of elementary abelian groups, Dade~\cite{Dade:1978b} proved that to detect projectivity of a  representation of an elementary abelian $p$-group $E$, one can restrict further to all \emph{cyclic shifted subgroups} of the group algebra $kE$, which then becomes purely a problem in linear algebra. This detection theorem, now   known as ``Dade's lemma", is the foundation for the theory of rank varieties for modules for finite groups pioneered by  Carlson~\cite{Carlson:1983a}, and further developed by Benson, Carlson, and Rickard~\cite{Benson/Carlson/Rickard:1996a}. Their work was absorbed and generalised to the theory of $\pi$-points for finite groups schemes by Friedlander and Pevtsova~\cite{Friedlander/Pevtsova:2005a, Friedlander/Pevtsova:2007a}.  

Theorem~\ref{th:detect} opens up the road to a theory of $\pi$-points for finite unipotent supergroup schemes. We take this up in a follow up paper~\cite{Benson/Iyengar/Krause/Pevtsova:2019a}, where it is used to establish a stratification theorem  for the stable module category, akin to the one in \cite{Benson/Iyengar/Krause/Pevtsova:2018a}.

\subsection*{Structure of the paper}
The strategy of the proof of Theorem~\ref{th:detect} is quite intricate and we found it expedient to divide the paper into two parts. Before delving into a summary of the parts, it would perhaps help to present a roadmap of the proof; it follows the one for finite unipotent group schemes given in \cite{Bendel:2001a}, but a  number of extra complications arise. 

The simplest scenario is that there is a surjective map from $G$ to either $\bbG_a^-\times \bbG_a^-$ or $\bbG_{a(1)}\times \bbG_{a(1)}$, for then the argument in \cite[Theorem 8.1]{Bendel:2001a} applies. Otherwise one reduces to the case where there is a surjective map $f\colon G\to \bbG_{a(r)} \times (\bbG_a^-)^\epsilon \times (\bbZ/p)^s$  with $r,s\ge 0$ and $\epsilon=0$ or $1$, such that  $H^{1,*}(f)$ is an isomorphism. It is easy to tackle the case when   $f$ is itself an isomorphism. When it is not an isomorphism, a standard argument yields that $H^{2,*}(f)$ has a kernel. The situation when this kernel contains an element of odd degree, that is to say, when $H^{2,1}(f)$ is not one-to-one, is dealt with in \cite{Benson/Pevtsova:bp2}.
The difficulty arises when the kernel of $H^{2,*}(f)$ is concentrated in even degrees. Even here there are two cases, as elaborated on further below. The first one allows us to drop to proper subgroups and is easy to handle. The second one leads to elementary supergroup schemes. This is where the major deviation from \cite{Bendel:2001a} occurs, and  requires the bulk of the work. It occupies Part II of this paper. 

Here is a more detailed description of the paper:  Part I, comprising
Sections~\ref{se:prelim} to \ref{se:inductive}, provides background material on finite supergroup schemes and extensions of a number of techniques  used in other contexts.  Section~\ref{se:prelim} starts things off with main definitions,  examples, and basic properties of supergroup schemes. Section~\ref{se:low} records some key facts on low degree cohomology modules. Section~\ref{se:steenrod} describes the action of Steenrod operations on the cohomology of finite supergroup schemes. The central calculation there is Theorem~\ref{th:B3.6} that establishes that a homogeneous  ideal in $H^{*,*}(\bbG_{a(r)} \times \bbG_a^- \times (\bbZ/p)^{\times s},k)$ stable under  the Steenrod operations and containing an element from $H^{2,0}$ must have an element of a  specific form. The proof follows closely the proofs of the analogous result for  $(\bbZ/p)^{\times s}$, due to Serre~\cite{Serre:1965a},  for  $\bbG_{a(r)}$, due to Bendel, Friedlander, and Suslin~\cite{Bendel/Friedlander/Suslin:1997b}, and for $\bbG_{a(r)} \times (\bbZ/p)^{\times s}$ due to Bendel~\cite{Bendel:2001a}, but the conclusion is different. Whereas for finite group schemes, such an ideal always has an element that is  a product of Bocksteins of elements in degree $1$, in the super case we get either a product of Bocksteins or a mysterious element $\zeta^2 - \gamma x_r$ with $|\zeta| =(1,1), |x_r| =(2,0), \gamma \in k$.  This element is responsible for the work we have to do in Part II. 

Part I culminates in Theorem~\ref{th:detection} that asserts that if a finite unipotent supergroup satisfies certain conditions, laid out in Hypothesis~\ref{hyp:inductive},  nilpotence (of cohomology elements) and projecitivty (of modules) are detected on proper sub-supergroup schemes after field extensions.  For finite group schemes (not super ones) the calculation with the Steenrod operations  in Section~\ref{se:steenrod} would  then yield that any unipotent group scheme that is not  isomorphic to $\bbG_{a(r)} \times (\bbZ/p)^{\times s}$ satisfies Hypothesis~\ref{hyp:inductive}.  And this is precisely the argument in Bendel~\cite{Bendel:2001a}. Thus,  up to the end of Part I we are mostly mimicking the techniques existing in the literature.  Life in the super world turns out to be more complicated, all because of the cohomology class $\zeta^2 - \gamma x_r$  that cannot be eliminated with the help of the Steenrod operations. The task of the second part  of the paper is to show that if a finite unipotent supergroup scheme does not satisfy 
Hypothesis~\ref{hyp:inductive}, then, in fact, it must be elementary. 

Part II begins in Section~\ref{se:Elem} with the construction of the
elementary supergroup schemes featuring in the statement of
Theorem~\ref{th:detect}. Their cohomology rings are calculated in
Section~\ref{se:calculations}. These calculations feed into the proof
of Theorem~\ref{th:elemcoh} that is a cohomological criterion for
recognising elementary supergroup schemes.  Theorem~\ref{th:detect} is
proved as Theorem~\ref{th:main}. Its consequences for the Steenrod
algebra are described in
Section~\ref{se:steenrod-algebra}. Appendix~\ref{sec:Witt} provides
background on Dieudonn\'e modules needed to describe elementary
supergroup schemes.

\subsection*{Acknowledgements} 
We gratefully acknowledge the support and hospitality of the Mathematical Sciences Research Institute in Berkeley, California where we were in residence during the semester on ``Group Representation Theory and Applications" in the Spring of 2018.  The American Institute of Mathematics in San Jose, California gave us a fantastic opportunity to carry out part of this project during intensive research periods supported by their ``Research in Squares"  program; our thanks to them for that. Dave Benson thanks Pacific Institute for Mathematical Science for its support during his research visit to the University of  Washington in the Summer of 2016 as a distinguished visitor of the Collaborative Research Group in Geometric and Cohomological Methods  in Algebra. Dave Benson and Julia Pevtsova have enjoyed the hospitality of City University while working on this project  in the summers of 2017 and 2018. We are grateful to Chris Drupieski and Jon Kujawa for useful and informative conversations and their interest in our work.

\part{Recollections} 
\section{Finite supergroup schemes}
\label{se:prelim}

We give a  compressed introduction to the terminology we shall employ  in the paper referring the reader to a number of excellent sources on super vector spaces,  super algebras and super groups schemes, such as, for example, a survey paper by A. Masuoka  \cite{Masuoka:2017a} or \cite{Drupieski/Kujawa:sv}. 

Throughout this manuscript $k$ will be a field of positive characteristic $p\ge 3$. We assume $k$ is perfect since  some of the structural results for supergroup schemes require that condition. It is clear that the  main theorem holds for an arbitrary field $k$ of characteristic $p$ once it is proved for a perfect field of the same characteristic.

An \emph{affine supergroup scheme} over $k$ is a covariant  functor from $\bbZ/2$-graded commutative $k$-algebras  (in the sense that $yx=(-1)^{|x||y|}xy$) to groups, whose underlying functor to sets is representable. If $G$ is a supergroup scheme then its \emph{coordinate ring} $k[G]$ is the representing object. By applying Yoneda's lemma to the group multiplication and inverse maps, it is a $\bbZ/2$-graded commutative Hopf algebra. We denote the comultiplication on $k[G]$  by $\Delta\colon k[G] \to k[G] \otimes k[G]$ and the counit map by $\varepsilon\colon k[G] \to k$  with $I = \ker \varepsilon$ being the augmentation ideal and note that these are  \emph {degree-preserving} (equivalently, even) algebra homomorphisms.  The correspondence between affine supergroup schemes and their coordinate algebras gives a contravariant equivalence of categories between affine supergroup schemes and $\bbZ/2$-graded commutative Hopf algebras. 

A \emph{finite} supergroup scheme $G$ is an affine supergroup scheme whose coordinate ring is finite dimensional. In this case, the dual $kG=\Hom_k(k[G],k)$ is a finite dimensional $\bbZ/2$-graded cocommutative Hopf algebra called the \emph{group ring} of $G$. This gives a covariant equivalence of categories between finite supergroup schemes and finite dimensional $\bbZ/2$-graded (equivalently, ``super") cocommutative Hopf algebras. 

\begin{equation} 
\label{anti-equi}
\left\{\begin{array}{c}  \text{finite super-}\\ \text{group schemes} 
\end{array}\right\} \quad \sim \quad  \left\{\begin{array}{c}  \text{finite dimensional {super-}}\\ \text{{cocommutative} Hopf algebras} 
\end{array}\right\}
\end{equation} 

We employ the notation $V = V_0 \oplus V_1$ for $\bbZ/2$-graded (equivalently, ``super") 
vector spaces, where $V_0$ are the {\it even} degree elements, and $V_1$ are the {\it odd} 
degree elements. {\it A $kG$-module} is a $\bbZ/2$-graded $k$-vector space on which $kG$ acts 
respecting the grading in the usual way. As in the ungraded setting, a $kG$-module has an 
equivalent description as a rational representation of the supergroup $G$ on the category of 
super vector spaces. We consider all modules including infinite 
dimensional ones.  The trivial module $k$ is the trivial one dimensional representation 
concentrated in the even degree. 

If $K$ is a field extension of $k$, and $G$ is an affine supergroup scheme, 
we write $K[G]$ for $K \otimes_k k[G]$, which is a
graded commutative Hopf algebra over $K$. This defines a supergroup scheme over
$K$ denoted $G_K$, and we have a natural isomorphism of Hopf superalgebras
$KG_K\cong K\otimes_k kG$.

For each $kG$-module $M$, we set
\[ 
M_K:=K \otimes_k M\quad\text{and}\quad M^K:=\Hom_k(K,M),
\] 
viewed as $KG_K$-modules.

The \emph{even part} $G_\ev$ of an affine supergroup scheme $G$ is the largest
sub-supergroup scheme whose coordinate ring contains no
odd degree elements (see \cite{Masuoka:2017a}). It may be regarded as an affine group scheme. 
Its coordinate ring $k[G_\ev]$ is the quotient of $k[G]$
by the ideal generated by the odd degree elements. This 
ideal is automatically a Hopf ideal, since the coproduct $\Delta$ applied to an 
odd degree element is necessarily a linear 
combination of tensors $ a \otimes b$ where either $a$ or $b$ is odd. 
An \emph{even subgroup scheme} of $G$ is a subgroup scheme of $G_\ev$.

\begin{example}
Any affine group scheme $G$ may be thought of as an affine supergroup scheme
with $G=G_\ev$.
\end{example}

Another way to look at the assignment $G \mapsto G_\ev$ is that it gives the right adjoint to the inclusion functor 
from the affine group schemes to affine supergroups schemes. 

\begin{definition}
If $G$ is an affine supergroup scheme, let $G^{(1)}$ be the base change of $G$ 
via the Frobenius map $x\mapsto x^p$ on $k$.
Then the \emph{Frobenius map} $F\colon G \to G^{(1)}$ corresponds to
the map of coordinate rings $k[G^{(1)}]\to k[G]$ given by $x \to x^p$.
The $r$th \emph{Frobenius kernel} $G_{(r)}$ of $G$ is defined to be the kernel 
of the iterate $F^r \colon G \to G^{(r)}$. 
\end{definition}

\begin{convention} 
\label{co:zero} 
By $G_{(0)}$ we always mean the trivial group scheme. 
\end{convention}

\begin{definition}
A finite supergroup scheme $G$ over $k$ is said to be \emph{unipotent} if $k$ is the unique irreducible $kG$-module, which may be either in even or odd degree. A supergroup scheme $G$ is \emph{connected} if $k[G]$ is local.

If $G$ is a finite connected supergroup scheme then for some $r\ge 0$ we have $G=G_{(r)}$. The least such value of $r$ is called the \emph{height} of $G$. 
Note that $G$ has height zero if and only if $G$ is the trivial supergroup scheme.
\end{definition}

\begin{lemma}
\label{le:pi0}
Any finite supergroup scheme $G$ is a semidirect product $G^0 \rtimes \pi_0(G)$ with $G^0$ connected and $\pi_0(G)$ the finite group of connected components.
\end{lemma}

\begin{proof}
See Lemma~5.3.1 of Drupieski \cite{Drupieski:2013a}. The proof uses the fact that $k$ is perfect and has odd prime characteristic.
\end{proof}

\begin{theorem}
\label{th:tensor-dec}
Let $G$ be a connected finite supergroup scheme.  Then there exist odd degree elements $y_1, \ldots, y_n \in k[G]$ such that we have an isomorphism of $\mathbb Z/2$-graded $k$-algebras
\[ 
k[G] \cong k[G_\ev] \otimes  \Lambda(y_1,\dots,y_n). 
\] 
 In particular, if $G_\ev$ is non-trivial then $G$ has the same height as $G_\ev$.
\end{theorem}

\begin{proof} 
Let $I$ be the augmentation ideal of $k[G]$.  Pick odd elements $\{y_1, \ldots, y_n\}$ such 
that their residues give a basis of the odd part of the super vector space $I/I^2$.  
Then the ideal $(y_1, \ldots, y_n)$ is a Hopf ideal, and we have an isomorphism 
$k[G]/(y_1, \ldots, y_n) \cong k[G_\ev]$. Since $k[G_\ev]$ is a connected finite group 
scheme, we can find algebraic generators $x_1^\prime, \ldots, x_m^\prime \in k[G_\ev]$ 
such that $k[G_\ev]$ is a truncated polynomial algebra on these generators (\cite[14.4]{Waterhouse:1979a}). 
Let $x_1, \ldots, x_m \in I$ be even liftings of $x_1^\prime, \ldots, x_m^\prime$ to $k[G]$, and let $B$ be the (even) 
subalgebra of $k[G]$ generated by $x_1, \ldots, x_m$.  By construction $\{x_1, \ldots, x_m\}$ give a basis of 
the even part of $I/I^2$. Moreover, the odd elements $y_1, \ldots, y_n$ square to zero and (super) commute, 
hence, generate a copy of $\Lambda(y_1, \ldots, y_n)$ in $k[G]$. 
We therefore have a surjective map 
\[ 
f\colon B \otimes \Lambda(y_1, \ldots, y_n) \to k[G]
\] 
We wish to show that this map is an isomorphism of algebras. The augmentation ideals are the same on both sides by constructions, and, hence,  it suffices to show that $f$ induces an isomorphism on the associated graded  algebras.  Note that $\gr k[G] \cong \bigoplus I^i/I^{i+1}$ inherits the structure 
of a Hopf algebra.

If $f$ is not an isomorphism, its kernel contains a nonzero polynomial involving both $x_i$ and $y_i$. Choose one which involves the minimal number of the variables $y_i$, let $r$ be the maximal index such that this polynomial involves $y_r$, and write it in the form
\[
a + by_r = 0
\] 
where $a$ and $b$ only involve $B$ and $y_1,\dots, y_{r-1}$. Apply the coproduct map $\Delta$ to obtain
\[ 
\Delta(a) + \Delta(b)(y_r \otimes 1 + 1 \otimes y_r) = 0\,.
\]
Since $\Delta(b) = b \otimes 1 + 1 \otimes b + I \otimes I$ (\cite[I.2]{Jantzen:2003a},  there is a term $b \otimes y_r$ in the sum which must vanish. We conclude that $b=0$ and, hence, $a=0$,   contradicting the minimality of $r$. This proves that $f$  is an isomorphism. In particular,  $B$ does not intersect the ideal $(y_1, \ldots, y_n)$, and  so the projection map $k[G] \to k[G]/(y_1, \ldots, y_n) \cong k[G_\ev]$ induces an isomorphism  $B \cong k[G_\ev]$. 
\end{proof}

\begin{remark}
\label{re:masuoka}
\begin{enumerate}
\item
Masuoka~\cite[Theorem~4.5]{Masuoka:2005a} proves, without the finiteness
hypothesis, that there is counital algebra 
isomorphism $k[G] \cong k[G_\ev] \otimes \Lambda((\Lie G)_{\odd}^\#)$.  
\item Since $k[\pi_0(G)]$ sits in even degree, Lemma~\ref{le:pi0} implies 
that the tensor decomposition of Theorem~\ref{th:tensor-dec} holds for any finite supergroup scheme.
\item The structure of the coordinate ring of an ungraded finite connected 
group scheme is known (\cite[Theorem 14.4]{Waterhouse:1979a}). 
Putting it together with Theorem~\ref{th:tensor-dec}, we conclude that for any 
finite connected supergroup scheme $G$ there exists a $k$-algebra isomorphism
\[
k[G] \cong k[x_1, \ldots, x_n]/(x_1^{p^{i_1}}, \ldots, x_n^{p^{i_n}}) \otimes \Lambda(y_1,\dots,y_m)
\]
where $x_i$ are even and $y_j$ are odd. 
\item The Frobenius map $F\colon k[G^{(1)}] \to k[G]$ kills $k[G^{(1)}]_\odd$ 
since odd elements square to $0$ by supercommutativity. Hence, the image of $F$ 
lands in $k[G_\ev]$, that is, the composite
$k[G/G_{(1)}] \to k[G] \to k[G_\ev]$ is injective. 
\end{enumerate}
\end{remark}

\begin{corollary}\label{co:GevG1}
If $G$ is a finite supergroup scheme then $G=G_\ev G_{(1)}$.  
\end{corollary}

\begin{proof}
It follows from Lemma~\ref{le:pi0} that
$G_\ev=G^0_\ev\rtimes\pi_0(G)$. So we may assume that $G$ is
connected.
It then follows from Theorem~\ref{th:tensor-dec} (see 
By Remark~\ref{re:masuoka} the composite
$k[G/G_{(1)}] \to k[G] \to k[G_\ev]$ is injective. Since this is an injective  
map of Hopf algebras, is it faithfully flat (see, for example, \cite[Theorem 14.1]{Waterhouse:1979a})
and, therefore, the corresponding map on group schemes  
$G_\ev \to G \to G/G_{(1)}$ is surjective. Hence, $G=G_\ev G_{(1)}$.
\end{proof}

\begin{warning}
The subgroup $G_{(1)}$ is normal in $G$, but $G_\ev$ 
need not be normal.
\end{warning}

\begin{example} The additive (super)group scheme $\mathbb G_a$ is a purely even group scheme, given by the assignment 
\[\mathbb G_a(R) = R_0^+,\]
where $R_0^+$ the additive group on the even part of a superalgebra $R$. We have $k[\mathbb G_a]=k[T]$ with the $T$ 
primitive in even degree. The Frobenius kernels $\mathbb G_{a(r)}$ are purely even connected unipotent supergroup schemes with 
$k[\mathbb G_{a(r)}]=k[T]/T^{p^r}$, and $T$ primitive. 
\end{example}

\begin{example}
We denote by $\bbG_a^-$ the finite supergroup scheme such that $k\bbG_a^-=k[\si]/(\si^2)$ with $\si$ primitive in odd degree. Then $\bbG_a^-$ is connected and unipotent. As a functor,  $\bbG_a^-$  is defined by  $\bbG_a^-(R)  = R_1^+$, the additive group on the  odd part of a superalgebra $R$.  

More generally, let $V$ be a finite-dimensional vector space, and let $\Lambda^*(V)$ 
be the $\bbZ/2$-graded exterior algebra on $V$ where the elements of $V$ are primitive of odd 
degree. With this convention,  $\Lambda^*(V)$ becomes a supercommutative Hopf algebra and, 
hence, is isomorphic to a group algebra of a product of copies of $\bbG_a^-$,
and hence corresponds to a connected unipotent finite supergroup scheme. 

\end{example}

\begin{example}\label{eg:t2pm}
Let $W_{1,1}^-$ be the finite supergroup scheme such that $kW_{1,1}^-=k[\si]/(\si^{2p})$
with $\si$ primitive in odd degree. Then $W_{1,1}^-$ has height $1$ and sits in
a nonsplit short exact sequence
\begin{equation}\label{eq:Ga-Ga1} 
1 \to \bbG_{a(1)} \to W_{1,1}^- \to \bbG_a^- \to 1. 
\end{equation}

More generally, let $W_{m,1}^-$ be the finite supergroup scheme with $kW_{m,1}^-=k[\si]/(\si^{2p^m})$ 
where $\si$ is primitive in odd degree
and $m\ge 1$. 
Then $W_{m,1}^-$ has height one, and it sits in a nonsplit short exact sequence
\begin{equation}\label{eq:Ga-Wm1} 
1 \to W_{m,1} \to W_{m,1}^- \to \bbG_a^- \to 1
\end{equation}
where $W_{m,1}$ denotes the Witt vectors of length $m$ and height one as described in Appendix~\ref{sec:Witt},
whose group algebra is $kW_{m,1}=k[s]/(s^{p^m})$, and $s=\si^2$ is
primitive in even degree.
\end{example}

\begin{example}
A \emph{$p$-restricted Lie superalgebra} $\fg = \fg_0 \oplus \fg_1$ is a $\bbZ/2$-graded Lie
algebra with a $p$-restriction map on the even part, and such that the
odd part is a $p$-restricted module over the even part. 
The $p$-restricted enveloping algebra $\mcU^{[p]}(\fg)$ is the group algebra of a connected 
finite supergroup scheme which is unipotent if and only if $\fg$ is unipotent.
\end{example}

\begin{lemma}\label{le:mcV}
Let $G$ be a finite supergroup scheme. Then the 
primitive elements in $kG$ form a $p$-restricted Lie superalgebra $\fg=\Lie(G)$ over $k$
with Lie bracket given by commutator and $p$-restriction map given by the $p$-power map in $kG$.
The natural map $\mcU^{[p]}(\fg) \to kG$ induces an isomorphism $\mcU^{[p]}(\fg)\to kG_{(1)}$.
\end{lemma}
\begin{proof}
See Lemma 4.4.2 of Drupieski \cite{Drupieski:2013a}.
\end{proof}

\begin{example}
Example~\ref{eg:t2pm} has height one, so  is of the form $\mcU^{[p]}(\fg)$. The $p$-restricted Lie superalgebra $\fg$ is generated by an element $\si$
in odd degree with  relation $[\si,\si]^{[p^m]}=0$.
\end{example}

\begin{remark} 
\label{re:lie}
If $G$ is a finite connected supergroup scheme of height $1$ with the corresponding Lie algebra $\fg$, 
then $\fg_0$ is an even (restricted) Lie algebra corresponding to $G_\ev$, that is, 
$\mcU^{[p]}(\fg_0)\cong kG_\ev.$  
\end{remark}

\begin{lemma}
\label{le:purelyodd} 
A finite unipotent supergroup scheme $G$ with $G^0_\ev = 1$ is isomorphic to $ (\mathbb G_a^-)^{\times r} \rtimes (\mathbb Z/p)^{\times s}$. 
\end{lemma}

\begin{proof} 
The assumption $G^0_\ev = 1$ implies that $G_0$ has height $1$, and, hence, 
corresponds to a Lie superalgebra $\fg$. By Remark~\ref{re:lie}, $\fg_0=0$, 
therefore, $ \mcU^{[p]}(\fg) = \Lambda^*(\fg_1)$, and, hence, $G_0 = 
(\mathbb G_a^-)^{\times \dim \fg_1}$.  The statement follows 
from Lemma~\ref{le:pi0}. 
\end{proof}

For sub-supergroup schemes $H, H^\prime \leqslant G$, the commutator sub-supergroup scheme is defined as in \cite[II.5.4.8]{Demazure/Gabriel:1970a} as a representable functor. We need an analogue of the following standard result in group theory.  

\begin{lemma} 
\label{le:normal}
Let $G$ be a finite supergroup scheme, and $H, H^\prime \unlhd G$ be normal sub-supergroup schemes.  
Then $[H,H^\prime]$ is normal in $G$. 
\end{lemma} 

\begin{proof} 
It suffices to check pointwise that for $a \in H(R), b \in H^\prime(R)$ and $c\in G(R)$, we have that 
$[a,b]^c \in [H(R),H^\prime(R)]$, where the latter commutator is as of discrete groups. This follows from 
the obvious identity
\begin{equation*} 
c(aba^{-1}b^{-1})c^{-1} = cac^{-1}cbc^{-1}ca^{-1}c^{-1}cb^{-1}c^{-1}.
\qedhere
\end{equation*}
\end{proof} 


\section{Low degree cohomology} 
\label{se:low}

The cohomology $H^{*,*}(G,k)$ of a finite supergroup scheme $G$ is isomorphic
to $\Ext^{*,*}_{kG}(k,k)$. The first index is homological, and the second
is the internal $\bbZ/2$-grading. Drupieski \cite{Drupieski:2013a,Drupieski:2016a}
has proved that $H^{*,*}(G,k)$ is a finitely generated $k$-algebra, which is
graded commutative in the sense that if $x\in H^{m,\alpha}(G,k)$ and 
$y\in H^{n,\beta}(G,k)$ then 
\[ yx = (-1)^{mn}(-1)^{\alpha\beta} xy. \]

We start by identifying the first cohomology group of $G$. 
\begin{lemma}\label{le:H1}
Let $G$ be a finite supergroup scheme with the group of connected components $\pi$.
Then we have $H^{1,0}(G,k)= \Hom_{\Gr/k}(G,\bbG_a)$ and $H^{1,1}(G,k)=\Hom_{\sGr/k}(G^0,\bbG_a^-)^\pi$. 
Moreover, $\Hom_{\sGr/k}(G,\bbG_a)\cong \Hom_{\sGr/k}(G^0,\bbG_{a})^\pi \times \Hom_{\Gr/k}(\pi, \bbG_a)$. 
\end{lemma}
\begin{proof} 
Identification of $H^{1,*}$ with $\Hom$ follows from the standard cobar resolution  used to compute cohomology $H^{*,*}(G,k)$. 
The last statement is proved as in \cite[Lemma~5.1]{Bendel:2001a}. 
\end{proof} 

\begin{lemma}\label{le:Ga1Ga-}
If $G$ is a non-trivial unipotent finite supergroup scheme 
then there is a non-trivial homomorphism from $G$ to either $\bbG_{a(1)}$ or
$\bbG_a^-$ or $\bbZ/p$.
\end{lemma}
\begin{proof} Since $G$ is unipotent, the group of connected components $\pi$  is a $p$-group.   If there are no non-trivial maps to $\bbZ/p$, then $\pi$ is trivial and $G$ is connected. For a finite connected supergroup scheme, if there are no non-trivial homomorphisms from $G$ to $\bbG_{a(1)}$ then 
there are also none to $\bbG_a$.  So if there are also no non-trivial homomorphisms from $G$ to $\bbG_a^-$, Lemma~\ref{le:H1} yields $\Ext^{1,*}_{kG}(k,k)=0$. As $kG$ is a local ring this implies  $G$ is trivial.
\end{proof}

If $f\colon G\to G^\prime$ is a group homomorphism then $f^*\colon H^{*,*}(G^\prime,k)\to H^{*,*}(G,k)$ preserves both the homological and the internal degree,  and commutes with the Steenrod operations (to be discussed in Section~\ref{se:steenrod}). If $N$  is a normal sub-supergroup scheme of $G$ then there is the Lyndon-Hochschild--Serre spectral sequence
\[ 
H^{*,*}(G/N,H^{*,*}(N,k)) \Rightarrow H^{*,*}(G,k) 
\]
in which the internal degrees are carried along, and preserved by all the differentials. The spectral 
sequence also gives  the five-term exact sequence: 
\[\xymatrix{1\ar[r] & H^1(G/N,k)\ar[r] & H^1(G,k) \ar[r] & H^1(N,k)^{G/N} \ar[r]^-{d_2} &  H^2(G/N,k) \ar[r] &H^2(G,k).}\]

\begin{lemma} 
\label{le:Lie.split} Let $\mcL^\prime \subset \mcL$ be Lie superalgebras such that $\mcL^\prime$ is odd and central 
in $\mcL$ and $\mcL/\mcL^\prime$ is even. Then $\mcL \simeq \mcL^\prime \times  \mcL/\mcL^\prime$.
\end{lemma}

\begin{proof}
Let $\mcL_\ev$ be the $p$-restricted Lie sub-superalgebra of even elements in $\mcL$. The assumption 
implies that it is normal and isomorphic to $\mcL/\mcL^\prime$; hence, $\mcL \simeq \mcL^\prime \times \mcL_\ev$.
\end{proof}

\begin{lemma}
\label{le:even} 
If a unipotent finite supergroup scheme $G$ has $H^{1,1}(G,k)=0$, then $G = G_{\rm ev}$. 
\end{lemma} 

\begin{proof}  Since $G/G_{(1)}$ is even by Corollary~\ref{co:GevG1}, the Lyndon-Hochschild-Serre 
spectral sequence applied to the supergroup extension $1\to G_{(1)} \to G \to G/G_{(1)} \to 1$ 
implies that $H^{*,1}(G,k) = H^{*,1}(G_{(1)}, k)$. Hence, the assumption together with 
Lemma~\ref{le:H1} imply that there are no non-trivial maps from $G_{(1)}$ to $\bbG_a^-$. 
We need to show that $G_{(1)}$ is purely even. 

Let $\mcL$ be the unipotent Lie superalgebra associated with $G_{(1)}$.    
Since $\mcL$ is unipotent, we can choose a central series 
\[\mcL_1 \subset \mcL_0 \subset \mcL\] such that 
$\mcL/\mcL_0$ is purely even and $\mcL_0/\mcL_1 \simeq  \Lie \bbG_a^-$. 
By Lemma~\ref{le:Lie.split}, we get that $\mcL/\mcL_1$ has 
$\Lie \bbG_a^-$ as a direct factor, so 
there is a surjective map from $G_{(1)}$ to $\bbG_a^-$, a contradiction. 
\end{proof}

The five-term exact sequence can be used in exactly the same way as in the
proof of \cite[Lemma 1.2]{Bendel/Friedlander/Suslin:1997a},
to prove the following analogue.

\begin{lemma}
\label{le:5-term}
Let $f\colon G\to \ov G$ be a surjective homomorphism of unipotent supergroup schemes. If $f^*\colon H^{1,*}(\ov G,k)\to H^{1,*}(G,k)$ is an isomorphism and  $f^*\colon H^{2,*}(\ov G,k)\to H^{2,*}(G,k)$ is injective then $f$ is an isomorphism. \qed
\end{lemma}

\begin{remark}
\label{re:injectivity} 
Lemma~\ref{le:H1} implies that the condition that $f^*\colon H^{1,*}(\ov G,k)\to H^{1,*}(G,k)$ 
is an isomorphism guarantees that any homomorphism from $G$ to $\bbG_{a(1)}$, $\bbG_a^-$ 
and $\mathbb Z/p$ factors through $\ov G$.
\end{remark}

\section{Steenrod operations}
\label{se:steenrod}

The Steenrod algebra acts on the cohomology 
of any $\bbZ$-graded cocommutative Hopf algebra, and hence also on the cohomology
of any affine supergroup scheme (\cite[Theorem~11.8]{May:1970a}, \cite{Wilkerson:1981a}). 
We recall how the Steenrod operations act using the re-indexing 
introduced in \cite{Benson/Pevtsova:bp2}. In order to make the indexing work for 
$\bbZ/2$-graded algebra, we index with half-integers. 

For $p$ odd, there are natural operations 
\begin{align*} 
\cP^i\colon H^{s,t}(G,k) &\to H^{s+2i(p-1),pt}(G,k) \\
\beta\cP^i\colon H^{s,t}(G,k) & \to H^{s+1+2i(p-1),pt}(G,k), 
\end{align*}
defined in the following cases: when $t$ is even, then  $i\in \bbZ$, and if $t$ is odd, then $i\in \bbZ+\half$.
Note that since $p$ is odd,
$pt$ is congruent to $t$ mod $2$, so the operations 
preserve internal degree as elements of $\bbZ/2$.

The Steenrod operations satisfy the following properties: 
\begin{enumerate}
\item[\rm (i)] $\cP^i=0$ if either $i<0$ or $i>s/2$,\newline
$\beta\cP^i=0$ if either $i<0$ or $i\ge s/2$;
\item[\rm (ii)] Semi-linearity: $\cP^i(ax) = a^p\cP^i(x)$ for $a \in k$;
\item[\rm (iii)] $\cP^i(x)=x^p$ if $i=s/2$;
\item[\rm (iv)] Cartan formula: \newline $\cP^j(xy) = \sum_{i}\cP^i(x)\cP^{j-i}(y)$,
\newline
$\beta\cP^j(xy)=\sum_{i}(\beta\cP^i(x)\cP^{j-i}(y)+\cP^i(x)\beta\cP^{j-i}(y))$;
\item[\rm (v)] The $\cP^i$ and $\beta\cP^i$ satisfy the Adem relations.
\end{enumerate}

We record its action on $H^{*,*}(\bbG_a^-,k)$ (see \cite[Proposition 3.1]{Benson/Pevtsova:bp2}).

\begin{proposition}
\label{pr:Steenrod-on-Ga^-}
One has $H^{*,*}(\bbG_a^-,k)\cong k[\zeta]$, a polynomial ring on  $\zeta$ in degree $(1,1)$. The action of the Steenrod operations on $H^{*,*}(\bbG_a^-,k)$ are given by $\cP^\half(\zeta)=\zeta^p$, $\beta\cP^\half(\zeta)=0$. \qed
\end{proposition}

Next, we describe the analogue of  Proposition 3.6 of \cite{Bendel:2001a} for $G=\bbG_{a(r)} \times (\bbG_a^-)^\epsilon \times (\bbZ/p)^{\times s}$ with $r,s\ge 0$, $\epsilon=0$ or $1$. If $\epsilon=1$ we have
\[ 
H^{*,*}(G,k)=k[x_1,\dots,x_r] \otimes
\Lambda(\lambda_1,\dots,\lambda_r) \otimes k[\zeta] \otimes
k[z_1,\dots,z_s] \otimes \Lambda(y_1,\dots,y_s) 
\]
while if $\epsilon=0$ the term $k[\zeta]$ is missing. The degrees and action of the Steenrod algebra are as follows. We make the assumption that $\lambda_{r+1} = 0=x_{r+1}$, that is, $\cP^0$ kills $\lambda_r$ and $x_r$. 

\begin{table}[!h]
\begin{tabular}{|c|c|c|c|c|c|c|c|c|}
\hline
& \small degree & $\cP^{0^{\phantom{0}}}$ & $\beta\cP^0$ & $\cP^\frac{1}{2}$ & $\beta \cP^\frac{1}{2}$ & $\cP^1$ & $\cP^i$ & $\beta\cP^i$  \\ 
&&&&&&&$(i\ge 2)$ & $(i\ge 1)$ \\ \hline
$\lambda_i$ & $(1,0)$ & $\lambda_{i+1}$ & $-x_i$& $$&  $$ & $0$ & $0$ & $0$ \\
$y_i$ & $(1,0)$ & $y_i$ & $z_i$ &$$&  $$ & $0$ & $0$ & $0$ \\
$\zeta$ & $(1,1)$ &&& $\zeta^p$ &$0$ &&&\\
$x_i$ & $(2,0)$ &  $x_{i+1}$ & $0$ &&&$x_i^p$ & $0$ & $0$ \\
$z_i$ & $(2,0)$ &  $z_i$ & $0$ &&&$z_i^p$  & $0$ & $0$\\ 
\hline
\end{tabular}
\caption{Steenrod operations}
\label{table:Steenrod}
\end{table}

We recall the following theorem of Serre \cite{Serre:1965a} which  is a prototype for both Proposition 3.6 of \cite{Bendel:2001a} and Theorem~\ref{th:B3.6} and will be used in the proof.  The precise  result we quote is a special case of Proposition 3.2 of \cite{Suslin:2006a};  it differs slightly from Serre's original formulation since we need to consider arbitrary coefficients, not just $\mathbb F_p$. 
 
\begin{theorem}
\label{th:Serre} 
Let  $I$ be a homogeneous ideal in $H^*((\bbZ/p)^{\times s},k)$ stable under the Steenrod  operations. If I contains a nonzero element of degree two, then there exists a finite family $\{u_i\} \subset H^2((\bbZ/p)^{\times s},k)$, each of which is a non-trivial  linear combination of $\{z_j\}$ with coefficients in $\mathbb F_p$ such that the product $\prod u_i \in H^*((\bbZ/p)^{\times s},k)$ lies in $I$.  \qed
\end{theorem}

Let
$f\colon G \to \ov G \cong \bbG_{a(r)} \times (\bbG_a^-)^\epsilon
\times (\bbZ/p)^{\times s}$, ($r,s \geq 0, \epsilon = 0$ or $1$) be a
surjective map of finite unipotent supergroup schemes.  The proof of
Theorem~\ref{th:main} uses in an essential way the description of the
kernel of the induced map on degree $2$ cohomology
\begin{equation}
\label{eq:kernel}
f^*\colon H^{2,*}(\ov G,k)\to H^{2,*}(G,k)
\end{equation}
under the assumption that 
\[
f^*\colon H^{1,*}(\ov G,k) \stackrel{\sim}{\to} H^{1,*}(G,k)
\] 
is an isomorphism. There are two  scenarios:  Theorem~\ref{th:B3.6} deals with the case when the kernel $I = \ker f^*$  has an element of degree $(2,0)$ whereas Theorem~\ref{th:bp2}  considers the case of degree $(2,1)$.  We use extensively the observation  that $I$ is stable under the Steenrod operations. 

The following theorem includes the case $\ov G = (\bbG_a^-)^\epsilon \times (\bbZ/p)^{\times s}$ 
in disguise; it corresponds to $r=0$ per our convention that $\mathbb G_{a(0)} =1$. 
\begin{theorem}
\label{th:B3.6}
Let $\ov G=\bbG_{a(r)} \times (\bbG_a^-)^\epsilon \times (\bbZ/p)^{\times s}$, with $r,s\ge 0,\ \epsilon=0$ or $1$. Let $I\subseteq H^{*,*}(\ov G,k)$ be a homogeneous ideal stable with respect to the action of the Steenrod operations.  Suppose $I$ contains a nonzero element of degree $(2,0)$. 
Then one of the following holds:
\begin{enumerate}[\quad\rm(i)]
\item
 Some element of the form $x_r^n\beta\cP^0(v_1)\dots\beta\cP^0(v_m)$  (with $n$ and $m$ not both zero)  lies in  $I$, 
where $v_1,\dots,v_m$ are nonzero elements of
$H^1((\bbZ/p)^{\times s},\bbF_p)\subseteq H^{1,0}(\ov G,k)$, or
\item
$I\cap H^{2,0}(\ov G,k)$ is one dimensional,  spanned by an element of the form $\zeta^2+\gamma x_r$ with $\gamma \in k$.
\end{enumerate}
\end{theorem}

\begin{proof}
We follow, for the most part, the notation and proof in Proposition 3.6 of \cite{Bendel:2001a}, with adjustments as appropriate to deal with the extra factor, $(\bbG_a^-)^\epsilon$.

Any nonzero element $u$ in $I\cap H^{2,0}(\ov G,k)$ has the form
\[ 
u=\alpha \zeta^2 +\sum_{1\le i < j \le r}a_{i,j}\lambda_i\lambda_j + \sum_{1\le j \le r} b_j x_j
+ \sum_{1\le i \le r, 1\le j \le s} c_{i,j}\lambda_iy_j + \sum_{1\le
  i<j\le s}d_{i,j}y_iy_j + \sum_{1\le j\le s}e_jz_j  
  \]
for  scalars $\alpha,a_{i,j},b_j,c_{i,j},d_{i,j},e_j$ which are not all zero, and the term $\alpha\zeta^2$ only occurs if $\epsilon=1$.

First suppose that each such $u$ has  $\alpha\ne 0$. In this case $I\cap H^{2,0}(\ov G,k)$ is one dimensional and $\epsilon=1$. Furthermore, $u$ has to be
sent to a multiple of itself by $\cP^0$. The Cartan formula implies   $\zeta^2$ is killed by $\cP^0$, and so is $u$. The condition $\cP^0(u)=0$ forces $u$ to be of the form
\begin{equation}
\label{eq:u1}  
u = \zeta^2 + \sum_{1\le i < r}a_{i,r}\lambda_i\lambda_r + b_r x_r + \sum_{1\le j \le s}c_{r,j}\lambda_r y_j.
\end{equation}

Assume on the other hand that there exists a $u$ with $\alpha=0$.  Repeated application of $\cP^0$ to such a $u$ results in an element of the form 
$\sum_{1\le i<j\le s}d_{i,j}^{p^{\ell}}y_iy_j+\sum_{1\le j \le s}e_j^{p^\ell}z_j$. So if at least one of the $d_{i,j}$ or $e_j$ is nonzero, we may apply
Theorem~\ref{th:Serre}, and this puts us in case (i) with $m>0$ and $n=0$. So we may assume 
\[ 
u= \sum_{1\le i < j \le r}a_{i,j}\lambda_i\lambda_j + \sum_{1\le j \le r} b_j x_j
+ \sum_{1\le i \le r, 1\le j \le s} c_{i,j}\lambda_iy_j. 
\]
Repeatedly applying $\cP^0$ and stopping just before we get zero, we can assume  $u$ has the form
\begin{equation}
\label{eq:u2} 
u = \sum_{1\le i < r}a_{i,r}\lambda_i\lambda_r+ b_r x_r + \sum_{1\le j\le s}c_{r,j}\lambda_ry_j. 
\end{equation}

So we are now in a situation where $u$ has either the form \eqref{eq:u1} or  \eqref{eq:u2}, and in the first case $I\cap H^{2,0}(G,k)$
is one dimensional. In either case, if some $c_{r,j}$ is nonzero, we apply $\beta\cP^0$ to get
\[ 
\beta\cP^0(u)= \sum_{1\le i < r} a_{i,r}^p \lambda_{i+1}x_r -
\sum_{1\le j \le s}c_{r,j}^p x_ry_j \in I. 
\]
Applying $\beta\cP^1$, we get
\[ \beta\cP^1\beta\cP^0(u)= - \sum_{1\le i<r}a_{i,r}^{p^2}x_{i+1}x_r^p -
\sum_{1\le j \le s} c_{r,j}^{p^2}x_r^pz_j \in I. \]
Now apply $\cP^p$ to get
\[ 
\cP^p\beta\cP^1\beta\cP^0(u)= - \sum_{1\le i<r}a_{i,r}^{p^3}x_{i+2}x_r^{p^2} -
\sum_{1\le j \le s} c_{r,j}^{p^3}x_r^{p^2}z_j \in I. 
\]
Successively applying $\cP^{p^2}$, $\cP^{p^3}$, \dots, we eventually conclude that $I$ contains an element of the form $\sum_j
c_{r,j}^{p^t}x_r^{p^{t-1}}z_j = x_r^{p^{t-1}}(\sum_j
c_{r,j}^{p^t}z_j)$.
The set of all such elements in $I$ is stable under the Frobenius map
(raising all the coefficients to the $p$th power), and therefore there
is a nonzero element with coefficients in $\bbF_p$. This puts us in
case (i) with $m=1$.

If every $c_{r,j}=0$ but some $a_{i,r}$ is nonzero, then  
\[ \beta\cP^1 \beta\cP^0(u)=-\sum_{1\le i < r} a_{i,r}^{p^2}x_{i+1}x_r^p. \]
Now we apply $\cP^p$, then $\cP^{p^2}$, and so on, and just before
we get zero, we get a multiple of a power of $x_r$. This gives case
(i) with $m=0$ and $n>1$.

It remains to consider the case when all $c_{r,j}$ and all $a_{i,r}$ are zero. Then, if $u$ has form \eqref{eq:u1} we are in case (ii), and if $u$ has form \eqref{eq:u2} we are in case (i) with $m=0$ and $n=1$.
\end{proof} 

To complete the description of the kernel of \eqref{eq:kernel}, we quote a result from \cite{Benson/Pevtsova:bp2} 
which describes what happens when the kernel of the map $f^*\colon H^{2,*}(\ov G,k) \to H^{2,*}(G,k)$
has an element of degree $(2,1)$.  Note that in this case we necessarily have $\epsilon =1$.

\begin{theorem}
\label{th:bp2}
Let $G$ be a finite unipotent supergroup scheme and $N$ a normal sub-super\-group
scheme with $G/N \cong \bbG_a^- \times \bbG_{a(r)} \times
(\bbZ/p)^{\times s}$. If the inflation  $H^{1,*}(G/N,k)\to H^{1,*}(G,k)$ is
an isomorphism and $H^{2,1}(G/N,k) \to H^{2,1}(G,k)$ is
not injective then there exists a nonzero element $\xi\in
H^{1,1}(G,k)$ such that  $\beta\cP^0(u)\xi^{p^{r+s-1}(p-1)}=0$ for all $u\in H^{1,0}(G,k)$. \qed
\end{theorem}


 \section{Super Quillen--Venkov} 
 \label{se:qv}
 
We require an analogue of the Quillen--Venkov lemma (\cite{Quillen/Venkov:1972a}). 
The proof in \cite{Quillen/Venkov:1972a}, and its later variants carry over to the present context; 
we adapt a purely representation--theoretic approach due to Kroll ~\cite{Kroll:1984a}.

 \begin{remark}\label{rk:notation}
 If $H\le G$ is a maximal
sub-supergroup scheme with $G=G^0 \rtimes \pi$ unipotent, then there are three
possibilities for $G/H$, namely $\bbG_{a(1)}$, $\bbZ/p$, and
$\bbG_a^-$.
\begin{itemize}
\item
If $G/H\cong \bbG_{a(1)}$ then there is an element $\lambda\in H^{1,0}(G^0,k)^\pi\subseteq H^{1,0}(G,k)$
corresponding to the homomorphism $G \to \bbG_{a(1)}$ as in Lemma
\ref{le:H1}, and an associated element $x=-\beta\cP^0(\lambda)\in
H^{2,0}(G,k)$. 
\item
If $G/H\cong \bbZ/p$ then there is an element $y\in H^{1,0}(\pi,\bbF_p)\subseteq
H^{1,0}(G,k)$ corresponding to the homomorphism $G\to \bbZ/p$ as in
Lemma~\ref{le:H1},
and an associated element $z=\beta\cP^0(y)\in H^{2,0}(G,k)$. 
\item
If $G/H\cong \bbG_a^-$ then there is an element $\zeta\in
H^{1,1}(G,k)\cong\Hom_{\sGr/k}(G,\bbG_a^-)$ corresponding to the homomorphism $G \to \bbG_a^-$ as in
Lemma~\ref{le:H1}.
\end{itemize}
\end{remark}

For $H < G$, we denote by 
\[
\ind_{H}^{G}\colon \Mod H \to \Mod G
\] 
the {\it induction} functor which is the right adjoint to the
restriction functor $\res_{H}^{G}\colon \Mod G \to \Mod H$ (following
the group scheme terminology here, as introduced, for example, in
\cite[I.3]{Jantzen:2003a}).  There is also the coinduction functor
\[
\coind_{H}^{G}\colon \Mod H \to \Mod G
\]
which is left adjoint to the restriction. In the unipotent case induction and coinduction are 
canonically isomorphic (see \cite[I.3]{Jantzen:2003a}) which we use implicitly in the proof below.  
If $H \le G$ is a normal subgroup, then the kernel of the canonical map $\coind_H^G k \to k$ 
is the {\it relative syzygy} $\Omega_{G/H}(k)$. 
The identity map on the trivial representation $k$ induces a map 
\begin{equation}
\label{eq:eta} 
\eta\colon \Omega(k) \to \Omega_{G/H}(k)
\end{equation} 
Similarly, we have a map 
\[ 
\eta'\colon \Omega^{-1}_{G/H}(k) \to \Omega^{-1}(k).
\]
We employ the same notation $\eta'$ for the shifts of this map. 

\begin{lemma}
\label{le:composite}
Let $H\le G$ be a normal sub-supergroup scheme of a finite unipotent
group scheme $G$ with $G/H$ isomorphic to $\bbZ/p$ or
$\bbG_{a(1)}$. Then $z=\beta\cP^0(y)$, respectively
$x=-\beta\cP^0(\lambda)\in H^{2,0}(G,k)$ (cf.\
Remark~\ref{rk:notation}), is represented by the composite
\[\xymatrix@R=3mm{\Omega(k) \ar[r]^-{\eta} &\Omega_{G/H}(k) \ar[r]^{\cong} & \Omega^{-1}_{G/H}(k) \ar[r]^-{\eta'} &
\Omega^{-1}(k)}. \]
\end{lemma}
\begin{proof} 
We prove this in the case where $G/H \cong \bbZ/p$. The case $G/H\cong
\bbG_{a(1)}$ is proved by replacing $z$ by $x$ everywhere.

The cohomology class $z = \beta\cP^0(y) \in \Ext_G^2(k,k)$ is
represented by the extension
\begin{equation} 
\label{eq:ind}
\xymatrix@R=3mm{ 
0\ar[r] & k \ar[r]  & \ind_H^G k\ar[r]& 
\ind_H^G k \ar[r]& k\ar[r]& 0.
}\end{equation}  
This follows from the fact that $\ind_H^G k \cong kG/H = k\mathbb Z/p$, and 
this sequence is the inflation of the extension
\[\xymatrix@R=3mm{ 
0\ar[r] & k \ar[r]  & k\mathbb Z/p\ar[r]& 
k\mathbb Z/p \ar[r]& k\ar[r]& 0 
}\] 
for $\mathbb Z/p$ representing the corresponding cohomology class 
(see, for example, \cite[I.3.4.2]{Benson:1991a}).

Next consider the following commutative diagram with exact rows.
\begin{equation*} 
\vcenter{\xymatrix@R=3mm{
0\ar[r] & \Omega^2(k) \ar[r]\ar[ddd]_{z} & P_1 \ar[rr]\ar[dr]\ar[ddd] 
&& P_0 \ar[r]\ar[ddd] & k \ar[r]\ar[ddd]^1 & 0 \\
& & & \Omega(k)\ar[ur]\ar[d]  \\
& & & \Omega_{G/H}(k)\ar[dr]\ar[dd]^(.3)\cong \\
0\ar[r] & k \ar[r] \ar[ddd]_1 & \ind_H^G k\ar[ur]\ar[rr]\ar[dr]\ar[ddd] & & 
\ind_H^G k \ar[r]\ar[ddd] & k\ar[r]\ar[ddd]^z & 0 \\
& & & \Omega^{-1}_{G/H}(k) \ar[ur]\ar[d] \\
& & & \Omega^{-1}(k)\ar[dr] \\
0 \ar[r] & k \ar[r] & P_{-1} \ar[rr]\ar[ur] & & P_{-2} \ar[r] & \Omega^{-2}(k) \ar[r] & 0
}}
\end{equation*}
By \eqref{eq:ind} the sequence in the middle row represents $z$.
So comparing with a projective
resolution as in the top row, the comparison map $\Omega^2(k)\to k$
represents $z$. Dually, comparing with an injective resolution as in
the bottom row, the comparison map $k \to \Omega^{-2}(k)$ also
represents $z$. Therefore the vertical composite map in the middle of
the diagram also represents $z$.
\end{proof}

Given $ \xi \in H^{s,t}(G,M)$, for each $n \geq 0$ we write $\xi^n$ for the class $\xi^{\otimes n} \in H^{ns, nt}(G, M^{\otimes n})$. 
\begin{proposition}\label{pr:QuillenVenkov}
Let $H$ be a maximal sub-supergroup scheme of a  finite 
supergroup scheme $G$, and let $M$ be a $G$-module. 
Suppose that $\xi \in H^{*,*}(G,M)$ restricts to zero on $H$. Then
\begin{enumerate}
\item[\rm (i)] if $G/H\cong \bbZ/p$ then $\xi^2$ is divisible by
the element $z=\beta \cP^0(y)\in H^{2,0}(G,k)$,
\item[\rm (ii)] if $G/H\cong \bbG_{a(1)}$ then $\xi^2$ is divisible by
the element $x=-\beta \cP^0(\lambda) \in H^{2,0}(G,k)$,
\item[\rm (iii)] if $G/H\cong \bbG_a^-$ then $\xi$ is divisible by 
the element $\zeta\in H^{1,1}(G,k)$.
\end{enumerate}
\end{proposition}

\begin{proof}
We shall start by proving (ii).
Let $\xi\in H^{n,*}(G,M)$, and choose a map
$\Omega^n(k) \to M$ representing $\xi$; by abuse of notation we call
this map $\xi$. We also use $\xi$ to
denote any shift of this map, as a map from $\Omega^{n+i}(k)$ to
$\Omega^i(M)$ for $i\in\bbZ$.

The  exact sequence $\xymatrix@=3mm{k\ar[r]^-{\varepsilon'}& \ind_{H}^G k \ar[r] & \Omega^{-1}_{G/H}(k)}$ induces a triangle in $\stmod(G)$
\[
M \otimes \Omega(\Omega^{-1}_{G/H}(k))\xrightarrow{1 \otimes \eta'}  M 
\xrightarrow{1 \otimes \varepsilon'} M \otimes \ind_{H}^G k \to
M \otimes \Omega^{-1}_{G/H}(k). 
\]
The assumption that $\xi$ restricts to zero on $H$ means that the
restriction of $\xi\colon \Omega^n(k) \to M$ to $H$
factors through a projective. Hence, so does the adjoint map 
\[ 
\Omega^n(k) \to \ind_H^G M = M \otimes \ind_{H}^G k. 
\]  
This adjoint factors as the composite of $\xi$ with $1 \otimes \varepsilon'$.  The fact that this composition factors through a projective implies that there exists a lifting  $\rho'\colon \Omega^n(k) \to  \Omega(\Omega^{-1}_{G/H}(k))$  making the following diagram commute: 
\[
\xymatrix 
{& \Omega^n(k) \ar@{-->}[dl]_-{\rho'}\ar[d]^-{\xi}\ar[dr]&&\\
M \otimes \Omega(\Omega^{-1}_{G/H}(k))\ar[r]^(0.7){1\otimes \eta'}& 
M \ar[r]^-{1 \otimes \varepsilon'} &M \otimes \ind_{H}^G k \ar[r]&
M \otimes \Omega^{-1}_{G/H}(k).} 
\]
Shifting by $\Omega^{-1}$, we get a commutative diagram 
\begin{equation}
\label{eq:rho1}
\xymatrix 
{& \Omega^{n-1}(k) \ar@{-->}[dl]_-{\rho'}\ar[d]^-{\xi}\\
M \otimes \Omega^{-1}_{G/H}(k)\ar[r]^-{1 \otimes \eta'}& M \otimes \Omega^{-1}(k).
}
\end{equation}

Similarly, we can factor $\xi\colon \Omega(k) \to M \otimes \Omega^{-n+1}(k)$ to obtain a commutative diagram 
\[\xymatrix 
{\Omega(k) \ar[r]^-{\eta} \ar[d]^-{\xi}& \Omega_{G/H}(k) \ar@{-->}[dl]^-{\rho}\\
M \otimes \Omega^{-n+1}(k)}
\]
Tensoring with $M$, we get a commutative diagram
\begin{equation}
\label{eq:rho2}
\xymatrix 
{M \otimes \Omega(k) \ar[r]^-{1 \otimes \eta} \ar[d]^-{1\otimes\xi}& 
M \otimes\Omega_{G/H}(k) \ar@{-->}[dl]^-{1\otimes \rho}\\
M \otimes M \otimes \Omega^{-n+1}(k)}
\end{equation}

Putting \eqref{eq:rho1} and \eqref{eq:rho2} together, we get the following diagram, 
where the composite of the maps in the middle row is $1\otimes x$ by Lemma~\eqref{le:composite}:
\[ \xymatrix{&&& \Omega^{n-1}(k) \ar@{-->}[dl]_-{\rho'}\ar[d]^-{\xi}\\
M \otimes \Omega(k) \ar@/^2.5pc/[rrr]|{1\otimes x}
\ar[r]^-{1 \otimes \eta} \ar[d]^-{1\otimes\xi}& 
M \otimes\Omega_{G/H}(k) \ar[r]^\cong 
\ar@{-->}[dl]^-{1\otimes \rho}& M \otimes 
\Omega^{-1}_{G/H}(k)\ar[r]^-{1 \otimes \eta'}& 
M \otimes \Omega^{-1}(k)\\
M \otimes M \otimes \Omega^{-n+1}(k)}\]
Completing the diagram, we see that $\xi^2=(1\otimes\xi)\circ\xi$ 
factors through $x$ either on the left or on the right. 
\[ \xymatrix{\Omega^{n+1}(k)\ar[d]^-\xi\ar[rrr]^-x&&& 
\Omega^{n-1}(k) \ar@{-->}[dl]_-{\rho'}\ar[d]^-{\xi}\\
M \otimes \Omega(k)\ar[r]^-{1 \otimes \eta} \ar[d]^-{1\otimes\xi}& 
M \otimes\Omega_{G/H}(k) \ar[r]^\cong \ar@{-->}[dl]^-{1\otimes \rho}& 
M \otimes \Omega^{-1}_{G/H}(k)\ar[r]^-{1 \otimes \eta'}& 
M \otimes \Omega^{-1}(k)\ar[d]^-{1\otimes\xi}\\
M \otimes M \otimes \Omega^{-n+1}(k)\ar[rrr]^-{1\otimes1\otimes x} 
&&& M \otimes M \otimes \Omega^{-n-1}(k).}\]


The same argument works for part (i). 
Part (iii) is similar but easier. Namely, we have
a short exact sequence of $kG$-modules
\[ 0 \to k \to \ind_{H}^G k \xrightarrow{\varepsilon} k \to 0. \]
We have failed to distinguish whether $k$ is in even or odd
degree, but the two ends are in opposite degrees. The connecting map for this in
$\stmod(kG)$ is $\zeta$, so we have a triangle
\[ k \to \ind_{H}^G k \xrightarrow{\varepsilon} k \xrightarrow{\zeta} \Omega^{-1}(k). \]
If $\xi \colon k \to \Omega^{-n}(M)$ restricts to the zero class on $H$ then the
composite with $\ind_{H}^G k \xrightarrow{\varepsilon} k$ is zero, and so
$\xi$ factors through $\zeta$.
\end{proof}

\section{Nilpotence and projectivity}
\label{se:proj}

We introduce the notion of nilpotence for cohomology classes and
discuss its detection. This is closely related to the detection of projectivity.

\begin{definition}  \label{de:nilp}
Let $G$ be a finite supergroup scheme and  $M$ be a $G$-module. We say that a class  $\xi \in H^{j,*}(G, M)$ is  \emph{nilpotent} if there exists $n \geq 1$ such that $\xi^n \in H^{jn,*}(G, M^{\otimes n})$ is zero.  
\end{definition}

In the remainder of the paper we employ the following terminology. 
Let $G$ be a finite supergroup scheme, and let $\mcH$ be a 
\emph{family of subgroups after field extension}, namely a family of 
pairs $(H,K)$ where $K$ is an extension field of $k$ and $H$ is
a sub supergroup scheme of $G_K$. 
Note that the embeddings of $H$ in  $G_K$ 
need not be defined over the ground field $k$. 

We say that 
nilpotence of cohomology elements is \emph{detected} on the family $\mcH$ if 
for any $G$-module $M$ and cohomology class 
$\xi \in H^{*,*}(G,M)$, we have that $\xi$ is nilpotent if and 
only if $\res^{G_K}_H(\xi_K) \in H^{*,*}(H, M_K)$ is nilpotent
for every $(H,K)\in\mcH$.

Similarly, we say that projectivity of modules is \emph{detected} on the 
family $\mcH$ if for any $G$-module $M$, we have that $M$ is 
projective if and only if $\res^{G_K}_H(M_K)$ is
projective as an $H$-module for every $(H,K)\in\mcH$.

In particular,  we say that nilpotence and projectivity are detected 
on proper subgroups of $G$ after field extensions if the family 
$\mcH$ can be taken to be the family of all pairs $(H,K)$ where
$K$ runs over all field extensions of $k$ and $H$ runs over all
proper subgroups of $G_K$. 
In practice, it always suffices to take 
$K$ to be an algebraically closed field of large enough
finite transcendence degree over $k$.

\begin{lemma}\label{le:hocolim} 
Let $G$ be a finite supergroup scheme, $M$ a $G$-module, and fix an element $\xi \in H^{j,*}(G,M)$ with $j>0$.
With $\xi\colon k \to \Omega^{-j}(M)$ denoting also the corresponding map on modules, let $X$ be the colimit
\[ 
X = \colim\{\xymatrix{k \ar[r]^-\xi & \Omega^{-j}(M) 
\ar[r]^-{1 \otimes \xi}& \Omega^{-2j}(M^{\otimes 2}) 
\ar[r]^-{1 \otimes 1 \otimes \xi}& \Omega^{-3j}(M^{\otimes 3}) 
\ar[r] & \cdots \}} 
\] 
Then $\xi$ is nilpotent if and only if $X$ is projective. 
\end{lemma}

\begin{proof}
If $\xi^n=0$, then the composite of any $n$  consecutive maps in the system defining $X$ factors through a projective, and so $X$ is
projective. Conversely,  if $X$ is projective, then the  map $k \to X$ factors through a projective. Since $k$ is finite dimensional, it factors through a finite 
dimensional projective, and hence  a finite composite  of maps in the defining system factors through a projective. 
This implies that the corresponding power of $\xi$ is zero.
\end{proof} 

Lemma~\ref{le:hocolim} immediately implies the following  result. 

\begin{theorem}
\label{th:projimpliesnilp}
Let $G$ be a finite supergroup scheme.  If a family $\mcH$ of proper  sub supergroup schemes after field extensions detects projectivity of $G$-modules, then it also detects  nilpotence  of cohomology elements.
\end{theorem} 

\begin{proof}
Let $M$ be a $G$-module and $\xi \in H^{j,*}(G,M)$ an element with $j>0$. Represent it by a map $\xi\colon k \to \Omega^{-j}(M)$, and consider the colimit 
$X = \colim \Omega^{-jn}(M^{\otimes n})$ as in Lemma~\ref{le:hocolim}.

Our assumption is that $\xi_K\da_H$  is nilpotent for each
$(H,K)\in\mcH$. That is, for some $n$ depending on $(H,K)$,   the
element $(\xi_K\da_H)^{\otimes n} \in H^{jn,*}(H, M_K^{\otimes n})$ is
zero.  Equivalently, the map $K \to \cdots \to \Omega^{-jn}(M^{\otimes
  n})$  factors through a projective upon restriction to $H$. Hence,
$X_K\da_H$ is projective. Since we assumed that projectivity  is
detected on the family $\mcH$, we conclude that $X$ is a projective $G$-module.  The statement now follows by Lemma~\ref{le:hocolim}. 
\end{proof} 

We omit the  proof of the following lemma since the proof is  similar to 
\cite[Lemma 3.5]{Benson/Iyengar/Krause/Pevtsova:2018a} if one replaces 
$\pi$-support with the cohomological support. See also \cite{Burke:2012a}. 
 
\begin{lemma}
\label{le:nilp}
Let $G$ be a finite supergroup scheme, and $M$ be a $G$-module. The following are equivalent: 
\begin{enumerate}[\quad\rm(a)]
\item 
$M$ is projective,
\item
any class $\xi \in \Ext_G^{>0, *}(M,M)$ is nilpotent.
\qed
\end{enumerate}
\end{lemma} 

Here is a partial converse to Theorem~\ref{th:projimpliesnilp}.

\begin{proposition}
\label{pr:nilpimpliesproj}  
Let $G$ be a finite supergroup scheme.  Suppose that nilpotence in cohomology of $G$-modules is detected  on a family $\mcH$ of proper subgroups of $G$ without field extension (i.e., each pair $(H,K)\in\mcH$ has $K=k$).
Then projectivity of modules is also detected on $\mcH$.
\end{proposition} 

\begin{proof}
Let $N$ be a $G$-module such that $N\da_H$  is projective for all $H\in\mcH$. Then $\Lambda = \End_k(N)$ is projective upon restriction to each $H\in\mcH$ so that for any cohomology class $\xi \in H^{*,*}(G, \Lambda)$, we have $\xi\da_H=0$.  Since nilpotency is detected on $\mcH$ we deduce that all elements $\xi \in H^{>0,*}(G,\Lambda) \cong \Ext^{>0,*}_G(N,N)$ are nilpotent. 
Now apply Lemma~\ref{le:nilp}.
\end{proof} 

\begin{remark} 
The full converse to Theorem~\ref{th:projimpliesnilp} 
is trickier. The argument above fails if we have to extend scalars. 
A deeper reason might be that it is not true that nilpotency of all 
elements in $H^*(G,M)$ implies that $M$ is projective. We refer 
the reader to a cautionary example described in Proposition~5.1 
of \cite{Benson/Carlson:JPAA}: take $G$ to be the Klein group 
$\bbZ/2 \times \bbZ/2$, $p=2$, and $M$ be an infinite 
dimensional module represented by the following diagram: 
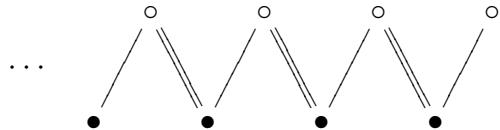
\[\xymatrix@=3mm{&& \circ\ar@{-}[ddl]\ar@{=}[ddr] && \circ\ar@{-}[ddl]\ar@{=}[ddr] && \circ\ar@{-}[ddl]\ar@{=}[ddr] &&\circ\ar@{-}[ddl]\\
\cdots &&&&&&&&\\
&\bullet&&\bullet&&\bullet&&\bullet&} \]
As computed in \cite[Proposition 5.1]{Benson/Carlson:JPAA}, 
all cohomology classes in $H^*(G,M)$ 
are nilpotent (of unbounded degree) whereas the module $M$ is not projective. 

\end{remark}


\section{Inductive detection theorem}\label{se:inductive}

We finish the first part of the paper with the inductive detection
theorem. The point of Theorem~\ref{th:detection} is to cover the cases
of the detection that are straightforward, leaving the task of showing
that the finite unipotent supergroup schemes not covered by
Hypotheses~\ref{hyp:inductive} are precisely the elementary supergroup
schemes from Definition~\ref{def:elementary}; see
Theorem~\ref{th:main}. It is in the preparation work for that theorem
that the degree $2$ cohomology element of Theorem~\ref{th:B3.6}
becomes relevant.

We separate out the hypotheses  since these will appear again in Section~\ref{se:role}.

\begin{hyp}\label{hyp:inductive}
The finite supergroup scheme $G$ is unipotent and satisfies at least one of the following:
\begin{enumerate}
\item[\rm (a)] There is a surjective map $G\to \bbG_{a(1)} \times \bbG_{a(1)}$.
\item[\rm (b)] There is a surjective map $G \to \bbG_a^- \times \bbG_a^-$.
\item[\rm (c)] There are nonzero elements 
\begin{align*} 
\lambda_1,\ldots \lambda_n&\in H^{1,0}(G^0,k)^\pi \subseteq H^{1,0}(G,k)\\
y_1,\dots,y_s&\in  H^{1,0}(\pi,\bbF_p)\subseteq H^{1,0}(G,k) \\
\zeta_1,\dots,\zeta_m&\in H^{1,1}(G,k)\cong \Hom(G, \bbG^-_a) 
\end{align*}
such that $\prod\beta\cP^0(\lambda_i)\prod\beta\cP^0(y_j)\prod \zeta_\ell =0$.
\end{enumerate}
\end{hyp}

\begin{theorem}\label{th:detection}
If Hypothesis~\ref{hyp:inductive} hold for $G$,  then
\begin{enumerate}
\item[\rm (i)] nilpotence of elements of $H^{*,*}(G,M)$ and
\item[\rm (ii)] projectivity of $kG$-modules 
\end{enumerate}
are detected on proper sub-supergroup schemes after field extension.
\end{theorem}

\begin{proof}
  The argument that if $G$ satisfies either condition (a) or (b), then
  projectivity of modules is detected on proper sub-supergroup schemes
  goes exactly as in the case 3(b) of the proof of
  \cite[Theorem~8.1]{Bendel:2001a}; so we will not reproduce it
  here. The main ingredient of the proof is the Kronecker quiver
  lemma, see \cite[Lemma 4.1]{Benson/Carlson/Rickard:1996a}.  Once we
  know detection of projectivity, the detection of nilpotents is
  implied by Theorem~\ref{th:projimpliesnilp}.

We now show that (c) implies detection of nilpotents in $H^{*,*}(G,M)$ on sub-supergroup schemes, without any field extensions. Let $\xi \in H^{n,*}(G,M)$ be  a cohomology class which restricts nilpotently to all proper subgroups of $G$, and let 
$\prod\beta\cP^0(\lambda_i)\prod\beta\cP^0(y_j)\prod \zeta_\ell =0$.  Each of the elements 
$\lambda_i$, $y_j$, $\zeta_\ell$ corresponds to a map from $G$ to $\mathbb Z/p$, $\bbG_{a(1)}$ or $\bbG_a^-$, 
with $\xi$ restricting nilpotently to the kernel of the corresponding map. 
Proposition~\ref{pr:QuillenVenkov} implies that $\xi^{2i+2j+\ell}$ is then divisible by 
$\prod\beta\cP^0(\lambda_i)\prod\beta\cP^0(y_j)\prod \zeta_\ell$, and is therefore zero.

Finally, since the case (c) does not involve field extensions, Proposition~\ref{pr:nilpimpliesproj} implies 
that we also have detection of projectivity in this case. 
 \end{proof}

\part{The detection theorem} 

\section{Witt elementary supergroup schemes}
\label{se:Elem}

In this section we introduce a family of Witt elementary supergroup schemes that play an essential role in our main detection theorem. 

\begin{notation} 
\label{not:diag}
We shall make an extensive use of diagrams to depict many of the unipotent connected 
supergroup schemes to be introduced in this section.  In these diagrams, $\circ$ denotes a 
composition factor isomorphic to $\bbG_{a(1)}$
and $\bul$ denotes a composition factor isomorphic to $\bbG_a^-$. A
single bond represents an extension of $\bbG_{a(1)}$ by
$\bbG_{a(1)}$ to make $\bbG_{a(2)}$ and the double bond represents an
extension of $\bbG_{a(1)}$ by $\bbG_{a(1)}$ to make $W_{2,1}$. The dashed
link denotes an extension of $\bbG_a^-$ by $\bbG_{a(1)}$ to make the
supergroup scheme $W_{1,1}^-$ discussed in Example~\ref{eg:t2pm}. 
\[
\xymatrix@=5mm{\circ \ar@{-}[d]& \circ \ar@{=}[d]& \bul \ar@{--}[d]\\
\circ & \circ & \circ \\
\mathbb G_{a(2)} & W_{2,1} & W_{1,1}^- 
}
\]
\end{notation}

\begin{example}
Let $\fg$ be the $p$-restricted Lie superalgebra described in Example~5.3.3 of Drupieski and Kujawa \cite{Drupieski/Kujawa:sv}. This is generated by
an odd degree element $\si$ and an  even degree element $s$ satisfying $[\si,\si]=2s^{[p]}$. This is unipotent if and only if some $s^{[p^m]}$ is zero.
If $n$ is minimal with this property then $\fg$ has a basis consisting of
$\si,s,s^{[p]},\dots,s^{[p^{m-1}]}$. The restricted enveloping algebra of $\fg$
is the group algebra of the finite supergroup scheme denoted $E_{m,1}^-$ with 
\[ kE_{m,1}^-=k[s,\si]/(s^{p^m},\si^2-s^p) \] where $s$ and $\si$ are
primitive. Note that $(E_{m,1}^-)_\ev\cong W_{m,1}$, the first
Frobenius kernel of length $m$ Witt vectors as introduced in
Appendix~\ref{sec:Witt}, so we have a short exact sequence
\begin{equation} 
\label{eq:Em1-} 
1 \to W_{m,1} \to E_{m,1}^- \to \bbG_a^- \to 1. 
\end{equation}
For $m \geq 2$, there are also short exact sequences
\[ 1 \to W_{m-1, 1}^- \to E_{m,1}^- \to \bbG_{a(1)} \to 1, \]
where $kW_{m-1,1}^- = k[\si]/\si^{2p^{m-1}}$ (see Example~\ref{eg:t2pm}), and 
\[ 1 \to W_{m-1,1} \to E_{m,1}^- \to \bbG_{a(1)} \times \bbG_a^-  \to 1 \]
where the group algebra of $W_{m-1,1}$ is generated by $s^p=\si^2$.
Using Notation~\ref{not:diag}, $E^-_{m,1}$ is represented with the following diagram.`
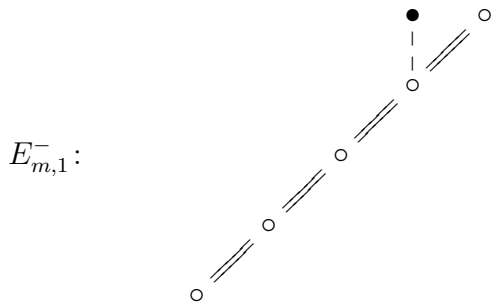
\[ E^-_{m,1} \colon \quad 
\vcenter{\xymatrix@=5mm{&&&& \bul \ar@{--}[d] & \circ\ar@{=}[dl] \\ &&&& \circ \ar@{=}[dl] \\
  &&&\circ \ar@{=}[dl] \\ &&\circ \ar@{=}[dl] \\ &\circ }} \]
\end{example}

As another example, we draw a diagram for $W_{m,1}^-$ of Example~\ref{eg:t2pm}. 
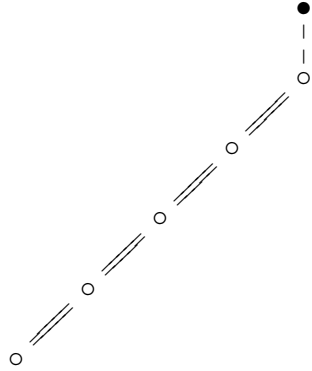
\[ \xymatrix@=5mm{&&&& \bul \ar@{--}[d] \\ &&&& \circ \ar@{=}[dl] \\
  &&&\circ \ar@{=}[dl] \\ &&\circ \ar@{=}[dl] \\ &\circ \ar@{=}[dl] \\
  \circ}  \]


\begin{lemma}\label{le:Ga^-Ga1Ga1}
If $G$ is a finite supergroup scheme which sits in a short exact sequence
\[ 1 \to \bbG_{a(1)} \times \bbG_{a(1)} \to G \to \bbG_a^- \to 1 \]
then there is a non-trivial homomorphism $G \to \bbG_{a(1)}$.
\end{lemma}

\begin{proof}
By Corollary~\ref{co:GevG1},  the height of $G$ is one so it is of the form $\mcU^{[p]}(\fg)$
with $\fg=\Lie(G)$. Then $\fg$ has a two dimensional even part with
trivial commutator and $p$-restriction map, and a 
one dimensional odd part. There is therefore a non-trivial homomorphism
from $\fg$ to the one dimensional trivial Lie algebra $\Lie(\bbG_{a(1)})$,
and this induces a non-trivial homomorphism from $G$ to $\bbG_{a(1)}$.
\end{proof}

Next we classify all extensions of $\bbG_a^- $ by $W_{m,1}$ complementing examples 
\eqref{eq:Ga-Wm1} and \eqref{eq:Em1-}. 

\begin{lemma} 
\label{lem:extensions} 

Let $G$ be a finite supergroup scheme fitting in an extension
\[
1 \to W_{m,1} \to G \to \bbG_a^- \to 1. 
\]
Then \[kG \cong k[s, \sigma]/(s^{p^m}, \sigma^2-s^{p^j})
\]  for some $0 \leq
 j \leq m-1$, where $\sigma$ is odd, $s$ is even, and both are primitive. 
 Hence, $G$ can be represented by the following picture:
 
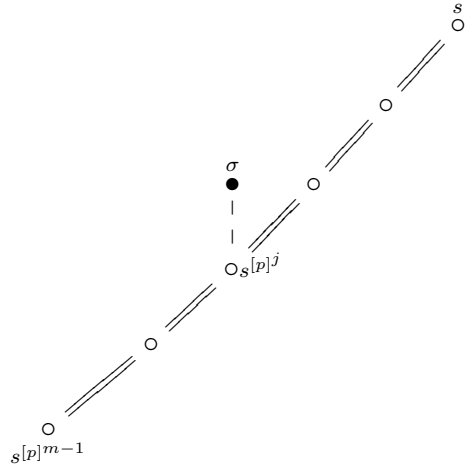
\[\xymatrix@=5mm{&&&&&\overset{s}{\circ}\ar@{=}[dl] \\
 &&&& \circ \ar@{=}[dl] \\
 && \overset{_{\phantom{3}}\sigma_{\phantom{3}}}{\bul} \ar@{--}[d] &\circ \ar@{=}[dl] \\ 
 && \circ \scriptstyle s^{[p]^j}\!\!\!\!\!\!\!\!\!\!\!
 \ar@{=}[dl] \\
 &\circ \ar@{=}[dl] \\
 \underset{s^{[p]^{m-1}}}{\circ}}  \]
 \end{lemma} 

\begin{proof} 
By assumption, $G_\ev = W_{m,1}$. Hence, $G$ has height 1 by 
Theorem~\ref{th:tensor-dec}. By Lemma~\ref{le:mcV}, there is a Lie superalgebra 
$\fg$ such that $\mcU^{[p]}(\fg) \cong kG$. Let $\sigma$ be a lifting to $\fg$ of the generator of $k\bbG_a^-$, 
and let $s$ be an algebraic generator of $kW_{m,1}$, that  is, $s, s^{[p]}, \ldots, s^{[p]^{m-1}}$ be a 
basis of the Lie algebra corresponding to $W_{m,1}$. Then we have 
\[
\half[\sigma,\sigma] = \sum\limits_0^{m-1} a_i s^{[p]^{i}}.
\] Let $j$ be the minimal index such that $a_j\not = 0$ and set 
$s^\prime = \sum\limits_j^{m-1} a_i s^{[p]^{i-j}}$. The generators $\sigma, s^\prime$ 
give the asserted presentation of $\mcU^{[p]}(\fg) \cong kG$.
\end{proof} 

\begin{construction}[$E_{m,n}^-$]
\label{not:Emn^-}
There is a homomorphism  $E_{m,1}^-\to \bbG_{a(1)}$ given by factoring out the ideal of $kE_{m,1}^-$ generated by $\si$. There is also a surjective
map $\bbG_{a(n)} \to \bbG_{a(1)}$ given by the $(n-1)$st power of the Frobenius map. We define $E_{m,n}^-$ to be the kernel of the map from the product to $\bbG_{a(1)}$, so that there is a short exact sequence
\[ 
1\to E_{m,n}^- \to E_{m,1}^- \times \bbG_{a(n)} \to \bbG_{a(1)} \to 1. 
\]
Its group ring is given by
\[ 
kE_{m,n}^-=k[s_1,\dots,s_{n-1},s_n,\si]/(s_1^p, \dots,s_{n-1}^p,s_n^{p^m},\si^2-s_n^p) 
\]
where $s_1,\dots,s_n$ are in even degree and $\si$ is in odd degree. The comultiplication is given by
\begin{align*}
\Delta(s_i)&=S_{i-1}(s_1\otimes 1,\dots,s_i\otimes 1,\ 1\otimes s_1,\dots,
             1\otimes s_i) \qquad (i\ge 1) \\
\Delta(\si)&=\si\otimes 1 + 1\otimes \si
\end{align*}
where the $S_i$ are as defined in Appendix~\ref{sec:Witt}, and come from the  comultiplication in $\Dist(\bbG_a)$.
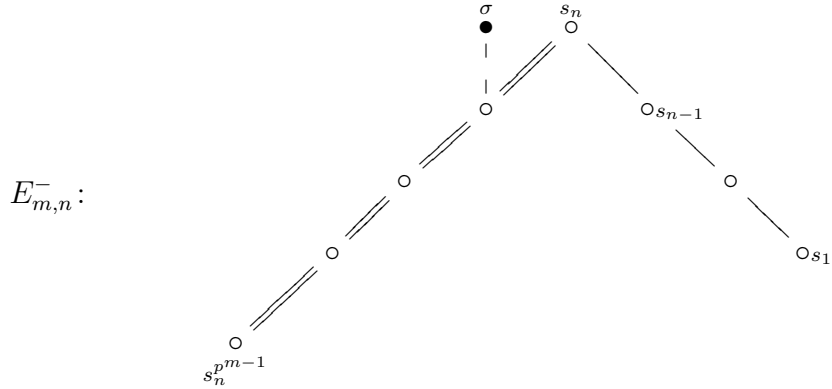
\[ 
E^-_{m,n} \colon \quad 
\vcenter{\xymatrix@=5mm{&&&& \overset{_{\phantom{3}}\sigma_{\phantom{3}}}{\bul} \ar@{--}[d] & \overset{s_n}{\circ}\ar@{=}[dl]\ar@{-}[dr] \\ 
  &&&& \circ \ar@{=}[dl]& & \circ\scriptstyle s_{n-1}\!\!\!\!\!\!\!\!\!\!\ar@{-}[dr] \\
  &&&\circ \ar@{=}[dl]&&&&\circ \ar@{-}[dr]\\ 
  &&\circ \ar@{=}[dl] &&&&&&\circ\scriptstyle s_1\!\!\!\!\!\\\ &\underset{s_n^{p^{m-1}}}{\circ}\\
  }} 
  \] 
We define 
  \[E_{m,n}\colon = (E_{m,n})^-_\ev\] 
  and observe that there is an isomorphism
   \[ E_{m,n} \cong W_{m,n}/W_{m-1.n-1}.\] 
\end{construction}

\begin{definition}\label{defn:Witt}
A finite supergroup scheme is \emph{Witt elementary} if it is isomorphic
to a quotient of $E_{m,n}^-$ by an even subgroup scheme. 
\end{definition}

\begin{remark}
\label{rem:E1n}
For $m=1$, $E^-_{m,n}$ splits as a direct product: 
\[E_{1,n}^- \cong \mathbb G_{a(n)} \times \mathbb G_a^-\] 
\end{remark}

\begin{lemma}
\label{le:Ga-split1} 
Let $G$ be a finite supergroup scheme with the connected component $G^0$ 
and the group of connected components $\pi = \pi(G)$ which is a $p$-group. 
If $G^0$ is an extension 
\[ 1 \to \bbG_{a(1)} \to G^0 \to \bbG_a^- \to 1 \]
then $G = G^0 \times \pi$.
\end{lemma}

\begin{proof}
Since $G^0$ has height, it corresponds to a 2-dimensional Lie superalgebra $\fg = \fg_0 \oplus \fg_1$ by 
Lemma~\ref{le:mcV}. 
Each part is 1-dimensional and must be stabilised by $\pi$.  Since $\pi$ is a $p$-group, it centralises both 
$\bbG_a^-$ and $\bbG_{a(1)}$; hence, centralises $G$.
\end{proof}

\begin{lemma}
\label{le:Ga-split2} 
If $G^0 = \bbG_{a(r)} \times \bbG_a^-$, and $\pi(G)$ is a $p$-group, then the subgroup $\bbG_a^-$ is centralised by $\pi(G)$.
\end{lemma}

\begin{proof} 
We must have that $G_{(1)} = \bbG_{a(1)} \times \bbG_a^-$ is centralised by $\pi$. Now apply Lemma~\ref{le:Ga-split1}. 
\end{proof}

\begin{construction}[\bf $E_{m,n,\mu}^-$]\label{not:Emnmu^-}
The group algebra of $E_{m+1,n+1}^-$ is described in 
Construction~\ref{not:Emn^-} except that we shift the indexing on the even generators $s_i$ 
down by $1$. With that shift, it has the form
\[kE_{m+1,n+1}^- = k[s_0,s_1,\dots,s_{n-1},s_n,\si]/(s_0^p, s_1^p,\dots, s_{n-1}^p,
s_n^{p^{m+1}},\si^2-s_n^p).\] 
 Let $k\bbG_{a(1)} = k[s]/s^p$ with $s$ primitive in even degree. 
For $\mu\in k$, define the supergroup scheme $E^-_{m,n,\mu}$ to be the quotient of 
$E_{m+1,n+1}^-$ given by the embedding $\bbG_{a(1)} \to E^-_{m,n,\mu}$ which sends 
$s$ to $s_0 - ms_n^{p^m}$. 
Thus, there is a short exact
sequence
\[ 
1 \to \bbG_{a(1)} \to E_{m+1,n+1}^- \to E_{m,n,\mu}^- \to 1. 
\]
In the language of Dieudonn\'e modules  introduced in Appendix~\ref{sec:Witt}, $E^-_{m,n,\mu}$ 
is quotient of $E_{m+1,n+1}^-$
by the subgroup scheme of $(E_{m+1,n+1}^-)_\ev\cong \psi(D_k/(V^{m+1},F^{n+1},p))$ given by applying $\psi$ to the
submodule of $D_k/(V^{m+1},F^{n+1},p)$ spanned by $F^n-\mu V^m$.  Explicitly, the group ring $kE_{m,n,\mu}^-$ is given by
\[ 
kE_{m,n,\mu}^- = k[s_1,\dots,s_{n-1},s_n,\si]/(s_1^p,\dots,s_{n-1}^p,
s_n^{p^{m+1}},\si^2-s_n^p) 
\]
where $s_1,\dots,s_n$ are in even degree and $\si$ is in odd degree. The comultiplication is given by 
\begin{align*}
\Delta(s_i)&=S_{i}(\mu s_n^{p^{m}}\otimes 1,s_1\otimes 1,\dots,s_i\otimes  1,\ 
1\otimes \mu s_n^{p^{m}},1\otimes s_1,\dots,1\otimes s_i) \\
\Delta(\si)&=\si\otimes 1 + 1\otimes \si.
\end{align*}
\[ 
E^-_{4,3,\mu} \colon \qquad 
\vcenter{\xymatrix@=5mm{\overset{_{\phantom{3}}\sigma_{\phantom{3}}}{\bul} \ar@{--}[d] & \overset{s_3}{\circ}\ar@{=}[dl]\ar@{-}[dr] \\ 
  \circ \ar@{=}[d]& & \circ\scriptstyle s_2\!\!\!\!\!\ar@{-}[dd] \\
  \circ \ar@{=}[d]& &  \\
  \circ \ar@{=}[dr]&&\circ\scriptstyle s_1\!\!\!\!\! \ar@{-}[dl]^\mu \\ 
  &\underset{s_3^{p^4}}{\circ}}} \]
  We define 
  \[
  E_{m,n,\mu} :=  (E^-_{m,n,\mu})_\ev.
   \]
\end{construction}

\begin{lemma}
\label{lem:elemext} 
Let $G$ be a finite unipotent supergroup scheme.  
\begin{enumerate}[\quad\rm(1)]
\item If for some $n\ge 2$, there is an extension 
\[
1 \to E_{m,n} \to G \to \bbG_a^- \to 1,
\]
 then $kG \cong k[s_1,\dots,s_{n-1},s_n,\si]/(s_1^p, \dots,s_{n-1}^p,s_n^{p^m},\si^2-s_n^{p^i} - \alpha s_1)$ 
 for some $1 \leq i \leq m-1$ and $\alpha \in k$,  where $s_1,\dots,s_n$ are in even degree, $\si$ is in odd 
 degree, and  comultiplication is given by the formulas in \eqref{not:Emn^-}. 
Hence, $G$ can be represented as follows: 
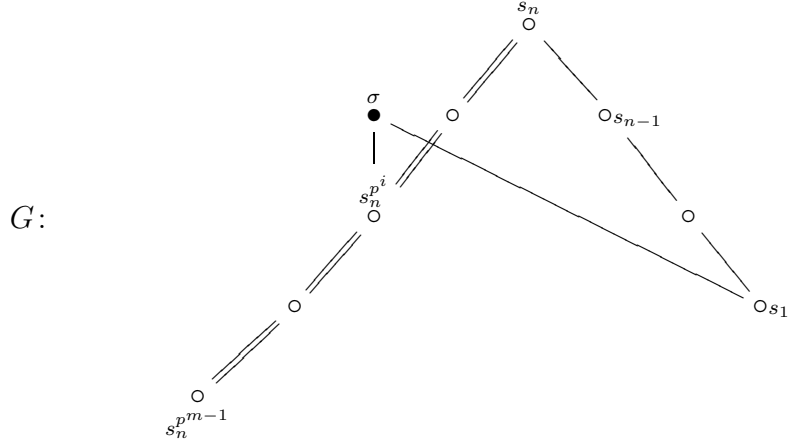
\[
G \colon \quad 
\vcenter{\xymatrix@=5mm{
&&&&  & \overset{s_n}{\circ}\ar@{=}[dl]\ar@{-}[dr] \\
&&&\overset{\sigma}{\bul}\ar@{-}[d]\ar@{-}[ddrrrrr]
& \circ \ar@{=}[dl]& & \circ\scriptstyle s_{n-1}\!\!\!\!\!\!\!\!\!\!\ar@{-}[dr] \\
&&
&\overset{s_n^{p^{i}}}\circ \ar@{=}[dl]&&&&\circ \ar@{-}[dr]\\
  &&\circ \ar@{=}[dl] &&&&&&\circ\scriptstyle s_1\!\!\!\!\!\\\
  &\underset{s_n^{p^{m-1}}}{\circ}\\
  }}
  \]  
\item 
If $G$ fits in the extension 
\[
1 \to E_{m,n, \mu} \to G \to \bbG_a^- \to 1, 
\]
then $kG = k[s_1,\dots,s_{n-1},s_n,\si]/(s_1^p,\dots,s_{n-1}^p,
s_n^{p^{m+1}},\si^2-s_n^{p^i})$ for some $1 \leq i \leq m-1$ and degrees and comultiplication 
as in \eqref{not:Emnmu^-}. 
\end{enumerate}
\end{lemma}

\begin{proof}
We handle only the first case; the second one is similar.  We have $(E_{m,n})_{(1)} = W_{m-1,1} \times \bbG_{a(1)}$, and hence $G_{(1)}$ fits into a short exact sequence:
\[ 
1 \to W_{m-1,1} \times \bbG_{a(1)} \to G_{(1)} \to \bbG_a^- \to 1. 
\]

Let $\fg = \Lie(G)$, so that by Lemma~\ref{le:mcV} we have
$kG_{(1)}\cong \mcU^{[p]}(\fg)$. Let $\si$ be a lift to $\fg$ of a generator
for $\Lie(\bbG_a^-)$. Then $\si$ has odd degree, and
$\frac{1}{2}[\si,\si]$ is some
element of $\Lie(W_{m-1,1}\times \bbG_{a(1)})$, which is the linear span of the
elements 
\[ s_n^{[p]}, s_n^{[p]^2},\dots,s_n^{[p]^{m-1}}, s_1\in
\Lie(E_{m,r}). \]
Arguing exactly as in the proof of Lemma~\ref{lem:extensions}, we can change the generator 
$s_n$ so that $\frac{1}{2}[\si,\si] = s_n^{[p]^i} + \alpha s_1$ without changing the comultiplication 
on $kE_{m,n}$. 
\end{proof} 

\begin{remark}
The finite supergroup schemes $E_{m,n}^-$ and $E_{m,n,\mu}^-$
also appear in the work of Drupieski and
Kujawa~\cite{Drupieski/Kujawa:ga},
where they are denoted $\bbM_{n;m}$ and $\bbM_{n+1;m,-\mu}$ 
respectively.

We also record the structure of the coordinate rings $k[E^-_{m,n}]$ and $k[E^-_{m,n,\mu}]$.
For $k[E^-_{m,n}]$ we have generators $w,x_1,\dots,x_{p^m-1},y$ with
$y$ odd and the remaining generators even. We have relations
$w^{p^{n-1}}=x_1$, $x_ix_j=\binom{i+j}{j}x_{i+j}$, $y^2=0$; which
implies that as an algebra it is a truncated polynomial ring generated by
$w,x_p,x_{p^2},\dots,x_{p^{m-1}},y$ with relations 
$w^{p^n},x_p^p,x_{p^2}^p,\dots,x_{p^{m-1}}^p,y^2$.

For the coalgebra structure, the elements
$w$ and $y$ are primitive, while 
\[ \Delta(x_\ell)=\sum_{i+j=\ell}x_i\otimes x_j +
  \sum_{i+j+p=\ell}x_iy\otimes x_jy. \]
The antipode negates $w$ and $y$, and sends $x_i$ to $(-1)^ix_i$.

The coordinate ring $k[E^-_{m,n,\mu}]$ is the subalgebra of $k[E^-_{m+1,n+1}]$
generated by the elements $w-\mu x_{p^m},x_p,x_{p^2},\dots,x_{p^{m-1}},y$
with the restriction of the comultiplication and antipode.
\end{remark}

\begin{theorem}\label{th:Witt-classification}
Every Witt elementary supergroup scheme is isomorphic to one of the following:
\begin{enumerate}
\item[\rm (i)] $\bbG_a^-$,
\item[\rm (ii)]  $E_{m,n}^-$ with $m,n\ge 1$,
\item[\rm (iii)] $E_{m,n,\mu}^-$ with $m,n\ge 1$ and $0\ne \mu\in k$. 
\end{enumerate}
The only isomorphisms between these are given by
$E_{m,n,\mu}^-\cong E_{m,n,\mu'}^-$ if and only if $\mu/\mu'=a^{p^{m+n}-1}$ for some $a\in k$.

Note that $E_{1,n}^-$ is isomorphic to $\bbG_{a(n)} \times \bbG_a^-$ for $n\ge 1$.
\end{theorem}
\begin{proof}
The quotient of $E_{m,n}^-$ by its entire even part is covered in part (i).
The quotient by a proper subgroup of $(E_{m,n}^-)_\ev$ uses
Theorem~\ref{th:Koch}, and gives parts (ii) and (iii).
\end{proof}

We recall Definition~\ref{def:elementary} from the Introduction: a finite supergroup scheme is {\it elementary} if it is isomorphic to a 
quotient of $E_{m,n}^- \times (\bbZ/p)^{\times s}$.
\begin{remark}
An elementary finite supergroup scheme is isomorphic to one of the following:
\begin{enumerate}
\item[\rm (i)] $\bbG_{a(n)} \times (\bbZ/p)^{\times s}$ with $n,s\ge 0$,
\item[\rm (ii)] $\bbG_{a(n)} \times \bbG_a^- \times (\bbZ/p)^{\times s}$ with $n,s\ge 0$,
\item[\rm (iii)] $E_{m,n}^- \times (\bbZ/p)^{\times s}$ with $m\ge 1$, $n\ge 2$, $s\ge 0$, or
\item[\rm (iv)] $E_{m,n,\mu}^-\times (\bbZ/p)^{\times s}$ with $m,n\ge 1$, $0\ne\mu\in k$
and $s\ge 0$.
\end{enumerate}
\end{remark}

\begin{definition}
The \emph{rank} of an elementary finite supergroup scheme is defined
to be $n+s$ in case (i), and $n+s+1$ in cases (ii)--(iv) of the above
remark.  
\end{definition}


\section{Cohomological calculations}
\label{se:calculations} 

This section is dedicated to computing the cohomology rings of the supergroup schemes introduced in Section~\ref{se:Elem}, and other preparatory results for use in the sequel.

\begin{proposition}
\label{pr:Ga1Ga1Zps}
If $G$ is a semidirect product $(\bbG_{a(1)} \times \bbG_{a(1)}) \rtimes
(\bbZ/p)^{\times s}$ with non-trivial action then there is an element 
$0\ne y\in H^1((\bbZ/p)^{\times s},k)\subseteq H^{1,0}(G,k)$ 
whose product with $0\ne\lambda \in H^1(\bbG_{a(1)},k)\subseteq H^{1,0}(G,k)$
is zero in $H^{2,0}(G,k)$.
\end{proposition}
\begin{proof}
The non-triviality of the product of a pair of elements in $H^1(G,k)\cong \Ext^1_{kG}(k,k)$ 
is the obstruction to 
producing a three dimensional module using these two extensions. So
the proposition follows from the fact that $G$ has a representation of
the form 
\begin{equation*} 
\begin{pmatrix}
1 & (\bbZ/p)^{\times s} & \bbG_{a(1)} \\ 0 & 1 & \bbG_{a(1)} \\ 0 & 0 & 1 
\end{pmatrix}. 
\qedhere
\end{equation*}
\end{proof}

We next discuss cohomology of abelian connected unipotent finite group schemes.
Recall from Appendix~\ref{sec:Witt} that as an augmented algebra, $kG$ is isomorphic to
a tensor product of algebras of the form $kW_{m,1}= k[s]/(s^{p^m})$.
Since cohomology of a finite group scheme $G$ in general only depends on the algebra structure of $kG$,
not on the comultiplication, we get the following description of the cohomology ring.

\begin{theorem}\label{th:HA}
Let $G$ be an abelian connected unipotent finite group scheme.
The cohomology ring $H^*(G,k)$ is a tensor product
of algebras of the form 
\[ H^*(W_{m,1},k) = k[x_m] \otimes \Lambda(\lambda_m) \]
where $\lambda_m$ has degree one and $x_m$ has degree two.

The surjective map
$W_{m,1} \to W_{m-1,1}$ induces an inflation map
\[ H^*(W_{m-1,1},k) \to H^*(W_{m,1},k) \]
sending $x_{m-1}$ to zero and $\lambda_{m-1}$ to $\lambda_m$.
On the other hand, the injective map $W_{m-1,1} \to W_{m,1}$ induces a restriction map
\[ H^*(W_{m,1},k) \to H^*(W_{m-1,1},k) \]
sending $x_m$ to $x_{m-1}$ and $\lambda_m$ to zero.
\end{theorem}
\begin{proof}
The cohomology of the algebra $k[s]/(s^{p^m})$ and the restriction and
inflation maps are well known from the cohomology theory of finite groups. 
See for example Chapter XII of Cartan and Eilenberg \cite{Cartan/Eilenberg:1956a}.
\end{proof}

\begin{proposition}\label{pr:HWm-}
The cohomology of the supergroup scheme $W_{m,1}^-$ of Example~\ref{eg:t2pm}
is given by 
\[ H^{*,*}(W_{m,1}^-,k) = k[x_m,\zeta_m]/(\zeta_m^2) \]
with $|x_m|=(2,0)$ and $|\zeta_m|=(1,1)$.

For $m\ge 2$ the surjective map $W_{m,1}^-\to W_{m-1,1}^-$ induces an inflation map
\[ H^{*,*}(W_{m-1,1}^-,k) \to H^{*,*}(W_{m,1}^-,k) \]
sending $x_{m-1}$ to zero and $\zeta_{m-1}$ to $\zeta_m$.
\end{proposition}
\begin{proof}
The $E_2$ page of the spectral sequence 
\[ H^{*,*}(\bbG_a^-,H^{*,*}(W_{m,1},k)) \Rightarrow H^{*,*}(G,k) \]
has a polynomial generator $\zeta_m$ on the base in degree $(1,1)$, an exterior generator
$\lambda_m$ on the fibre in degree $(1,0)$ and a polynomial generator $x_m$ on the fibre
in degree $(2,0)$. The only differential is $d_2$, and this is determined by
$d_2(\lambda_m)=\zeta_m^2$, $d_2(x_m)=0$.
The inflation maps follow from Theorem~\ref{th:HA}.
\end{proof}

\begin{proposition}\label{pr:Ga^-Gar}
If $G$ is a nonsplit extension
\[ 1 \to \bbG_{a(r)} \to G \to \bbG_a^- \to 1 \]
with $r\ge 1$ then the inflation of $\zeta\in H^{1,1}(\bbG_a^-,k)$ to $G$
squares to zero in $H^{2,0}(G,k)$.
\end{proposition}
\begin{proof}
By Corollary~\ref{co:GevG1}, we have a nonsplit extension
\[ 1 \to \bbG_{a(1)} \to G_{(1)} \to \bbG_a^- \to 1. \]
Hence, $G_{(1)}\cong W_{1,1}^-$ by Lemma~\ref{le:Ga^-Ga1Ga1}.
This implies that ignoring the comultiplication, we have
$kG \cong k\bbG_{a(r-1)} \otimes kW_{1,1}^-$. The result follows
from the case $m=1$ of Proposition~\ref{pr:HWm-}.
\end{proof}

\begin{lemma} 
\label{le:W2_2}
If $G$ is an extension 
\[1 \to W_{2,2} \to G \to \bbG_a^- \to 1,\] 
then there exists a surjective map $G \to W_{2,1}$.  
\end{lemma} 

\begin{proof}
Since $G = G_{(1)}G_\ev$ by Corollary~\ref{co:GevG1},  
taking the first Frobenius kernels, we get an extension
\[1 \to W_{2,1} \to G_{(1)} \to \bbG_a^- \to 1.\] 
Hence, $G/G_{(1)} \cong W_{2,2}/W_{2,1}\cong W_{2,1}$. 
\end{proof} 

\begin{lemma}
\label{le:elempi} 
Let $G$ be a unipotent finite supergroup scheme,  and $f\colon G \to \ov G = \bbG_{a(r)} \times \bbG_a^- \times (\bbZ/p)^{\times s}$  a surjective map of supergroup schemes. Assume that 
\begin{itemize} 
\item[(a)] $f^*\colon H^{1,*}(\ov G,k) \to H^{1,*}(G,k)$ is an isomorphism, and
\item[(b)] $f^*$ is one-to-one restricted to $H^{2}((\bbZ/p)^{\times s},k) \subset  H^{2,0}(\ov G,k)$ 
\end{itemize} 
Then $\pi_0(G)$, the group of connected components of $G$, is isomorphic to $(\bbZ/p)^{\times s}$. 
\end{lemma} 

\begin{proof}
Set $\pi = \pi_0(G)$ and let $\ov \pi$ be the Frattini quotient for $\pi$, that is, the maximal 
quotient isomorphic to an elementary abelian $p$-group. Then the map $f$ factors through $G^0 \rtimes \ov\pi$ 
and we have a commutative diagram 
\[\xymatrix{G^0 \rtimes \pi \ar[r]\ar@{->>}[d] &\pi \ar@{->>}[d]\\
G^0 \rtimes \ov\pi \ar[r]\ar@{->>}[d] &\ov \pi \ar@{=}[d]\\
\ov G \ar[r] &(\mathbb Z/p)^{\times s}
}\]
If $\pi \to \ov \pi$ is not an isomorphism, Lemma~\ref{le:5-term} implies that there 
exists an element $u$ in $H^{2}(\ov \pi,k) = H^2((\mathbb Z/p)^{\times s},k)$ which pulls back to 
zero in $H^2(\pi,k)$ and, hence, in $H^{2,0}(G,k)$. Inflating the class $u$ to 
$H^{2,0}(\ov G,k)$,  we get an element in $\Ker f^* \cap H^{2}((\bbZ/p)^{\times s},k)$ 
contradicting assumption (b). Hence, $\pi \cong (\mathbb Z/p)^{\times s}$. 
\end{proof}

The result below is a denouement of the preceding developments. It's import is that, in the situation of Theorem~\ref{th:B3.6}(ii), various finite 
(super)group schemes cannot be quotients of $G$, $G^0$ and  $G^0_{\rm ev}$. Theorem~\ref{th:forbidden} together with Theorem~\ref{th:bp2} 
are the major inputs in the proof of the detection Theorem~\ref{th:main}. 

\begin{theorem}
\label{th:forbidden} 
Let $G$ be a unipotent finite supergroup scheme,  and $f\colon G \to \ov G = \bbG_{a(r)} \times \bbG_a^- \times (\bbZ/p)^{\times s}$ 
a surjective map of supergroup schemes.
Assume that 
\begin{itemize} 
\item[(a)] $f^*\colon H^{1,*}(\ov G,k) \to H^{1,*}(G,k)$ is an isomorphism, and
\item[(b)] $I = \Ker \{f^*\colon H^{2,*}(\overline G, k) \to H^{2,*}(G,k)\}$ 
is one dimensional, spanned by an element of the form 
$\zeta^2+\gamma x_r$ with $0\neq \gamma \in k$.
\end{itemize} 
Then the following statements hold.
\begin{enumerate}[\quad\rm(I)]
\item $G$ cannot have as a quotient the following supergroup schemes:
\begin{enumerate}[\quad\rm(i)]
\item
	$(\bbG_{a(1)} \times \bbG_{a(1)}) \rtimes (\bbZ/p)^{\times s}$,
\item
	$(\bbG_a^- \times \bbG_a^-) \rtimes (\bbZ/p)^{\times s}$.
\end{enumerate}
\item 
The restriction $f_0 = f{\downarrow_{G^0}}\colon G^0 \to \ov G^0$ satisfies  the following cohomological conditions:  
\begin{enumerate}[\quad\rm($a_0$)] 
\item
$f_0^*\colon H^{1,*}(\ov G^0,k) \to H^{1,*}(G^0,k)$ is an isomorphism,
 \item
  $\Ker f^*_0 \cap H^{2,0}(\ov G^0,k)$ is one dimensional, spanned by $\zeta^2+\gamma x_r$.
\end{enumerate} 
\item
The following connected supergroup schemes cannot be quotients of $G^0$:
\begin{enumerate}[\quad\rm(i)]
\item
$H$ given by a nonsplit extension  $1 \to \bbG_{a(r)} \to H \to \bbG_a^- \to 1$, 
\item
$W_{m,1}^-$,
\item
$W_{2,1}$.
\end{enumerate}
\item
$G^0_{\rm ev}$ cannot have $W_{2,2}$ as a quotient.  
\end{enumerate} 
\end{theorem}

\begin{proof} 
(I).  Let $\rho\colon G \to H$ be a surjective map of unipotent group schemes, 
and suppose that $H$ surjects further on a group scheme $H^\prime$ which is 
isomorphic to $\bbG_{a(1)}$, $\bbG_a^-$ or $\mathbb Z/p$. By Remark~\ref{re:injectivity}, 
we have a commutative diagram 

\begin{equation}
\label{eq:comm1} 
 \xymatrix{ G \ar@{->>}[r]^-f\ar@{->>}[d]_-\rho& \ov G \ar@{->>}[d]^-{\ov \rho}\\
H \ar@{->>}[r]^{\chi}& H^\prime
}
\end{equation} 

Lemma~\ref{le:H1} implies that $\ov \rho: \ov G \to H^\prime$ induces an injective map 
on $H^{1,*}$.  Moreover, the explicit calculation of cohomology for $\ov G$ further implies 
that the map $H^{*,*}(H^\prime, k) \to H^{*,*}(\ov G, k)$ is injective. Since $H^\prime = \bbG_{a(1)}, 
\bbG_a^-$ or $\mathbb Z/p$, we have that $H^{1,*}(H^\prime, k)$ is a $1$-dimensional vector space. 
Let $\alpha \in H^{1,*}(H^\prime, k)$ be a linear generator. Then the assumption $(a)$ together with 
the commutativity of \eqref{eq:comm1} imply that 
\[ 0 \not = (\ov \rho \circ f)^*(\alpha) = (\chi \circ \rho)^*(\alpha) \in H^{1,*}(G,k).
\]
In Case (I.i), assume that there is a surjective map $G \to H$ where $H=(\bbG_{a(1)} \times \bbG_{a(1)}) \rtimes (\bbZ/p)^{\times s}$. 
There are maps $\chi: H \to H^\prime$ 
with $H^\prime = \bbG_{a(1)}, \mathbb Z/p$.  By Proposition~\ref{pr:Ga1Ga1Zps}, taking for 
$\alpha$ the elements $y$ and $\lambda$, we obtain a relation 

\[f^*(\ov \rho^*(y) \ov \rho^*(\lambda)) = \rho^*(\chi^*(y)\chi^*(\lambda)) =0. \] 

Hence, $0 \not =  \ov \rho^*(y) \ov \rho^*(\lambda)$ is in $I$ which contradicts the assumption $(b)$ completing 
the proof in that case.

In Case (I.ii), we assume  there is a surjective map $G \to H$ where $H=(\bbG_a^- \times \bbG_a^-) \rtimes (\bbZ/p)^{\times s}$.
Cohomology of $H$ is computed explicitly in \cite{Benson/Pevtsova:bp2}; there exist non trivial elements $\lambda_1 \in H^{1,0}(H,k)$ and 
$\zeta \in H^{1,1}(H,k)$ such that $\lambda_1\zeta =0$. Arguing as in (I.i), we get a contradiction with the assumption $(b)$ again. 

(II.$a_0$). Let $\pi = \pi_0(G)$ be the group of connected components of $G$. 
By Lemma~\ref{le:elempi}, we have $\pi \cong (\bbZ/p)^{\times s}$, which is the same as $\ov \pi = \pi_0(\ov G)$. 
The map $f\colon G \to \ov G$ induces a commutative diagram of five-term sequences: 
\[
  \xymatrix{ H^{1,*}(\pi,k) \ar@{=}[d] \ar[r] & H^{1,*}(G,k) \ar[r] &
    H^{1,*}(G^0,k)^\pi \ar[r]& H^{2,*}(\pi,k) \ar[r] \ar@{=}[d]&H^{2,*}(G,k) \\
    H^{1,*}(\ov\pi,k) \ar[r] & H^{1,*}(\ov G,k) \ar[r]\ar[u]_{f^*} &
    H^{1,*}(\ov G^0,k) \ar[u]_{f_0^*}\ar[r]^0& H^{2,*}(\ov\pi,k)
    \ar[r] &  H^{2,*}(\ov G,k) \ar[u]_{f^*}\\
  }
\]
Since $f^*$ is an iso on $H^{1,*}$, we conclude that it induces an isomorphism $H^{1,*}(\ov G^0,k) \cong H^{1,*}(G^0,k)^\pi$.  
It remains to show that $\pi$ acts trivially on $H^{1,*}(G^0,k)$. 
 
By Lemma~\ref{le:H1}, 
We have $H^{1,*}(G^0,k) \cong H^{1,0}(G^0,k) \oplus H^{1,1}(G^0,k)$ with $H^{1,0}(G^0,k)  
\cong \Hom(G^0, \mathbb G_a)$, $H^{1,1}(G^0,k) \cong \Hom(G^0, \mathbb G_a^-)$ by 
Lemma~\ref{le:H1}, and the action of $ \pi$ fixing the even and odd parts. 
 
The assumption that $f^*$  is an isomorphism on $H^1$ together with Lemma~\ref{le:H1} imply that 
\begin{align} 
\dim \Hom(G^0, \mathbb G_a^-)^\pi =1 \\
\Hom(G^0, \mathbb G_a)^\pi = \Hom(\mathbb G_{a(r)}, \mathbb G_{a(r)}).
\end{align}  
Hence, to show that $\pi$ acts trivially on $H^{1,*}(G_0,k)$, 
we need to show the same two equalities for $\Hom(G^0, \mathbb G_a^-)$ and $\Hom(G^0, \mathbb G_a)$. 

We first show that $\dim_k\Hom(G^0, \mathbb G_a^-) = 1$. Suppose $\dim _k \Hom(G^0, \mathbb G_a^-) \geq 2$. 
Since $\pi$ is a $p$-group, there exists a two-dimensional $\pi$-invariant subspace of $\Hom(G^0, \mathbb G_a^-)$ and, 
hence, a $\pi$-invariant quotient of the form $\mathbb G_a^-\times \mathbb G_a^-$. But this implies that $G$ 
has a quotient of the form $H = (\bbG_a^- \times \bbG_a^-) \rtimes (\bbZ/p)^{\times s}$ which is disallowed by (I.ii). Hence, 
$\dim_k\Hom(G^0, \mathbb G_a^-) = 1$. 

We now consider $\Hom(G^0, \mathbb G_a)$. First, since $G^0$ is
finite, there exists a number $n$ such that 
$\Hom(G^0, \mathbb G_a) = \Hom(G^0, \mathbb G_{a(n)})$. Pick the maximal $n$ so that the map $ G^0 \to \mathbb G_{a(n)}$ is surjective. 
The standard projection $\mathbb G_{a(n)}) \to \mathbb G_{a(1)}$ 
induces a map on Hom spaces $\Hom(G^0, \mathbb G_{a(n)}) \to \Hom(G^0, \mathbb G_{a(1)})$; the action of $\pi$ descends 
along this map since the Frobenius map is $\pi$-equivariant. If $\dim_k\Hom(G^0, \mathbb G_{a(1)}) >1$, then arguing just as 
in the case of $\mathbb G_a^-$ we deduce a contradiction with (I.i).  Hence, $\Hom(G^0, \mathbb G_{a(1)})=1$. Therefore, 
\[\Hom(G^0, \mathbb G_{a(n)}) \cong \Hom(\mathbb G_{a(n)}, \mathbb G_{a(n)})
\]
It remains to show that $n=r$. Note that $\Hom(\mathbb G_{a(n)}, \mathbb G_{a(n)}) \simeq \mathbb G_a^{\times n}$ as a group scheme, with the action of $\pi$ preserving the  group scheme structure.  Since $\mathbb G_a^{\times n}$ is connected, the action of $\pi$ must be trivial, hence, $r=n$.

(II.$b_0$). The projection $f\colon G \to \ov G$ 
induces a map on spectral sequences making the following diagram commute: 
\begin{equation}\label{eq:spectral}\xymatrix{& H^{\text{{\bf *},*}}(\pi, H^{\text{{\bf *},*}}(G^0,k)) \ar@{=>}[r] & H^{\text{{\bf *},*}}(G,k) \\
H^{\text{{\bf *},*}}(\mathbb Z/p^{\times s},k) \otimes H^{\text{{\bf *},*}}(\ov G^0, k) 
 \ar[r]^-\sim &  H^{\text{{\bf *},*}}(\ov \pi, H^{\text{{\bf *},*}}(\ov G^0,k))
\ar[u] \ar@{=>}[r] & H^{\text{{\bf *},*}}(\ov G,k)\ar[u]^-{f^*}
}\end{equation} 
Here, the star for the internal degree is preserved by the spectral sequence. The bottom sequence collapses at 
the $E_2$ page giving an isomorphism $H^{\text{{\bf *},*}}(\mathbb Z/p^{\times s},k) \otimes H^{\text{{\bf *},*}}(\ov G^0, k) 
\cong H^{\text{{\bf *},*}}(\ov G,k)$. 
Since $\zeta^2 + \gamma x_r \in H^{2,0}(\ov G^0, k) = H^{2,0}(\bbG_{a(r)} \times \bbG^-_a,k)$, 
we conclude that it belongs to the kernel of $f_0^*$.  It remains to show that this class generates 
the kernel of $f^{*}_0$  on $H^{2,0}$. 

Let
\begin{equation} 
\label{eq:filtration}
\xymatrix{H^{2,0}(G,k)& F^1H^{2,0}(G,k)\ar@{_(->}[l]\ar[d]^{\equiv}&  F^0H^{2,0}(G,k)\ar@{_(->}[l]\ar[d]^{\equiv}\\
& F^1H^{2,0}(\ov G,k) & F^0H^{2,0}(\ov G,k)\ar@{_(->}[l]
}
\end{equation} 
be the filtration on $H^{2,0}$ with subquotients giving the $E_\infty$ term of the spectral sequences. 

We consider another diagram induced by $f$: 
\begin{equation}
\label{eq:diagram} 
\xymatrix{H^{2,0}(\ov G^0,k) \ar[r]^{f_0^*}\ar@{^(->}[d]^\rho & H^{2,0}(G^0,k)^\pi \\
H^{2,0}(\ov G,k)\ar@<1ex>[u] \ar[r]^{f^*} & H^{2,0}(G,k)\ar[u]^i 
}
\end{equation} 
The left vertical map induced by the embedding $\ov G^0 < \ov G$ splits since 
\[H^{2,0}(\ov G,k) \cong H^2(\pi,k) \oplus H^1(\pi, H^{1,0}(\ov G^0,k)) \oplus H^{2,0}(\ov G^0,k).\]
The left vertical map $\rho\colon H^{2,0}(\ov G^0,k)\hookrightarrow H^{2,0}(\ov G,k)$ is 
the identification of $H^{2,0}(\ov G^0,k)$ with the last direct summand. 

The right vertical map $i\colon H^{2,0}(G,k) \to H^{2,0}(G^0,k)^\pi$ is the edge homomorphism 
of the top row spectral sequence in \eqref{eq:spectral}, hence, 
\begin{equation} 
\label{eq:keri}
\Ker i = F^1H^{2,0}(G,k). 
\end{equation} 
 
Let $\alpha \in H^{2,0}(\ov G^0,k)$ be a class in the kernel of $f^*_0$.  Then $f_0^*(\alpha) = i f^* \rho(\alpha) =0$ implies that $f^* \rho(\alpha) \in \Ker i = F^1H^{2,0}(G,k)$.  Since $F^1H^{2,0}(G,k) \cong F^1H^{2,0}(\ov G,k)$  by \eqref{eq:filtration}, there exists  $\beta \in F^1H^{2,0}(\ov G,k) = H^2(\pi,k) \oplus H^1(\pi, H^{1,0}(\ov G^0,k))$, such that $f^*(\rho(\alpha)) =  f^*(\beta)$, that is, 
\[
f^*(\rho(\alpha) - \beta) =0.
\] 
Assumption (b) now implies that $\rho(\alpha) - \beta$ is a multiple of $\zeta^2 + \gamma x_r$ and, hence, $\rho(\alpha) - \beta \in \Im \rho$. 
Therefore, $\beta \in \Im \rho$.  This implies that $\beta = 0$ since $\Im \rho \cap F^1H^{2,0}(\ov G,k) = 0$. We conclude that 
$f^*(\rho(\alpha)) = 0$, and, hence,  $\alpha$ is a multiple of $\zeta^2 + \gamma x_r$. Hence the kernel is one-dimensional. 

(III).  We apply the same argument as in Case (I) but to $f_0\colon G^0 \to \ov G^0$.  Once again, we have a 
commutative diagram of surjective maps: 
\[
\xymatrix{ G^0 \ar@{->>}[r]^-{f_0}\ar@{->>}[d]_-\rho& \ov G^0 \ar@{->>}[d]^-{\ov \rho}\\ H \ar@{->>}[r]^{\chi}& H^\prime
}
\]

For (III.i),  Proposition~\ref{pr:Ga^-Gar} gives an element $\zeta \in H^{1,1}(H^\prime, k)$ such that 
$\chi^*(\zeta)^2 =0$. Hence, commutativity of the diagram above implies that $0 \not = (\ov \rho^*(\zeta))^2$ 
is in the kernel of $f_0^*$ contradicting the assumption II($b_0$), and completing the proof in this case. 

Case (III.ii) follows from Proposition~\ref{pr:HWm-} in a similar fashion taking $H^\prime = \bbG_a^1$ and $\alpha = \zeta_m$. 
 
If $G^0$ has a quotient $W_{2,1}$, then $\beta\cP^0(\lambda_2)$, 
where $\lambda_2$ is a degree $(1,0)$ cohomology generator of $H^{*,*}(W_{2,1},k)$,
is in the kernel of $f_0^*$, contradicting II($b_0$).  

Finally, Case (IV) follows from Lemma~\ref{le:W2_2} and case (II.iii). 
\end{proof}

\begin{corollary}\label{co:Aeven}
Let $G$ be a unipotent finite supergroup scheme satisfying 
the assumptions of Theorem~\ref{th:forbidden}. Let $A = G/[G_{\rm ev}, G_{\rm ev}]$. 
Then $A^0_{\rm ev}$ is isomorphic to a quotient of $E_{m,n} = (E_{m,n}^-)_{\rm ev}$ for some $m,n >0$. 
\end{corollary}
\begin{proof} 
First we claim that
$\dim_k\Hom_{\Gr/k}(A_\ev,\bbG_{a(1)})=1$. This is because if this
dimension is two or greater then $G$, and, hence, $G^0$, has a quotient which is a nonsplit
extension of $\bbG_a^-$ by $\bbG_{a(1)}$, which is not allowed by 
Theorem~\ref{th:forbidden}. 

Next, we claim that $\dim_k\Hom_{\Gr/k}(A^0_\ev,\bbG_{a(1)})=1$. This
is because if this dimension is two or greater then $G$ has a quotient
which is a semidirect product
$(\bbG_{a(1)}\times\bbG_{a(1)})\rtimes(\bbZ/p)^{\times s}$ with non-trivial
action. This is  once again disallowed by Theorem~\ref{th:forbidden}. 

By Theorem~\ref{th:forbidden}(III), $A^0_\ev$ does not have $W_{2,2}$ as a quotient.  
Together with the condition $\dim_k\Hom_{\Gr/k}(A^0_\ev,\bbG_{a(1)})=1$ this allows us to  
apply Lemma~\ref{ab:quotients}, concluding that 
$A^0_\ev$ is isomorphic to a quotient of the group scheme $E_{m,n}$.
\end{proof} 

Now for the promised computation of cohomology of Witt elementary supergroup schemes. 

\begin{theorem}\label{th:HEmn^-}
The cohomology of the group $E_{m,n}^-$ (as defined in \eqref{not:Emn^-}) 
$(m\ge 2,\ n \ge 1)$ is given by
\[ H^{*,*}(E_{m,n}^-,k) = 
k[x_{m,1},\dots,x_{m,n},\zeta_m] \otimes 
\Lambda(\lambda_{m,1},\dots,\lambda_{m,n}) \]
with $|x_{m,i}|=(2,0)$, $|\zeta_m|=(1,1)$ and
$|\lambda_{m,i}|=(1,0)$. 

For $m\ge 3$, the surjective map $E_{m,n}^-\to E_{m-1,n}^-$ induces an inflation map
\[ H^{*,*}(E_{m-1,n}^-,k)\to H^{*,*}(E_{m,n}^-,k) \]
sending $x_{m-1,i}$ to $x_{m,i}$ $(1\le i\le n-1)$, $x_{m-1,n}$ to zero, 
$\zeta_{m-1}$ to $\zeta_m$ and $\lambda_{m-1,i}$ to $\lambda_{m,i}$
$(1\le i \le n)$.

The surjective map $E_{2,n}^- \to E_{1,n}^- = \bbG_{a(n)} \times \bbG_a^-$
induces an inflation map
sending $x_i$ to $x_{2,i}$ $(1\le i \le n-1)$, $x_n$ to $\zeta^2$,
$\zeta_2$ to $\zeta$ and $\lambda_i$ to $\lambda_{2,i}$. In
particular, the kernel of 
\[ H^{2,0}(\bbG_{a(n)} \times \bbG_a^-,k) \to
H^{2,0}(E_{2,n}^-,k) \] 
is one dimensional, spanned by $\zeta^2-x_n$.
\end{theorem}
\begin{proof}
Again, we use the fact that the cohomology only depends on the algebra
structure of the group algebra and not on the comultplication. The
algebra structure is described in Definition~\ref{defn:Witt}, and is a
tensor product $k\bbG_{a(n-1)} \otimes kE_{m,1}^-$. The first factor gives
the generators
$\lambda_{m,1},\dots,\lambda_{m,n-1},x_{m,1}\dots,x_{m,n-1}$, so
we need to compute $H^{*,*}(E_{m,1}^-,k)$. We do this using the
spectral sequence
\[ H^{*,*}(\bbG_a^-,H^{*,*}(W_{m,1},k))\Rightarrow
H^{*,*}(E_{m,1}^-,k). \]
This has the same $E_2$ page as the spectral sequence in the proof of
Proposition~\ref{pr:HWm-}, but all the differentials are zero. This
accounts for the generators $x_{m,n}$, $\zeta_m$ and $\lambda_{m,n}$.
The inflation maps again follow from Theorem~\ref{th:HA}.
\end{proof}

\begin{theorem}\label{th:HEmnmu^-}
The cohomology of the group $E_{m,n,\mu}^-$ of 
\eqref{not:Emnmu^-} is given by
\[ H^{*,*}(E_{m,n,\mu}^-,k) = 
k[x_{m,1,\mu},\dots,x_{m,n,\mu},\zeta_{m,\mu}] \otimes 
\Lambda(\lambda_{m,1,\mu},\dots,\lambda_{m,n,\mu}) \]
with $|x_{m,i,\mu}|=(2,0)$, $|\zeta_{m,\mu}|=(1,1)$ and
$|\lambda_{m,i,\mu}|=(1,0)$. 

The surjective map $E_{m+1,n+1}^- \to E_{m,n,\mu}^-$ induces an inflation
map
\[ H^{*,*}(E_{m,n,\mu}^-,k) \to H^{*,*}(E_{m,n}^-,k) \]
sending each element to the corresponding element without the
subscript $\mu$, except that it sends $x_{m,n,\mu}$ to zero.
\end{theorem}
\begin{proof}
The proof is essentially the same as for Theorem~\ref{th:HEmn^-}.
\end{proof}

\begin{remark}
The computation in Theorem~\ref{th:HEmn^-} also appears in  Proposition~3.2.1\,(1) and (3) and Lemma~3.2.4 of Drupieski and Kujawa~\cite{Drupieski/Kujawa:ga}. Similarly, Theorem~\ref{th:HEmnmu^-} should be compared with Proposition~3.2.1\,(4) and (5) of~\cite{Drupieski/Kujawa:ga} and Lemma~3.1.1\,(3) and Remark~2.2.3\,(1)  of~\cite{Drupieski/Kujawa:cs}. 
\end{remark}

\begin{remark}  
We tabulate the action of the Steenrod operations on  $H^{*,*}(E_{m,n}^-,k)$, for use in the proof of Theorem~\ref{th:elemcoh}.   
The table for $H^{*,*}(E_{m,n, \mu}^-,k)$ looks exactly the same after adding $\mu$ to all the indices; cf.~Table~\ref{table:Steenrod}.

\begin{center}
\begin{tabular}{|c|c|c|c|c|c|c|c|c|c|}
\hline
 \small degree & &$\cP^{0^{\phantom{0}}}$ & $\beta\cP^0$ & $\cP^\frac{1}{2}$ & $\beta \cP^\frac{1}{2}$ & $\cP^1$ & $\cP^i$ & $\beta\cP^i$  &\\ 
&&&&&&&$(i\ge 2)$ & $(i\ge 1)$& \\ \hline
 $\lambda_{m,i}$ & $(1,0)$ & $\lambda_{m, i+1}$ & $-x_{m,i}$& $$&  $$ & $0$ & $0$ & $0$ & $1\leq i <n$\\
$\lambda_{m,n}$ & $(1,0)$ & $0$ & $-\zeta_m^2$& $$&  $$ & $0$ & $0$ & $0$ &\\
$\zeta_m$ & $(1,1)$ &$$ & $$ & $\zeta_m^p$ & $0$ & $$ & $$  & $$&\\
$x_{m,i}$ & $(2,0)$ &  $x_{m,i+1}$ & $0$ & $$ & $$ &$x_i^p$ & $0$ & $0$ &$1\leq i <n$ \\
$x_{m,n}$ & $(2,0)$ &  $0$ & $0$ & $$ & $$ &$0$ & $0$ & $0$ &\\
\hline
\end{tabular}
\end{center}
\end{remark}


\section{Cohomological characterisation of elementary supergroups}
\label{se:coh}
The purpose of this section is to show that elementary supergroups as
introduced in Definition~\ref{def:elementary} can be characterised
cohomologically.  Recall that for
$\ov G = \mathbb G_{a(r)} \times \mathbb G_a^- \times (\mathbb
Z/p)^{\times s}$, we employ the following notation for the standard
generators in cohomology:
\[
H^{*,*}(\ov G,k) = k[x_1,\dots,x_r] \otimes
\Lambda(\lambda_1,\dots,\lambda_r) \otimes k[\zeta] \otimes
k[z_1,\dots,z_s] \otimes \Lambda(y_1,\dots,y_s) 
\]
Theorems~\ref{th:HEmn^-} and \ref{th:HEmnmu^-} show that if $G$ is  an elementary supergroup scheme equipped with a surjection $G \to \ov G$ which induces an isomorphism on $H^{1,*}$, then either $f$ is an isomorphism or $\Ker f^*$ falls under the case  (ii) of Theorem~\ref{th:B3.6}. Theorem~\ref{th:elemcoh} proves a partial converse to this statement,  and is the key step in the proof of Theorem~\ref{th:detect}.

\begin{lemma}\label{le:abelian}
Let $1 \to Z \xrightarrow{f} G \xrightarrow{\psi} A \to 1$ be a central extension of 
group schemes with $Z\cong \bbG_{a(1)}$ and $A$ abelian. If the connecting homomorphism 
$d_2\colon H^1(Z,k) \to H^2(A,k)$ is zero then $G$ is abelian.
\end{lemma}

\begin{proof}
The five-term sequence of the central extension shows that there is
an element $\tilde u\in H^{1}(G,k)$ whose restriction is $f^*(\tilde u)=u\in H^{1}(Z,k)$.
Applying Lemma~\ref{le:H1}, we see that there is a homomorphism
$\phi\colon G \to \bbG_a$ whose composite with $Z \to G$ is nonzero.
Then $(\psi,\phi)\colon G \to A \times \bbG_a$ is an embedding, and $G$ is
a subgroup scheme of an abelian group scheme, hence abelian.
\end{proof}

The following proposition, which is the key observation necessary for the proof of Theorem~\ref{th:elemcoh}, 
gives a cohomological criterion to establish that certain extensions of abelian finite group schemes are abelian themselves. 

\begin{proposition}\label{pr:abelian}
Let $1\to Z \to G \to A \to 1$ be a central extension of group
schemes with $Z\cong \bbG_{a(1)}$ and $A$ abelian. 
The following are equivalent:
\begin{enumerate}
\item[\rm (i)]
$G$ is abelian.
\item[\rm (ii)] There exists an abelian finite group scheme $A^\prime$
  and a surjective map $A^\prime \to A$ such that the composition
  $\xymatrix{H^1(Z,k) \ar[r]^-{d_2} & H^2(A,k) \ar[r]& H^2(A^\prime,
    k)}$ is zero.  The induced map in cohomology sends
  $d_2(\lambda) \in H^2(A,k)$ to zero in $H^2(A',k)$ for all
  $0\ne \lambda \in H^1(Z,k)$.
\end{enumerate}
\end{proposition}

\begin{proof} 
(i) $\Rightarrow$ (ii): Take $A'=G$ and use the five-term sequence.

(ii) $\Rightarrow$ (i): Let $1\to A'' \to A' \to A \to 1$ be the short exact sequence 
given by the surjection $A^\prime \to A$.  
Form the pullback $X$ of $G \to A$ and $A'\to A$:
\[ \xymatrix{&&1\ar[d] & 1 \ar[d] \\
&& Z\ar@{=}[r]\ar[d] & Z\ar[d] \\
1\ar[r] & A'' \ar[r] \ar@{=}[d] & X\ar[r]\ar[d]  & G \ar[r] \ar[d] & 1 \\
1 \ar[r] & A'' \ar[r] &A' \ar[r]\ar[d] & A \ar[r]\ar[d] & 1 \\
&& 1 & 1} \]
If $d_2(\lambda)$ goes to zero in $H^2(A',k)$ then the sequence
\[ 1 \to Z \to X \to A' \to 1 \]
satisfies the conditions of Lemma~\ref{le:abelian}, and so
$X$ is abelian. Since $G$ is a quotient of $X$, it follows that $G$ is abelian.
\end{proof}

\begin{theorem}
\label{th:elemcoh} 
Let $G$ be a unipotent finite supergroup scheme, 
and $f\colon G \to \ov G = \bbG_{a(r)} \times \bbG_a^- \times (\bbZ/p)^{\times s}$ 
a surjective map of supergroup schemes.
Assume that 
\begin{enumerate} [\quad\rm(1)] 
\item $f^*\colon H^{1,*}(\ov G,k) \to H^{1,*}(G,k)$ is an isomorphism, 
\item $\Ker(f^*)\cap H^{2,0}(\ov G,k)$ is one dimensional, spanned by an 
element of the form $\zeta^2+\gamma x_r$ with $0\ne \gamma\in k$.
\item  $\Ker(f^*)\cap H^{2,1}(\ov G,k) =0$,
\item There does not exist $i \in \bbZ_{\geq 0}$ and $y \in H^1((\bbZ/p)^{\times s}, \bbF_p) \subset H^{1,0}(\ov G,k)$ such that 
$\zeta^{2p^i}\beta\cP^0(y)$ or $\zeta^{2p^i+2}$ lie in $\Ker \{f^*:H^{*,*}(\ov G,k) \to H^{*,*}(G,k)\}$.
\sloppy{

}

\end{enumerate} 
Then $G$ is isomorphic to 
$E^-_{m,r}\times (\bbZ/p)^{\times s}$ or $E^-_{m, r+1, \mu}$ for some $m \geq 1$, $\mu \in k$.  

\end{theorem}

\begin{proof}  
The proof has three essential reduction steps: 
\begin{enumerate}
\item[Step (1)] The first step is to show that $G_\ev$ is normal in $G$, and $G/G_\ev \cong \mathbb G_a^-$. 
\item[Step (2)] Let $A = G/[G_\ev,G_\ev]$. The second step is to show that $A$ is isomorphic to either 
$E^-_{m,r}\times (\bbZ/p)^{\times s}$ or $E^-_{m, r+1, \mu}$ for some $m \geq 1$, $\mu \in k$.  
\item[Step (3)] Finally, we show that $G \cong A$.
\end{enumerate} 

\vspace{0.1in} 

By Lemma~\ref{le:elempi}, $\pi \cong (\mathbb Z/p)^{\times s}$. 

Let $\psi\colon \ov G \to \bbG_a^-$ be the projection map,  and let $H = \Ker\{\psi \circ f\colon G \to \ov G \to \bbG_a^-\}$. We now show that $H = G_\ev$, proving Step (1). 

We have the five-term sequence associated  with the extension $1 \to H \to G \to \bbG_a^-\to 1$ of which we only need the odd internal degree part:
\[ 
\xymatrix{0 \ar[r]& H^{1,1}(\bbG_a^-,k)\ar[r]^{\sim}_{(\psi \circ f)^*}& 
H^{1,1}(G,k)\ar[r]^-{\rm res}_-{=0}& H^{1,1}(H,k)^{\bbG_a^-} 
\ar[r]^-{d_2}_-{=0}&H^{2,1}(\bbG_a^-,k)\ar@{^(->}[r]_-{f^*}& H^{2,1}(G,k).}
\]
The first map is an isomorphism since $f^*$ is an isomorphism by assumption (i), 
and $\psi^*$ is an isomorphism on $H^{1,1}$ since we know cohomology of $\ov G$ and $\bbG_a^-$ explicitly. 
Assumption (iii) implies that the last map is an embedding. Hence,  
$H^{1,1}(H,k)^{\bbG_a^-}=0$ and, therefore, $H^{1,1}(H,k)=0$ since $\bbG_a^-$ is unipotent. We conclude that $H$  
is even by Lemma~\ref{le:even}.  Since $G/H \cong \bbG_a^-$, $H$ is the largest even 
subgroup scheme; hence, $H = G_\ev$ which proves the claim. 
 
Assumption $(1)$ implies that $r$ is maximal such that there is a surjective $\pi$-invariant map $G^0 \to \bbG_{a(r)}$ 
since any $\pi$-invariant surjection $G^0 \to \bbG_{a(s)}$ induces an embedding in cohomology 
$H^1(\bbG_{a(s)},k)\hookrightarrow H^{1,0}(G,k)$ by Lemma~\ref{le:H1}.  
We claim that $r$ is also maximal subject to the existence of a $\pi$-invariant surjective map $G^0_\ev \to \bbG_{a(r)}$. 
Suppose, to the contrary, that there is a $\pi$-invariant surjective map $G^0_\ev \to \bbG_{a(r+1)}$, and let $N$ be the kernel.
Since $G^0_\ev = G^0 \cap G_\ev$, we have that $G^0/G^0_\ev \cong G/G_\ev \cong \bbG_a^-$. 
We have a commutative diagram of $\pi$-invariant homomorphisms: 
\[\vcenter{\xymatrix{ & N \ar[d]\ar@{=}[r] &N\ar[d] && \\
1 \ar[r] & G^0_\ev \ar[r] \ar[d]& G^0 \ar[r] \ar[d] & \bbG_a^- \ar[r]\ar@{=}[d] & 1 \\
1 \ar[r] & \bbG_{a(r+1)} \ar[r] & \overline{G}^0 = G^0/N \ar[r] & \bbG_a^- \ar[r] & 1 \\
}}\]
If the extension on the bottom row splits,  then it $\pi$-splits by Lemma~\ref{le:Ga-split2}. 
Hence, there is a $\pi$-invariant surjective map $G^0 \to \bbG_{a(r+1)}$ which contradicts 
maximality of $r$.  On the other hand, if the map does not split, then the inflation of 
$\zeta\in H^{1,1}(\bbG_a^-,k)$ to $\overline{G}^0$  is a non trivial cohomology class in 
$H^{1,1}(\ov G^0,k)$ which squares to zero in $H^{2,0}(G,k)$ by Proposition~\ref{pr:Ga^-Gar}. 
Inflating $\zeta$ further to $G^0$ via the projection $G^0 \to \overline G^0$, we get a non trivial 
$\pi$-invariant cohomology class in $H^{1,1}(G^0,k)$ which squares to $0$.  Hence $\zeta^2$ is in 
the kernel of the map $f^*: H^{2,0}(\ov G,k) \to H^{2,0}(G,k)$ which contradicts assumption
(2). We therefore conclude that $r$ is maximal such that there is a surjective map $G^0_\ev \to \bbG_{a(r)}$ as claimed.

\vspace{0.1in}

Since $G_\ev \unlhd G$ is a normal subgroup scheme, Lemma~\ref{le:normal} implies that $[G_\ev,G_\ev] \unlhd G$. 
Let $A=G/[G_\ev,G_\ev]$, so that $A_\ev$ is the abelianisation of
$G_\ev$. 
\begin{claim} We have that $A$ is isomorphic to $E_{m,r}^- \times (\bbZ/p)^{\times s}$ or 
$E_{m, r+1, \mu}^- \times (\bbZ/p)^{\times s}$ for some $m \geq 1$. \end{claim}

\begin{proof}[Proof of the Claim] Corollary~\ref{co:Aeven} implies that 
$A^0_\ev$ is isomorphic to a quotient of
$E_{m,n}$ for some $m,n\ge 1$.
By Theorem~\ref{th:Witt-classification}, this implies that $A^0_\ev$ is isomorphic either to
$E_{m,n}$ or to $E_{m,n,\mu}$ for some $0\ne \mu\in k$.  
We divide into two cases according to these two possibilities.
Looking at homomorphisms from these to $\bbG_{a(r)}$, we see that in the first
case $n=r$, while in the second case $n=r+1$.

{\it Case I: $A^0_\ev \cong E_{m,r}$.} This case splits further into two subcases. 
\begin{enumerate} 
\item $r=1$. Since $E_{m,1} = W_{m,1}$, we have that $A^0 = A^0_{(1)}$ fits in the extension 
\[ 1 \to W_{m,1} \to A^0 \to \bbG_a^- \to 1, \]
and, hence, is described by Lemma~\ref{lem:extensions}. The cohomological restriction in the assumption (2)
implies that the only allowed possibility is $A^0 \cong E_{m,1}^-$ since in any other 
case $A^0$, and, hence, $G^0$, will have a quotient isomorphic to $W_{2,1}$ or $W_{1,1}^-$ which are
disallowed by Theorem~\ref{th:forbidden}.III(iii). Hence, $A \cong E_{m,1}^- \times (\bbZ/p)^{\times s}$. In terms of 
the diagrams for the possibilities 
for $A^0$ as in Lemma~\ref{lem:extensions}, the only way to attach the node $\sigma$ to avoid quotients 
isomorphic to $W_{2,1}$ and $W_{1,1}^-$ is to the node marked with $s^{[p]}$. 

\item $r>1$.  In this case, $kA^0$ is described by Lemma~\ref{lem:elemext}(1). The coefficient $\alpha$ in 
the relation $\sigma^2 - s^{[p]^i} -\alpha s_1$ must be zero since $G^0$, and, hence, $A^0$, has a 
quotient isomorphic to $\bbG_{a(r)}$.  The parameter $i$ must be $1$ since for $i>1$ there will be a 
quotient isomorphic to $W_{2,1}$.  In terms of the picture in Lemma~\ref{lem:elemext}, the node $\sigma$ 
can be connected {\it only} to the node $s^{[p]}$, for otherwise the top two nodes on the left arm will 
form a quotient isomorphic $W_{2,1}$. Hence, $A^0 \cong E_{m,r}^-$. Since $A_\ev$ is abelian, the 
group of connected components of $A$ acts trivially on $A^0_\ev$. Since $(\bbZ/p)^{\times s}$ is a $p$-group, it also acts 
trivially on the quotient $A^0/A^0_\ev \cong \mathbb G_a^-$. Therefore, 
$A \cong E_{m,r}^- \times (\bbZ/p)^{\times s}$.
\end{enumerate}

{\it Case II: $A^0_\ev \cong E_{m,r+1, \mu}$.}
 This case is similar.  The possibilities for $kA^0$ are given by Lemma~\ref{lem:elemext}(ii). 
 All of them but one are disallowed by Theorem~\ref{th:forbidden}.III(iii). We conclude 
 that $A^0 \cong E_{m,r+1, \mu}^-$, and, therefore, we can identify $A$ with
$E^-_{m,r+1,\mu}\times (\bbZ/p)^{\times s}$.
\end{proof}

Now that we have identified $A=G/[G_\ev,G_\ev]$, it remains to show
that $G  = A$,  that is $[G_\ev,G_\ev] =1$. We prove this by contradiction. Assume that 
$G \not = A$.  

Note that $[G_\ev,G_\ev] \subseteq [G,G] \subseteq G^0$ since the 
group of connected components of $G$ is abelian.  Hence, 
$[G_\ev,G_\ev]$ is a connected unipotent finite group scheme. 
Therefore, there exists a maximal
proper subgroup $N$ of $[G_\ev,G_\ev]$ such that $[G_\ev,G_\ev]/N\cong \bbG_{a(1)}$ 
giving rise to a central extension
\begin{equation}
\label{eq:central} 
1 \to \bbG_{a(1)} \to G/N \to A \to 1.
\end{equation}  
Let $\psi\colon G \to A$, $\phi: A \to \ov G$ be the projection maps; we factor $\psi$ as $\xymatrix{ G  \ar[r]^-{\psi_2} & G/N \ar[r]^-{\psi_1} &A}$. 
The map  $f\colon G \to \overline G$ then factors as follows: 
\[\xymatrix{ 
f\colon G  \ar[r]^-{\psi_2} & G/N \ar[r]^-{\psi_1} & A \ar[r]^-{\phi} & \overline G = \bbG_{a(r)} \times \bbG_a^- \times (\bbZ/p)^{\times s} 
}
\]
Since $\phi$, $\psi$ are surjective, the induced maps on $H^1$ are injective. Since the composition 
\[
\xymatrix{f\colon H^{1,*}(\ov G,k) \ar@{^(->}[r]^-{\phi^*} & H^{1,*}(A,k) \ar@{^(->}[r]^-{\psi^*} & H^{1,*}(G,k)}
\] 
is an isomorphism, we conclude that
 \[\xymatrix{\psi^*\colon H^{1,*}(A,k) \ar[r]^-\sim &H^{1,*}(G,k)}\]
is also an isomorphism. 

We again consider two cases:  $A \cong E_{m,r}^- \times (\bbZ/p)^{\times s}$ and $ A \cong E_{m, r+1, \mu}^- \times (\bbZ/p)^{\times s}$.  

\noindent 
{\it Case I:} $A \cong E_{m,r}^- \times (\bbZ/p)^{\times s}$. Assume that $m \geq 2$.  By Theorem~\ref{th:HEmn^-}, 
\begin{enumerate}
\item $H^{*,*}(A,k) \cong k[x_{m,1},\dots,x_{m,n},\zeta_m] \otimes 
\Lambda(\lambda_{m,1},\dots,\lambda_{m,n}) \otimes k[z_1, \ldots, z_s] \otimes \Lambda(y_1, \ldots, y_s)$, 
\item $\phi^*: H^{1,*}(\ov G, k) \to H^{1,*}(A,k)$ is an isomorphism, 
\item $\Ker \{\phi^*: H^{2,*}(\ov G, k) \to H^{2,*}(A,k)\} = k\langle \zeta^2 + \gamma x_r\rangle$.
\end{enumerate} 
In particular, $\phi^*$ and $f^*$ have the same kernel on $H^{2,*}(\ov G,k)$. Therefore, 
$\psi_1^*\colon H^{2,*}(A,k) \to H^{2,*}(G/N,k)$ is one-to-one restricted to 
$\phi^*(H^{2,*}(\ov G,k))$. 

Consider the five-term exact sequence induced by \eqref{eq:central}:
\[
\xymatrix{H^{1,*}(A,k) \ar[r] & H^{1,*}(G/N,k) \ar[r] & H^{1,*}(\bbG_{a(1)},k) \ar[r]^-{d_2}& H^{2,*}(A,k) \ar[r]^{\psi_1^*}& H^{2,*}(G/N,k) } 
\]
Let $\lambda \in H^{1,0}(\bbG_{a(1)},k)$ be any linear generator. Since $G_\ev/N$ is 
non-abelian,  $d_2(\lambda) \in H^{2,0}(A,k)$ is a nonzero 
element in the kernel of $\psi_1^*$ by Lemma~\ref{le:abelian}, and, hence, a nonzero 
element in the kernel of $\psi^*$. By Theorem~\ref{th:HEmn^-}, the only 
linear generator of $H^{2,0}(A,k)$ which is {\it not} in the image of $\phi^*$, is $x_{m,r}$. Hence, 
replacing $\lambda$ by a nonzero multiple if necessary, we may assume that
\[d_2(\lambda) = x_{m,r} + u\]
where $u \in \phi^*(H^{2,*}(\ov G,k))$.

\begin{claim} 
\label{cl:u} 
$u = \alpha \zeta^2$ for some $\alpha \in k$. 
\end{claim}
\begin{proof}
We prove the claim by consecutive application of Steenrod operations, similarly to Theorem~\ref{th:B3.6}.  
Since any element in $\Ker \{\psi^*:  H^{2,*}(A,k) \to H^{2,*}(G,k)\}$ must have the form $ax_{m,r} + v$ 
with $ a \not = 0$ and $ v \in \phi^*(H^{2,0}(\ov G,k))$, we have 
\[ 
\dim_k \Ker \{\psi^*\colon  H^{2,*}(A,k) \to H^{2,*}(G,k)\} = 1.
\] 
Let 
\[ u = \alpha \zeta_m^2 +\sum_{1\le i < j \le r}a_{i,j}\lambda_{m,i}\lambda_{m,j} + \sum_{1\le j <r} b_j x_{m,j}
+ \sum_{1\le i \le r, 1\le j \le s} c_{i,j}\lambda_{m,i}y_j + \sum_{1\le
  i<j\le s}d_{i,j}y_iy_j + \sum_{1\le j\le s}e_jz_j  \]
for some constants $\alpha,a_{i,j},b_j,c_{i,j},d_{i,j},e_j\in k$ which are
not all zero.  Since $\Ker \psi^*$ is stable under the Steenrod operations and $\cP^0(x_{m,r})=0$, 
we conclude that $\cP^0(u) =0$, which forces $u$ to be of the form
\begin{equation*}\label{eq:u3} 
u = \alpha\zeta_m^2 + \lambda_{m,r}(\sum_{1\le i < r}a_{i}\lambda_{m,i})  + \lambda_{m,r} (\sum_{1\le j \le s}c_{j}y_j).
\end{equation*}
Proposition~\ref{pr:Steenrod-on-Ga^-}  together with the Cartan formula imply (by induction) that 
\begin{equation*}
\cP^i(\zeta_m^{2i}) = \zeta_m^{2pi} 
\end{equation*}
and all other Steenrod operations vanish on $\zeta_m^{2i}$. 

Since $\beta\cP^0(\lambda_{m,r}) = \zeta_m^2$, applying $\beta\cP^0$ to \eqref{eq:u3}, we get 
\begin{equation*}\label{eq:u4} 
\zeta_m^2(\sum_{1\le i < r}a^p_{i}\lambda_{m,i+1})  + \zeta_m^2 (\sum_{1\le j \le s}c^p_{j}z_j).
\end{equation*}
Applying $\beta\cP^1$, we get 
\begin{equation*}\label{eq:u5} 
\zeta_m^{2p}(\sum_{1\le i < r-1}a^{p^2}_{i}\lambda_{m,i+2})  + \zeta_m^{2p+2}  + \zeta_m^{2p} (\sum_{1\le j \le s}c^{p^2}_{j}z_j).
\end{equation*}
If there is a nonzero coefficient $c_j$, then applying $\cP^p, \cP^{p^2}, \ldots $ and then taking invariants 
under the Frobenius map as in the proof of Theorem~\ref{th:B3.6}, we eventually get that the kernel of the map
$\psi^*\colon H^{2,*}(A,k) \to H^{2,*}(G,k)$ contains an element $\zeta_m^{2p^i}\beta\cP^0(y)$ with 
$y \in H^1((\bbZ/p)^{\times s},\mathbb F_p) \subset H^{1,0}(\ov G, k) \cong H^{1,0}(A, k)$.  This means that $f^*$ vanishes on 
$\zeta^{2p^i}\beta\cP^0(y)$, contradicting assumption (4). Hence, we can assume  
that all coefficients $c_j$ are zero.

Suppose there is a coefficient $a_i \not = 0$. Then \eqref{eq:u5} has the form
\begin{equation}\label{eq:u6} 
\zeta_m^{2p}(\sum_{1\le i < r-1}a^{p^2}_{i}\lambda_{m,i+2})  + \zeta_m^{2p+2}.
\end{equation}
 Applying $\cP^p, \cP^{p^2}, \ldots $ and stopping right before everything annihilates, we conclude that $\zeta_m^{2p^i +2} \in \Ker \psi^*$, 
 once again contradicting assumption (4). 

Hence, all coefficients, except for possibly $\alpha$, are zero. This proves the claim.
\end{proof}

Since $G/G_\ev \cong A/A_\ev \cong \bbG_a^-$, the extension \eqref{eq:central} restricts to an extension on the even subgroup schemes.
\begin{equation}
\label{eq:centralev} 
\xymatrix{1 \ar[r] & \bbG_{a(1)} \ar[r] \ar@{=}[d]& G/N\ar[r]& A\ar[r]&  1
\\
1 \ar[r] & \bbG_{a(1)} \ar[r] & G_\ev/N \ar@{^(->}[u]\ar[r] & A_\ev \ar@{^(->}[u]\ar[r] & 1
}
\end{equation}   
This gives rise to a commutative diagram of the corresponding 5-term sequences: 
\[\xymatrix{H^{1,*}(A,k) \ar[r]\ar[d] & H^{1,*}(G/N,k) \ar[r] \ar[d] & H^{1,*}(\bbG_{a(1)},k) \ar@{=}[d]\ar[r]^-{d_2}& H^{2,*}(A,k) \ar[d]^{\res} \ar[r]^{\psi_1^*}& H^{2,*}(G/N,k)\ar[d]  \\ 
H^{1,*}(A_\ev,k) \ar[r] & H^{1,*}(G_\ev/N,k) \ar[r] & H^{1,*}(\bbG_{a(1)},k) \ar[r]^-{d^{\ev}_2}& H^{2,*}(A_\ev,k) \ar[r]^-{\psi_{1,\ev}^*}& H^{2,*}(G_\ev/N,k) } 
\]
By Claim~\ref{cl:u}, $d_2(\lambda) = x_{m,r} + \alpha \zeta^2$. Since $\zeta$ goes to $0$ under the restriction map $H^{*,*}(A,k) \to H^*(A_\ev,k)$, we get that $d_2^\ev(\lambda) = x_{m,r} \in H^2(A_\ev,k) = H^2(E_{m,r},k)$. 

Consider the standard surjection map: $E_{m+1, r} \times (\bbZ/p)^{\times s} \to E_{m,r} \times (\bbZ/p)^{\times s} $. By Theorem~\ref{th:HEmn^-}, $d_2^\ev(\lambda) = x_{m,r}$ vanishes when inflated to $H^2(E_{m+1, r},k)$. Proposition~\ref{pr:abelian} now implies that $G_\ev/N$ is abelian. This contradicts 
the choice of $N$, and completes the proof that $G = A$ in this case.  

It remains to consider the case $m=1$, that is, when $A \cong E_{1,r}^- \times (\bbZ/p)^{\times s} = \ov G$.  In this case $d_2(\lambda) =  \zeta^2 + \gamma x_r$, and, hence, $d_2^{\ev}(\lambda) = \gamma x_r$. Considering the surjective map $E_{2,r} \times (\bbZ/p)^{\times s} \to\mathbb G_{a(r)} \times (\bbZ/p)^{\times s}$, we conclude by Proposition~\ref{pr:abelian} that $G_\ev/N$ is abelian, getting a contradiction again. Hence, $G = A$ in the case $m=1$. 

{\it Case II:} $A \cong E_{m,r+1,\mu}^- \times (\bbZ/p)^{\times s}$. The 
proof is very similar, replacing $x_{m,r}$ with $x_{m, r+1, \mu}$ from Theorem~\ref{th:HEmnmu^-}.  
The  corresponding abelian cover which plays the role of $A^\prime$ in Proposition~\ref{pr:abelian} 
in this case is the canonical map $E_{m,r+1} \to E_{m,r+1,\mu}$.
\end{proof}


\section{The main detection theorem}
\label{se:role}
The proof of the main detection Theorem~\ref{th:detect} effectively splits into two parts. 
The first part covers the case when $G$ satisfies Hypothesis~\ref{hyp:inductive}.  
The techniques needed to deal with this case are mostly adaptations of what was done 
for finite group schemes (without the grading) and are summarised in Part I of the paper. The only, 
but significant, exception is Theorem~\ref{th:bp2} which requires 
extensive new calculations for cohomology of supergroup schemes done in \cite{Benson/Pevtsova:bp2}. 
In the ungraded case the only group schemes which fail Hypotheses~\ref{hyp:inductive} are 
the elementary ones, that is, finite groups schemes isomorphic to $\bbG_{a(r)} \times (\bbZ/p)^{\times s}$, 
which form the detection family. Hence, the inductive detection Theorem~\ref{th:detection} gives the full 
detection theorem in the ungraded case. 

In the super case however we have to deal with case (ii) of Theorem~\ref{th:B3.6} when 
the kernel of the map on cohomology induced by $f\colon G \to \ov G$ has an element of the form $\zeta^2 - \gamma x_r$. The new 
technology developed in Part II culminating in the cohomological characterization of the elementary supergroup 
schemes in Theorem~\ref{th:elemcoh} is what we need to deal with this case.

Theorem~\ref{th:detect} is an immediate consequence of the following theorem.  
We employ terminology of a {\it detection family} introduced in the beginning of Section~\ref{se:proj}.

\begin{theorem}\label{th:main}
Suppose that $G$ is a finite unipotent supergroup scheme which 
is not isomorphic to a quotient of some $E^-_{m,n}\times (\bbZ/p)^{\times s}$.
Then
\begin{enumerate}
\item[\rm (i)] nilpotence of elements in cohomology of modules and
\item[\rm (ii)] projectivity of $kG$-modules 
\end{enumerate}
are detected on proper sub-supergroup schemes after field extension.
\end{theorem}

\begin{proof}
Let $G=G^0\rtimes \pi$ with $G^0$ connected and $\pi$ finite.
Since $G$ is unipotent, so is $G^0$, and $\pi$ is a finite $p$-group.
If $\pi$ is not elementary abelian, then by Theorem~\ref{th:Serre}, $G$ satisfies
case (c) of Hypothesis~\ref{hyp:inductive}, and we are done. So we
now assume that $\pi = (\bbZ/p)^{\times s}$ is elementary abelian.
By Lemma~\ref{le:H1},
\begin{align*} 
H^{1,0}(G,k) &\cong \Hom_{\sGr/k}(G^0,\bbG_a)^\pi \times \Hom(\pi,\bbG_a) \\
H^{1,1}(G,k) &\cong \Hom_{\sGr/k}(G^0,\bbG_a^-)^\pi.
\end{align*}
We examine
the dimensions 
\begin{align*}
\delta&=\dim_k\Hom_{\sGr/k}(G^0,\bbG_{a(1)})^\pi \\
\epsilon&=\dim_k \Hom_{\sGr/k}(G^0,\bbG_a^-)^\pi.
\end{align*}
Since $\pi$ is unipotent, if $\delta=0$ then we have $\Hom_{\sGr/k}(G^0,\bbG_{a(1)})=0$, and if
$\epsilon=0$ then $\Hom_{\sGr/k}(G^0,\bbG_a^-)=0$. Thus, $\delta=\epsilon =0$, 
then $G^0$ is trivial by Lemma~\ref{le:Ga1Ga-}, hence, $G\cong (\bbZ/p)^{\times s}$, 
and we are done. We may therefore
assume that one of them is nonzero.
If either $\delta$ or $\epsilon$ is greater than one then we are in case
(a) or (b) of Hypothesis~\ref{hyp:inductive}, and we are done by Theorem~\ref{th:detection}. 
So each is either zero or one, and they are not both zero.

The action of the Frobenius map $F\colon \bbG_a\to \bbG_a$
induces a map 
\[ 
F\colon\Hom_{\sGr/k}(G^0,\bbG_a) \to \Hom_{\sGr/k}(G^0,\bbG_a) 
\]
which commutes with the action of $\pi$.  A $\pi$-invariant map $G^0\to \bbG_a$ lands in $\bbG_{a(1)}\le \bbG_a$ if and
only if it is in the kernel of $F$. So there exists $r\ge 0$ and a surjective map 
\[ 
\xi\in\Hom_{\sGr/k}(G^0,\bbG_{a(r)})^\pi 
\] 
such that $\xi,F(\xi),\dots,F^{r-1}(\xi)$ is a $k$-basis for
$\Hom_{\Gr/k}(G^0,\bbG_a)^\pi$. The map $\xi$ extends to a surjective map
\[ 
f\colon G \to \overline G\, \cong\, \bbG_{a(r)} \times (\bbG_a^-)^\epsilon \times (\bbZ/p)^{\times s} 
\]
and  $f^*\colon H^{1,*}(\ov G,k) \to H^{1,*}(G,k)$. This construction accounts both for the case 
$\delta =0$ (with $r=0$ so that $\ov G = \mathbb G_a^- \times (\mathbb Z/p)^{\times s}$) and 
$\epsilon = 0$ (with $\ov G = \mathbb G_{a(r)} \times (\mathbb Z/p)^{\times s}$, $r \geq 1$).

If $f$ is an isomorphism then $G\cong \ov G$ is isomorphic to a quotient of
$E_{1,r}^- \times (\bbZ/p)^{\times s}$ contradicting the assumption of the theorem. 
Otherwise, by Lemma~\ref{le:5-term}, 
\[ 
f^*\colon H^{2,*}(\ov G,k)\to H^{2,*}(G,k) 
\] 
is not injective. If the kernel contains an element of degree $(2,1)$, then 
by Theorem~\ref{th:bp2} we are in case (c) of
Hypothesis~\ref{hyp:inductive}, so we are done by   Theorem~\ref{th:detection}. 
Therefore, we may assume that the kernel contains an element of degree $(2,0)$ and we have two 
cases according to Theorem~\ref{th:B3.6}. In the first case, it
 contains an element of the form
\[ x_r^n\beta\cP^0(v_1)\dots\beta\cP^0(v_m), \]
which again puts us in case (c) of Hypothesis~\ref{hyp:inductive}, and we again apply 
Theorem~\ref{th:detection}. In the second case, the kernel is generated by $\zeta^2 + \gamma x_r$. 
If $\gamma =0$, then we can apply Theorem~\ref{th:detection} once again, since Hypothesis~\ref{hyp:inductive} 
is satisfied by the image of $\zeta^2$. 

The upshot of this is that we may assume that we are in case
(ii) of Theorem~\ref{th:B3.6} with $\gamma \not = 0$ and that $f^*$ induces an isomorphism on $H^{2,1}$.  
Hence, $G$ satisfies the hypotheses (1), (2) and (3) of Theorem~\ref{th:elemcoh}. 
If it fails hypothesis (4) of Theorem~\ref{th:elemcoh}, then we are in case (c) of
Hypothesis~\ref{hyp:inductive} one last time. Otherwise, 
$G$ is isomorphic to a quotient of  $E^-_{m,r}\times (\bbZ/p)^{\times s}$ for some $m \geq 2$ 
by Theorem~\ref{th:elemcoh}. 
\end{proof}

There is another notion of nilpotency for elements of
$H^{*,*}(G,\Lambda)$ where $\Lambda $ is a unital $G$-algebra. Namely,
$\xi \in H^{i,*}(G,\Lambda)$ is nilpotent if for some $n>0$, the image
of $\xi^{\otimes n} \in H^{in,*}(G,\Lambda^{\otimes n})$ in
$H^{in,*}(G,\Lambda)$ is zero. The following analogue of
Theorem~\ref{th:main} for this notion of nilpotents has both a weaker
hypothesis and a weaker conclusion.

\begin{theorem}\label{th:nilp}
Let $G$ be a finite unipotent supergroup scheme over
a field $k$, and $\Lambda$ be unital $G$-algebra. Then 
an element $x\in H^{i,*}(G,\Lambda)$ is nilpotent, that is $x^n \in H^{in,*}(G,\Lambda)$ is zero for some $n>0$, 
if and only if for every extension field $K$ of $k$ and every elementary 
sub-supergroup scheme $E$ of $G_K$, the restriction of 
$x_K\in H^{*,*}(G_K,\Lambda_K)$ to $H^{*,*}(E,\Lambda_K)$ 
is nilpotent, that is some power of $x_K$ vanishes in $H^{*,*}(E,\Lambda_K)$.
\end{theorem}

\begin{proof} First, we claim that the analogue of Theorem~\ref{th:detection} holds for $H^{*,*}(G, \Lambda)$ with this notion of nilpotency. 
Indeed, If we take $M = \Lambda$  in Proposition~\ref{pr:QuillenVenkov} then the conclusion clearly holds for 
$\xi^2 \in H^{*,*}(G,\Lambda)$. Hence, if $G$ satisfies Hypothesis~\ref{hyp:inductive}(c),  the proof of Theorem~\ref{th:detection} 
carries over to this case. 

If we assume that Hypotheses~\ref{hyp:inductive} (a) or (b) hold, then the proof is identical to that of Case II(b) in \cite[Theorem 6.1]{Bendel:2001a} 
(see also \cite[Theorem 2.5]{Bendel/Friedlander/Suslin:1997b}) so we will not reproduce it here.  

With these observations, the proof of the analogue of Theorem~\ref{th:main} 
is again identical to the one we give above.
\end{proof}

In \cite{Benson/Iyengar/Krause/Pevtsova:2017b}, we show that projectivity for modules of finite 
group schemes is detected on the family of elementary subgroup schemes after {\it coextension} 
of scalars. In the following theorem we state that this also holds for finite unipotent supergroup schemes. 

\begin{theorem}
\label{th:coext}
Let $G$ be a finite unipotent supergroup scheme over
a field $k$ of positive characteristic $p>2$, and $M$ be a $kG$-module. 
Then the following hold.
\begin{enumerate}
\item[\rm (i)] An element $x\in H^{*,*}(G,M)$ is nilpotent if and only
if for every extension field $K$ of $k$ and every elementary 
sub-supergroup scheme $E$ of $G_K$, the restriction of $x_K\in H^{*,*}(G_K,M^K)$
to $H^{*,*}(E,M^K)$ is nilpotent.
\item[\rm (ii)]  A $kG$-module
$M$ is projective if and only if for every extension field $K$ of $k$
and every elementary sub-supergroup scheme $E$ of $G_K$, the
restriction of $M^K$ to $E$ is projective.
\end{enumerate}
\end{theorem}

\begin{proof}
The proof of Theorem~\ref{th:main} carries over to this case almost without change. 
The only difference occurs when $G$ satisfies (a) or (b) of Hypothesis~\ref{hyp:inductive}.  
Then we still proceed exactly as in \cite[Theorem 8.1]{Bendel:2001a} but appeal to 
\cite[Lemma 4.1]{Benson/Iyengar/Krause/Pevtsova:2017b}
for the main ingredient which is the appropriate version  of the Kronecker quiver lemma for $M^K$.
\end{proof} 


\section{The Steenrod algebra}
\label{se:steenrod-algebra}

An \emph{affine $\bbZ$-graded group scheme} over $k$ is a covariant functor
from $\bbZ$-graded commutative $k$-algebras (again, the convention is that
$yx=(-1)^{|x||y|}xy$) to groups, whose underlying functor to sets is
representable. If $G$ is an affine $\bbZ$-graded group scheme over $k$
then its coordinate ring $k[G]$ is the representing object. It is a
$\bbZ$-graded commutative Hopf algebra. This gives a contravariant
equivalence of categories between affine $\bbZ$-graded group schemes
and $\bbZ$-graded commutative Hopf algebras.

An affine  $\bbZ$-graded group scheme $G$ has \emph{finite type} if each
graded piece of $k[G]$ is finite dimensional. In this case, the graded dual 
$kG^{i}=\Hom_k(k[G]^{-i},k)$ is a $\bbZ$-graded cocommutative Hopf
algebra of finite type. This gives a covariant equivalence of
categories between $\bbZ$-graded group schemes of finite type and
$\bbZ$-graded cocommutative Hopf algebras of finite type.

We are interested in particular in the \emph{finite} $\bbZ$-graded
group schemes; these are the ones for which not only is each graded
piece finite dimensional, but the total rank as a $k$-vector space is finite. 

Finite $\bbZ$-graded group schemes satisfy a detection theorem 
similar to Theorem~\ref{th:main}. In order to formulate it we start 
by observing that elementary supergroup schemes have natural $\bbZ$-grading. 

Recall that the group algebra of a $E^-_{m,n}$ has the following form: 
\[ kE_{m,n}^-=k[s_1,\dots,s_{n-1},s_n,\si]/(s_1^p, \dots,s_{n-1}^p,s_n^{p^m},\si^2-s_n^p).\] 
We give it a $\bbZ$-grading by assigning degrees to the generators as follows:
\begin{equation}
\label{eq:grading}
|\sigma|=ap^n, \, |s_i| = 2ap^{i-1}
\end{equation}
where $a$ is an odd integer. 
The Hopf algebra structure is compatible with this grading and, hence, $E_{m,n}^-$ becomes a $\bbZ$-graded group scheme. We call such a group scheme a {\it $\bbZ$-lifting} of $E^-_{m,n}$. We write $\widetilde E^-_{m,n}$ for such a $\bbZ$-lifting without specifying the parameter $a$. 
For a finite group $\pi$ we give its group algebra $k\pi$ a $\bbZ$-grading 
by putting it in degree $0$. 

\begin{definition} A finite $\bbZ$-graded group scheme is called 
	{\it elementary} if it is a quotient of $\widetilde E^-_{m,n} \times (\bbZ/p)^{\times s}$ where $\widetilde E^-_{m,n}$ as a $\bbZ$-lifting of  $E^-_{m,n}$.
	\end{definition} 
\begin{remark} 
	\label{rem:grading} 
	Special cases include $\bbZ$-liftings of $\bbG_a^-$ and $\bbG_{a(r)}$. Even though these liftings \emph{a priori }depend on the choice of the degree in which we put the generator of the coordinate algebra $k[\bbG_a^-] \cong k[t]/t^2$ or $k[\bbG_{a(r)}] \cong k[T]/T^{p^r}$, we use the same notation for the $\bbZ$-graded version of $\bbG_a^-$ and $\bbG_{a(1)}$ suppressing this degree.  
	\end{remark}
We define a \emph{folding} functor 
\begin{equation}
\Fl\colon \bbZ\textrm{-graded algebras} \to \bbZ/2\textrm{-graded algebras}
\end{equation}
by sending $A = \bigoplus\limits_{i\in\bbZ}A_i$ to $\Fl(A) = \bar A$ 
with $\bar A_{\odd} = \bigoplus\limits_{i \in  \bbZ} A_{2i}$ and $\bar A_{\ev} 
= \bigoplus\limits_{i \in \bbZ} A_{2i+1}$. For any $\bbZ$--graded algebra $A$ 
there is an induced functor 
\begin{equation}
\Fl\colon A\text{-}\!\mod \to \bar A\text{-}\!\mod
\end{equation}
sending a $\bbZ$-graded $A$-module $M$ to a $\bbZ/2$-graded $\bar A$-module $\bar M$: 
\[\bar M = \bar M_\ev \oplus \bar M_\odd,
\quad \bar M_\ev = \bigoplus\limits_{i \in \bbZ} M_{2i}, \quad
\bar M_\odd = \bigoplus\limits_{i \in \bbZ} M_{2i+1}.\]

Finally, if $A$ is a $\bbZ$-graded cocommutative Hopf algebra corresponding to a group scheme $G$, we denote by $\bar G$ the  supergroup scheme with the group algebra $\bar A$. 
\begin{example}\label{ex:Zelem}  Let $\widetilde E^-_{m,n}$ be a $\bbZ$-lifting of $ E^-_{m,n}$.
Then $\Fl(k\widetilde E^-_{m,n}) = k E^-_{m,n}$ for any $\bbZ$-lifting $\widetilde E^-_{m,n}$ of $E^-_{m,n}$ as in \eqref{eq:grading}.
More generally, ``folding" a $\bbZ$-graded elementary group 
scheme results in an elementary supergroup scheme. 
\end{example}

A commutative $\bbZ$-graded $k$-algebra is a \emph{$\bbZ$-graded field} if
every homogeneous element is invertible. These are field extensions
$K$ of $k$ in degree zero, and rings of Laurent polynomials of the form
$K[u,u^{-1}]$ where $u$ has non-zero even degree.
Let $k[u^{\pm}] $ be the $\bbZ$-graded field $k[u,u^{-1}]$, where $u$ has degree $2$.  
Over a $\bbZ$-graded field, every graded module is free.
This means that it is isomorphic to a direct sum of shifts of $k[u^\pm]$. 
For a $\bbZ$-graded algebra $A$, let $A[u^\pm] = A \otimes
k[u^\pm]$.  If $N$ is a module over $\bar A=\Fl(A)$, we define the structure 
of $A[u^\pm]$-module on $N[u^\pm] = N \otimes k[u^\pm]$ as follows.
For $a_i \in A_i$, and $n_{\epsilon} \in N$ homogeneous elements with 
$i \in \bbZ$, $\epsilon = 0,1$ let 
\[a_i \circ n_{\epsilon} = \bar a_i n_{\epsilon} \otimes u^{\left[\frac{i+\epsilon}{2}\right]},\]
where $\bar a_i \in \bar A$ is the element corresponding to $a_i \in A_i$. Extending $k[u^\pm]$-linearly, we get  
the desired action 
\[(A \otimes k[u^\pm]) \times (N \otimes k[u^\pm]) \to N \otimes k[u^\pm]. \] 
This defines an \emph{unfolding} functor:
\begin{align}\label{eq:functorG}
\mcG: \, \bar A\text{-}\!\mod& \longrightarrow{} A[u^\pm]\text{-}\!\mod \\
 N & \longrightarrow N[u^\pm].\notag
\end{align}

\begin{proposition}\label{prop:folding} 
Let $A$ be a finitely generated $\bbZ$-graded algebra. 
The functor $\mcG\colon \bar A\text{-}\!\mod \to
A[u^\pm]\text{-}\!\mod$ of \eqref{eq:functorG} is an 
equivalence of categories.  Moreover, it fits into a commutative diagram:
	\[
	\xymatrix{A\text{-}\!\mod \ar[d]_-{-\otimes
            k[u^\pm]}\ar[r]^-{\Fl}& \bar A\text{-}\!\mod \ar[dl]_-{\mcG}\\
	A[u^\pm]\text{-}\!\mod
} \]
and takes projective modules to projective modules. 
\end{proposition}
\begin{proof} Commutativity of the diagram amounts to checking 
that folding and then unfolding via the functor $\mcG$ is simply 
extending scalars by the graded field $k[u^\pm]$. This is a direct 
calculation. The claim about projective modules follows from the 
fact that $\mcG$ is additive and $\mcG(\bar A) = A[u^\pm]$. 

To show that $\mcG$ is an equivalence, we note that multiplication by
the invertible element $u: M_i \to M_{i+2}$ is an isomorphism, and,
hence, identifies all odd (and, respectively, all even) homogeneous
components of an $A[u^\pm]$-module $M$. Hence, sending $M$ to $M_0
\oplus M_1$ gives a functor inverse to 
$\mcG$.
%
\end{proof}

\begin{corollary}\label{cor:Zproj}
Let $A$ be a finitely generated $\bbZ$-graded algebra. 
Then a graded $A$-module $M$ is projective if and only if the 
graded $\bar A$-module $\bar M$ is projective. 
\end{corollary}
\begin{proof}
This follows from Proposition~\ref{prop:folding} and the fact 
that extending scalars to a graded field does not affect projectivity.
\end{proof}

\begin{theorem}\label{th:Zgr}
Let $G$ be a finite $\bbZ$-graded unipotent group scheme, and $M$ be a $kG$-module. Then the
following hold.
\begin{enumerate}
\item[\rm (i)] An element $\xi$ of $H^{*,*}(G,M)$ is nilpotent if and only
  if for every $\bbZ$-graded field extension $K$ of $k$, and every  elementary
  subgroup scheme $E$ of $G_K$, the restriction of $\xi_K\in
  H^{*,*}(G_K,M_K)$ to $H^{*,*}(E,M_K)$ is nilpotent.
\item[\rm (ii)] A $kG$-module $M$ is projective if and only if for
  every $\bbZ$-graded field extension $K$ of $k$, and every elementary
  subgroup scheme $E$ of $G_K$, the restriction of $M_K$ to $E$ is
  projective.
\end{enumerate}
\end{theorem}

\begin{proof}
We prove statement (ii). The argument for (i) is similar. 
Let $M$ be a $kG$-module satisfying the condition in (ii). 
By Corollary~\ref{cor:Zproj} it suffices to show that 
$\bar M$ is a projective $k\bar G$-module. 

Let $K/k$ be a (non-graded) field extension and $\overline{G}_K$ be 
the finite supergroup scheme with the group algebra ${K\bar G}$. 
Let $\bar E \subset \bar{G}_K$ be an elementary sub
supergroup scheme. Let $E$ be a $\bbZ$-graded lifting of $\bar E$. The inclusion $\bar E \subset \bar{G}$ lifts to an embedding $KE[u^\pm] \subset
KG[u^\pm]$. Indeed, to construct such a lifting we first place the generator $s_n$ of $KE$ into appropriate degree in $KG[u^\pm]$ using the parameter $u$ and then work along the relations to place $\sigma$ and $s_{n-1}, \ldots, s_1$.  
By assumption, the restriction of $M_K[u^\pm]$ to $KE[u^\pm]$ is
projective. Proposition~\ref{prop:folding} implies that the
restriction of $M_K$ to ${K\bar E}$ is projective. Since
this holds for any elementary sub supergroup scheme $\bar E$, 
Theorem~\ref{th:main} implies that $\bar{M}$ is 
projective as ${k\bar G}$-module. Hence, $M$ is projective. 
%
\end{proof}

Let $\mcA$ denote the Steenrod algebra over $\bbF_p$. Recall from Milnor~\cite{Milnor:1958a}, Steenrod and Epstein~\cite{Steenrod/Epstein:1962a}
that for $p$ odd, the graded dual $\mcA^*$ of $\mcA$ is a tensor product
\[ k[\xi_1,\xi_2,\dots] \otimes \Lambda(\tau_0,\tau_1,\dots) \]
of a polynomial ring in generators $\xi_n$ of degree $2p^n-2$ and
an exterior algebra in generators $\tau_n$ of degree $2p^n-1$.
We also set $\xi_0=1$. With this notation, the comultiplication
is given by
\[ \Delta(\xi_n)=\sum_{i=0}^n \xi_{n-i}^{p^i} \otimes \xi_i,\qquad
\Delta(\tau_n)=\tau_n\otimes 1 + \sum_{i=0}^n
\xi_{n-i}^{p^i}\otimes\tau_i. \]

If $A$ is a finite dimensional Hopf subalgebra of $\mcA$ then
the graded dual
$A^*$ is a finite dimensional quotient of $\mcA^*$. Let $\bar G$ be the
finite  supergroup scheme corresponding to the folding of $A$,
so that $\bbF_p\bar G\cong \bar A$ and $\bbF_p[\bar G]\cong \bar A^*$. We
use the same letters $\xi$, $\tau$ to denote the generators in the folded $\bbZ/2$-
graded algebra $\bar A$.
Then $\bbF_p[\bar G_{(1)}]$ is a quotient of $\mcA^*$ by a Hopf ideal
containing $\xi_1^p,\xi_2^p,\dots$.  Letting
$\bar\xi_n$ and $\bar\tau_n$ be the images of $\xi_n$ and
$\tau_n$ in this quotient, for $n\ge 1$ we have
\[ \Delta(\bar\xi_n)=\bar\xi_n\otimes 1 + 1\otimes\bar\xi_n,\qquad
\Delta(\bar\tau_n)=\bar\tau_n\otimes 1 + 1
\otimes\bar\tau_n+\bar\xi_n\otimes\bar\tau_0 \]
while $\Delta(\bar\tau_0)=\bar\tau_0\otimes 1 + 1\otimes \bar\tau_0$.
In other words, $\bar\xi_n$ ($n\ge 1$) and $\bar\tau_0$ are primitive,
and $\bar\tau_n$ ($n\ge 1$) are primitive modulo $\bar\tau_0$. 
Furthermore, $\xi_n$ is even whereas $\tau_0, \tau_n$ are odd.

If we isolate a single $n$, and dualise these relations for
$\bar\xi_n$, $\bar\tau_n$ and $\bar\tau_0$ we get the restricted
universal enveloping algebra of a three dimensional restricted Lie
superalgebra consisting of the matrices
\[ \left(\begin{array}{cc|c} 0&*&* \\ 0&0&* \\ \hline 0&0&0
\end{array}\right) \]
in $\GL(2|1)$. The dual elements $\bar\xi^*_n$ and $\bar\tau^*_n$ to
$\bar\xi_n$ and $\bar\tau_n$
are in the top row, and the dual element $\bar\tau^*_0$ to $\bar\tau_0$ is in
the second row. The only non-trivial commutator relation is
$[\bar\xi^*_n,\bar\tau^*_0] = \bar\tau^*_n$.

Dualising, 
we get a homomorphism  $\bar G_{(1)}\to
\bbG_a^-$, and the kernel is isomorphic to a subgroup scheme of
$(\bbG_{a(1)})^{\times s} \times (\bbG_a^-)^r$. Every subgroup scheme
again has this form, so we have proved the following lemma.

\begin{lemma}\label{le:St(1)}
Let $A$ be a finite dimensional Hopf subalgebra of the
Steenrod algebra, and let $\bar G$ be the supergroup scheme corresponding to the $\bbZ/2$-graded folding $\bar A$. Then there is a (possibly trivial)
homomorphism $\bar G_{(1)}\to \bbG_a^-$ whose kernel is isomorphic
to $(\bbG_{a(1)})^{\times s} \times (\bbG_a^-)^r$ for some $r,s\ge 0$.
The subgroup $(\bbG_a^-)^r$ is normal, and the quotient is commutative.
In particular, there is no sub supergroup scheme
isomorphic to $W_{m,1}^-$ for $m\ge 1$.
\qed
\end{lemma}

Conceptually, what we have done amounts to showing that
the first Frobenius kernel of the Steenrod algebra is an
extension of $\bbG_a^-$ by an
infinite product of copies of $\bbG_{a(1)}\times \bbG_a^-$,
with gradings tending to infinity, in such a way that over
each factor the extension is the one described by a $\bbZ$-lifting of 
the above subgroup of $GL(2|1)$.

\begin{proposition} \label{pr:SteenrodElem}
Let $A$ be a finite dimensional sub Hopf algebra of the Steenrod
algebra $\mcA$ over $\bbF_p$, and let $G$ be the
corresponding finite unipotent connected $\bbZ$-graded group scheme.
If $E$ is an elementary $\bbZ$-graded subgroup scheme of $G$
then $E\cong \bbG_{a(n)} \times \bbG_a^-$.
\end{proposition}
\begin{proof}
By Theorem~\ref{th:Witt-classification}, we have that
$\bar E \cong E_{m,n}^-$ or $\bar E \cong E_{m,n, \mu}^-$.
We need to show that $m=1$. But if $m \ge 2$,
the statement follows from the observation that
$(E_{m,n}^-)_{(1)}$ and $(E_{m,n, \mu}^-)_{(1)}$ both contain
$W_{m-1,1}^-$ as a subgroup scheme. But
$(\bbG_{a(1)})^{\times s} \times (\bbG_{a}^-)^{\times r}$ does not,
and therefore by Lemma~\ref{le:St(1)} neither does $\bar G_{(1)}$.
\end{proof}

The detection theorem for the finite dimensional subalgebras of the 
Steenrod algebra now follows from Theorem~\ref{th:Zgr}
and Proposition~\ref{pr:SteenrodElem}. Recall from Remark~\ref{rem:grading} that we use the notation $\bbG_{a(r)}$, $\bbG_{a}^-$ for $\bbZ$-liftings of the corresponding supergroup schemes. 

\begin{theorem}
\label{th:Steenrod}
Let $A$ be a finite dimensional sub Hopf algebra of the Steenrod
algebra $\mcA$ over $\bbF_p$. Then $A$ is the group algebra of a
$\bbZ$-graded finite  group scheme. The following hold:
\begin{enumerate}[\quad\rm(1)]
\item For an $A$-module $M$, an element $\xi$ in $H^{*,*}(A,M)$
is nilpotent if and only if for every $\bbZ$-graded field extension
$K$ of $k$,  the restriction of $\xi_K\in H^{*,*}(A_K,M_K)$ to
every subgroup scheme of $A_K$ isomorphic 
to $\bbG_{a(r)}$, $\bbG_a^-$, or $\bbG_{a(r)} \times \bbG_a^{-}$
is nilpotent.
\item An $A$-module $M$ is projective if and only if for every
$\bbZ$-graded field extension $K$ of $k$, the restriction of 
$M_K$ to every subgroup
scheme of $A_K$ isomorphic to $\bbG_{a(r)}$, $\bbG_a^-$, or
$\bbG_{a(r)} \times \bbG_a^-$ is projective. \qed
\end{enumerate}
\end{theorem} 

Nakano and Palmieri \cite{Nakano/Palmieri:1998a} also considered the problem of
finding a detecting family for the mod $p$ Steenrod algebra. They do not consider
field extensions, and arrive at a larger family of detecting subalgebras, which they
call ``quasi-elementary''.



\begin{appendix}
\section{Witt vectors and Dieudonn\'e modules}
\label{sec:Witt}

Recall that finite commutative connected unipotent group schemes form an abelian category
$\fA$ which is equivalent to an appropriate category of Dieudonn\'e modules. This is described for
example in Fontaine \cite{Fontaine:1977a}, but we give an outline here. What will interest
us is the Dieudonn\'e modules killed by $p$, which were classified by Koch \cite{Koch:2001a}.

We begin with a brief recollection concerning the Witt vectors. 
Define a polynomial $w_n$ in
variables $Z_0,\dots,Z_n$ with integer coefficients by
\[ w_n(Z_0,\dots,Z_n) = p^nZ_n + p^{n-1}Z_{n-1}^p + \dots +
Z_0^{p^n}. \]
Then the polynomials $S_i$ and $P_i$ in variables
$X_0,\dots,X_n,Y_0,\dots,Y_n$, 
again with integer coefficients, are defined by
\begin{align*} 
w_n(S_0,\dots,S_n) &= w_n(X_0,\dots,X_n)+w_n(Y_0,\dots,Y_n), \\
w_n(P_0,\dots,P_n) &= w_n(X_0,\dots,X_n)w_n(Y_0,\dots,Y_n).
\end{align*}
So for example $S_0=X_0+Y_0$, $P_0=X_0Y_0$, 
\[ S_1 = X_1 + Y_1 + \frac{(X_0+Y_0)^p - X_0^p-Y_0^p}{p}, \qquad
P_1 = pX_1Y_1 + X_0^pY_1 + X_1Y_0^p, \]
and so on. 

Witt vectors $W(k)$ over $k$ are vectors $(a_0,a_1,\dots)$
with $a_i\in k$, where $S_i$ and $P_i$ give the coordinates of
the sum and product:
\begin{align*}
(a_0,a_1,\dots)+(b_0,b_1,\dots) &= (S_0(a_0,b_0),S_1(a_0,a_1,b_0,b_1), \dots) \\
(a_0,a_1,\dots)(b_0,b_1,\dots) &= (P_0(a_0,b_0),P_1(a_0,a_1,b_0,b_1), \dots).
\end{align*}
Thus for example if $k=\bbF_p$
then $W(k)$ is the ring of $p$-adic integers $\bbZ_p$. More generally, $W(k)$ is
a local ring of mixed characteristic $p$. 
The Frobenius 
endomorphism of $k$ lifts to a ring endomorphism of $W(k)$ denoted $\bsi$.
It is defined by $(a_0,a_1\dots)^\bsi=(a_0^p,a_1^p,\dots)$.

More generally, if $A$ is a commutative $k$-algebra then $W(A)$ is the ring of
Witt vectors over $A$, defined using the same formulae. 
This defines a functor from commutative $k$-algebras to rings.
The additive part of this functor defines an affine group scheme over
$k$ denoted $W$, the \emph{additive Witt vectors}. If
we stop at length $m$ vectors, we obtain $W_m$, and we write $W_{m,n}$
for the $n$th Frobenius kernel of $W_m$.

There are two endomorphisms $V$ and $F$ of $W$ of interest to us. These are 
the Verschiebung $V$ defined by
\[ V(a_0,a_1,\dots)=(0,a_0,a_1,\dots), \]
and the Frobenius $F$ given by
\[ F(a_0,a_1,\dots)=(a_0^p,a_1^p,\dots). \]
These commute, and their product corresponds to multiplication by $p$ on Witt vectors.
Multiplication by a Witt vector $x\in W(k)$ also gives an endomorphism of $W$ which we
shall denote $x$ by abuse of notation. These are related to $V$ and $F$ by the relations
$Vx^\bsi=xV$ and $Fx=x^\bsi F$. 

We write $W_{m}$ for the group scheme of Witt vectors of length $m$, 
corresponding to the quotient $W(k)/(p^m)$ of $W(k)$. This is a group scheme
with a filtration whose quotients are $m$ copies of the additive group $\bbG_a$.
We write $W_{m,n}$ for the $n$th Frobenius kernel of $W_m$. This is a finite
group scheme with a filtration of length $mn$ whose quotients are copies of
$\bbG_{a(1)}$. 

The \emph{Dieudonn\'e ring} $D_k$ is generated 
over $W(k)$ by two commuting variables $V$ and $F$ satisfying the following 
relations:
\[ FV=VF=p,\qquad Vx^\bsi = xV, \qquad Fx=x^\bsi F \]
for $x\in W(k)$. Then $W$ is a module over $D_k$, as are its
quotients $W_m$ and their finite subgroup schemes $W_{m,n}$.

Recall that there is a duality on $\fA$ called Cartier duality, which
corresponds to taking the $k$-linear dual of the corresponding Hopf algebras.
We denote the Cartier dual of $G$ by $G^\sharp$.

Now consider the subcategory $\fA_{m,n}$ of $\fA$ consisting of the
those group schemes $G$ in $\fA$ such that $G$ has height at most
$n$ and the Cartier dual $G^\sharp$ has height at most $m$. Then
there is a covariant equivalence of categories between $\fA_{m,n}$ 
and the category $\mod(D_k/(V^m,F^n))$ of finite length modules over
the quotient ring $D_k/(V^m,F^n)$. This equivalence is given by
the functor
\[ \Hom_\fA(W_{m,n},-)\colon\fA_{m,n}\to\mod(D_k/(V^m,F^n)). \]

Write $\hat D_k$ for the corresponding completion $\displaystyle\lim_\leftarrow D_k/(V^m,F^n)$
Then every $\hat D_k$-module of finite length is a module for some quotient
of the form $D_k/(V^m,F^n)$, and these equivalences combine to
give an equivalence between $\fA$ and the category
$\fl(\hat D_k)$ of $\hat D_k$-modules
of finite length. Let us write 
\[\psi\colon \fl(\hat D_k)\to \fA\] 
for this equivalence. Thus for example $\psi(D_k/(V^m,F^n))\cong W_{m,n}$
and $\psi(D_k/(V^m,F^n,p)) \cong W_{m,n}/W_{m-1,n-1} \cong E_{m,n}$, where 
the last notation is introduced in Defintion~\ref{defn:Witt}.

Let $G = \psi(M)$ be a finite unipotent abelian group scheme, so that
$M$ is a finite length $D_k/(V^m,F^n)$-module for some $m,n\ge 1$. 
If we are only interested in the algebras structure of $G$, 
this means that we can ignore the action of $F$ on $M$ and just look 
at finite length modules for $D_k/(V^m,F)=W(k)[V]$ with $xV=Vx^\bsi$ 
($x\in W(k)$). Such modules are always direct sums
of cyclic submodules, and the cyclic modules are just truncations at
smaller powers of $V$. Translating through the equivalence $\psi$, we
have the following.

\begin{lemma}
\label{le:algebra} 
Let $G$ be a finite unipotent abelian group scheme. Then $kG$ is a isomorphic 
to a tensor product of algebras of the form $kW_{m,1} \simeq k[s]/s^{p^m}$. 
\end{lemma}

\begin{lemma}\label{ab:quotients}
Let $G$ be a finite unipotent abelian group scheme. If 
$\dim_k\Hom_{\Gr/k}(G,\bbG_{a(1)})=1$ and $G$ does not have $W_{2,2}$ as a quotient, 
then $G$ is isomorphic to a quotient of the group scheme $E_{m,n}$. 
\end{lemma}

\begin{proof} 
The condition $\dim_k\Hom_{\Gr/k}(G,\bbG_{a(1)})=1$ implies that the corresponding Dieudonn\'e 
module is cyclic,  $G_\ev\cong\psi(D_k/I)$ for some ideal $I$ containing $V^m$ and $F^n$ for some 
$m$, $n$. Not having $W_{2,2}$ as a quotient implies that $p =FV$ kills $D_k/I$, and, hence, $G$ 
is isomorphic to a quotient $D_k/(V^m, F^n, p)$. But the latter is precisely $E_{m,n}$. 
\end{proof} 

The last thing we need is the classification of the quotients of the group scheme 
$E_{m,n}$.  In terms of Dieudonn\'e modules, we have
\[ E_{m,n} = \psi(D_k/(V^m,F^n,p)). \]
The isomorphism classes of quotients of $D_k/(V^m,F^n,p)$ were classified
by Koch \cite{Koch:2001a}. The main results of that paper may be stated as
follows.

\begin{theorem}
\label{th:Koch}
Every nonzero finite quotient of $\hat D_k/(p)$ as a left 
$\hat D_k$-module is isomorphic to either $M_{m,n}=D_k/(V^m,F^n,p)$
(of length $m+n-1$) or $M_{m,n,\mu}=D_k/(F^n-\mu V^m,p)$ 
(of length $m+n$) for some $m,n \ge 1$
and $0\ne \mu\in k$. The only isomorphisms among these modules are given
by $M_{m,n,\mu}\cong M_{m,n,\mu'}$ if and only if $\mu/\mu'=a^{p^{m+n}-1}$ for some $a\in k$.
\end{theorem}

\begin{proof}[Outline of proof]
Let $M$ be a nonzero finite quotient of $\hat D_k/(p)$, let $m$ be the height of $M^\sharp$ and
$n$ be the height of $M$. Then $M$ is a finite quotient of $D_k/(V^m,F^n,p)$. So either
$M$ is isomorphic to $D_k/(V^m,F^n,p)$ or the kernel is at least one dimensional. If the
kernel has length one, then it is in the socle, which has length two, and is the image of $V^{m-1}$ and $F^{n-1}$.
By minimality of $m$ and $n$, the kernel is then $(F^{n-1}-\mu V^{m-1})$ for some $0\ne \mu \in k$.
If $M$ is equal to this, we have $M \cong M_{m-1,n-1,\mu}$. Otherwise $M$ is a proper quotient of $M_{m-1,n-1,\mu}$.
But the socle of $M_{m-1,n-1,\mu}$ is one dimensional, spanned by the image of $V^{m-1}$, so in this
case $M$ is a quotient of $M_{m-1,n-1}$, which implies that $m$ and $n$ are not minimal. This contradiction
proves that these are the only isomorphism types. 

The dimensions of $M/F^iM$ and $M/V^iM$ distinguish all isomorphism classes, with the
possible exception of isomorphisms between $M_{m,n,\mu}$ and $M_{m,n,\mu'}$. Such an isomorphism is
determined modulo radical endomorphisms by a scalar $a\in k^\times \subseteq W(k)^\times$. The equation 
$(F^n-\mu V^m)a=b(F^n-\mu' V^m)$ implies that $b=a^{\bsi^n}$ and $\mu a=b^{\bsi^m}\mu'$. Thus
\begin{equation*}
\mu/\mu'=a^{\bsi^{m+n}}a^{-1}=a^{p^{m+n}-1}. 
\qedhere
\end{equation*}
\end{proof}

\begin{remark}
Note that if $k=\bbF_p$ then this condition on $\mu$ and $\mu'$ is only
satisfied if $\mu=\mu'$, so there are $p-1$ isomorphism classes of $M_{m,n,\mu}$. 
But if $k$ is algebraically closed then the isomorphism
type of $M_{m,n,\mu}$ is independent of $\mu$.
\end{remark}
 
 \end{appendix}

\bibliographystyle{amsplain}

\end{document}